\documentclass[11pt]{amsart}
\usepackage{amsmath,amscd,amssymb,amsfonts,amsthm}
\usepackage[colorlinks=true, pdfstartview=FitV, linkcolor=blue, citecolor=blue, urlcolor=blue, breaklinks=true]{hyperref}
\usepackage{wasysym}
\usepackage[cmtip,matrix,arrow]{xy}

 \usepackage{dutchcal}

\newtheorem{theorem}{Theorem}[section]
\newtheorem{corollary}[theorem]{Corollary}
\newtheorem{proposition}[theorem]{Proposition}

\newtheorem{lemma}[theorem]{Lemma}
\theoremstyle{definition}
\newtheorem{definition}[theorem]{Definition}
\theoremstyle{remark}

\newtheorem{remark}[theorem]{Remark}

\newtheorem{example}[theorem]{Example}
\newtheorem{examples}[theorem]{Examples}
\newcommand\A{\mathcal{A}}

\newcommand{\Ger}[1]      {[\![#1]\!]}

\newcommand{\Lie}[1]      {[\!\![#1]\!\!]}

\newcommand\M{\mathcal{M}}

\newcommand{\W}{\mathcal{W}}

\renewcommand{\L}{\mathcal{L}}

\newcommand{\ca}{\mathcal}

\newcommand{\R}{\mathbb{R}}

\newcommand{\Z}{\mathbb{Z}}

\renewcommand{\S}{\mathcal{S}}

\newcommand\lie[1]{\mathfrak{#1}}

\newcommand{\g}{\lie{g}}

\renewcommand{\a}{\mathsf{a}}

\newcommand{\on}{\operatorname}
\newcommand{\Aut}{ \on{Aut} }

\newcommand{\Hom}{ \on{Hom}}

\renewcommand{\subset}{\subseteq}

\renewcommand{\ker}{ \on{ker}}

\newcommand\qu{/\kern-.7ex/} 
\newcommand{\lra}{\longrightarrow}
\newcommand{\hra}{\hookrightarrow}
\newcommand{\ra}{\rightarrow}

\renewcommand{\d}{{\mathrm{d}}}

\newcommand{\f}{\frac}

\newcommand{\p}{\partial}
\renewcommand{\l}{\langle}
\renewcommand{\r}{\rangle}

\newcommand{\eeq}{\end{eqnarray*}}
\newcommand{\beq}{\begin{eqnarray*}}

\newcommand{\wh}{\widehat}
\newcommand{\wt}{\widetilde}
\newcommand{\mf}{\mathfrak}

\newcommand{\rra}{\rightrightarrows}

\newcommand{\da}{\dasharrow}
\newcommand{\GL}{\on{GL}}

\newcommand{\LA}{{\mathcal{LA}}}
\newcommand{\VB}{{\mathcal{VB}}}
\newcommand{\DVB}{{\mathcal{DVB}}}

\newcommand{\Ra}{\Rightarrow}
\newcommand{\IM}{{\on{IM}}}
\newcommand{\lin}{{\on{lin}}}
\newcommand{\core}{{\on{core}}}
\newcommand{\CC}{{E}}
\newcommand{\cc}{{e}}
\newcommand{\ggamma}{{\varepsilon}}
\newcommand{\duer}{\mathbin{\raisebox{3pt}{\varhexstar}\kern-3.85pt{\rule{0.24pt}{5pt}}}\,}

\begin{document}
\sloppy
\title{The Weil algebra of a double Lie algebroid}

\author{Eckhard Meinrenken}
\address{University of Toronto, Department of Mathematics,
40 St George Street, Toronto, Ontario M4S2E4, Canada }
\email{mein@math.toronto.edu}

\author{Jeffrey Pike}
\address{University of Toronto, Department of Mathematics,
40 St George Street, Toronto, Ontario M4S2E4, Canada }
\email{jpike@math.toronto.edu}

\begin{abstract}
Given a double vector bundle $D\to M$, we define a bigraded bundle of algebras
$W(D)\to M$ called the `Weil algebra bundle'. The space $\W(D)$ of sections of this algebra bundle `realizes' the algebra of functions on the supermanifold $D[1,1]$. We describe in detail the relations between the Weil algebra bundles of $D$ and those of the double vector bundles $D',\ D''$ obtained from $D$ by duality operations. We show that $\VB$-algebroid  structures on $D$ are equivalent to horizontal or vertical differentials on two of the Weil algebras and a Gerstenhaber bracket on the third. Furthermore, Mackenzie's definition of a double Lie algebroid is equivalent to compatibilities between two such structures on any one of the three Weil algebras. In particular, we obtain a `classical' version of Voronov's result characterizing double Lie algebroid structures.  In the case 
that $D=TA$ is the tangent prolongation of a Lie algebroid, we find that $\W(D)$ is the  Weil algebra of the Lie algebroid, as defined by Mehta and Abad-Crainic. We show that the deformation complex of Lie algebroids, the theory of IM forms and IM multivector fields, and 2-term representations up to homotopy, all have natural interpretations in terms of our Weil algebras. 
\end{abstract}
\maketitle
\tableofcontents
\section{Introduction}\label{sec:intro}
A well-known result of Vaintrob \cite{vai:lie} states that the Lie algebroid structures on a given vector bundle $A\to M$ are equivalent to homological vector fields $Q$ on the graded supermanifold $A[1]$. In classical language, the smooth functions on $A[1]$ are the sections of the exterior bundle $\wedge A^*$, and Vaintrob's observation says that Lie algebroids $A$ are completely determined by the associated \emph{Chevalley-Eilenberg complex} 
\[ (\Gamma(\wedge A^*),\d_{CE}).\]
If $A$ is obtained by applying the Lie functor to a Lie groupoid $G\rra M$, then the Chevalley-Eilenberg complex is the infinitesimal counterpart to the complex of groupoid cochains on $G$; the \emph{van Est map} of Weinstein-Xu \cite{wei:ext} gives a cochain map from the groupoid complex to the Lie algebroid complex. 

We are interested in generalizations of this theory to \emph{double Lie algebroids}, as introduced by Mackenzie \cite{mac:dou2,mac:not,mac:ehr}. A double Lie algebroid is a double vector bundle 
\[ 
\xymatrix{ {D} \ar[r] \ar[d] & B \ar[d]\\
	A \ar[r] & M}
\]
for which all the sides are equipped with Lie algebroid structures, and with a certain compatibility 
condition between the  horizontal and vertical Lie algebroid structures.  As a first example, the tangent bundle of any Lie algebroid is a double Lie algebroid \cite[Example 4.6]{mac:not}. Double Lie algebroids are of  importance in second order differential geometry \cite{mac:dou,mac:dou1}. They arise as the infinitesimal counterparts of double Lie groupoids, and as such appear in Poisson geometry and related areas of mathematics.  Voronov \cite{vor:q}  proved that  a double Lie algebroid structure on $D$ is equivalent to two commuting homological  vector fields $Q_h,Q_v$, of bidegrees $(1,0)$ and $(0,1)$, on the bigraded supermanifold $D[1,1]$ obtained from $D$ by a parity shift in both vector bundle directions. Put differently, the algebra of functions 
on $D[1,1]$ is a double complex, generalizing the Chevalley-Eilenberg complex of a Lie algebroid. One of the aims of this paper to give a
coordinate-free, `classical' description of this double complex, avoiding the use of super-geometry.  If the double Lie algebroid is obtained by applying the Lie functor to a double Lie groupoid or $\LA$-groupoid, then this double complex will serve as the codomain of the corresponding van Est map. This application to van Est theory will be developed in a forthcoming paper; for the tangent prolongation of a Lie groupoid, one recovers the van Est map of Abad-Crainic \cite{aba:wei}.

To explain our construction, let $D$ be any double vector bundle with side bundles $A,B$. 
By results of Grabowski-Rotkiewicz \cite{gra:hig}, the compatibility of the two vector bundle structures reduces to the condition that the horizontal  and vertical scalar multiplications commute. The submanifold on which the 
two scalar multiplications coincide is itself a vector bundle over $M$, called the \emph{core} of $D$. 
We denote by 
\[ \CC=\on{core}(D)^*\] 
its dual bundle. 
There is a vector bundle $\wh{\CC}\to M$ whose space of sections consists of the smooth functions on $D$ that are \emph{double-linear}, i.e., linear both horizontally and vertically. It fits into an exact sequence 
\[ 0\to A^*\otimes B^* \stackrel{i_{\wh{\CC}}}{\lra} \wh{\CC}\to \CC\to 0,\]
where the map $\wh{\CC}\to \CC$ is given by 
the restriction of double-linear functions to the core, while
the map $i_{\wh{\CC}}$ is given by the multiplication of linear functions on $A,B$. 
Our definition of the Weil algebra bundle is as follows:
 \medskip

{\bf Definition.}
The \emph{Weil algebra bundle} of the double vector bundle $D$ is the bundle (over $M$) of  bigraded super-commutative algebras
\[ W(D)=(\wedge A^*\otimes \wedge B^*\otimes \vee \wh{\CC})/\sim\ ,\]
taking the quotient by the (fiberwise) ideal generated by elements of the form 
\[
\alpha \beta- i_{\wh{\CC}}(\alpha\otimes \beta)\]
for $(\alpha,\beta)\in A^*\times_M B^*$.  Here, generators $\alpha\in A^*$ have bidegree $(1,0)$, generators $\beta\in B^*$ have bidegree $(0,1)$, and generators $\wh{\cc}\in \wh{\CC}$ have bidegree $(1,1)$.  The bigraded super-commutative algebra $\W(D)=\Gamma(W(D))$ will be called the \emph{Weil algebra} of $D$. 

\bigskip

Double vector bundles come in triples, with cyclic permutation of the roles of the vector bundles $A,B,\CC$ over $M$: 
\[ 
\xymatrix{ {D} \ar[r] \ar[d] & B \ar[d]\\
	A \ar[r] & M}\ \ \ \ \ \ \ \ 
	\xymatrix{ {D'} \ar[r] \ar[d] & \CC \ar[d]\\
		B \ar[r] & M}\ \ \ \ \ \ \ 
\xymatrix{ {D''} \ar[r] \ar[d] & A \ar[d]\\
		\CC \ar[r] & M}
\]
The double vector bundle $D'$ is (essentially) obtained by taking the dual of $D$ as a vector bundle over $B$ and  interchanging the roles of horizontal and vertical structures;  similarly $D''\cong (D')'$. (One has $(D'')'=D$.) See Section \ref{sec:dvb} for a more precise description. 
Accordingly, we have three Weil algebras
\[ \W(D),\ \  \W(D'),\ \ \W(D'').\] 

A linear Lie algebroid structure of $D$ as a vector bundle over $A$, also known as a \emph{$\VB$-algebroid} structure of $D$ over $A$,  is equivalent to a double-linear Poisson structure on  $D''$, and also to a $\VB$-algebroid structure of $D'$ over $\CC$.
(See \cite{mac:dou}.) Section \ref{sec:PDVB} explains in detail how these structures are expressed in terms of the Weil algebras. In particular, one finds:
\bigskip

\noindent{\bf Theorem I.} Let $D$ be a double vector bundle. Then the following are equivalent: 
\begin{enumerate}
\item a $\VB$-algebroid structure of $D$ over $A$,
\item a vertical differential $\d_v$ on $\W(D)$,
\item a horizontal differential $\d_h'$ on $\W(D')$, 
\item a Gerstenhaber bracket (of bidegree $(-1,-1)$) on $\W(D'')$. 
\end{enumerate}
\bigskip
There is a wealth of identities relating the differentials and brackets on the Weil algebras, through duality pairings and contraction operators. Aspects of this `Cartan calculus'  are explored in Section \ref{sec:relations}. \medskip

Using cyclic permutations of $D,D',D''$, one has similar results when starting out with a $\VB$-algebroid structure of $D$ over $B$ or with a double-linear Poisson structure on $D$. Section \ref{sec:DLA} deals with 
the situation that $D$ has any \emph{two} of these structures; in particular we prove:
\bigskip

\noindent{\bf Theorem II.} Let $D$ be a double vector bundle, with $\VB$-algebroid structures over 
$B$ as well as over $A$. Then the following are equivalent: 
\begin{enumerate}
\item $D$ is a double Lie algebroid, 
\item the horizontal and vertical differentials $\d_h,\d_v$ on $\W(D)$ commute,
\item the horizontal differential $\d_h'$ on $\W(D')$ is a derivation of the Gerstenhaber bracket, 
\item the vertical differential $\d_v''$ on  $\W(D'')$ is a derivation of the Gerstenhaber bracket. 
\end{enumerate}
\bigskip

If one uses the identification of $\W(D)$ with functions on the supermanifold $D[1,1]$, 
the equivalence (a) $\Leftrightarrow$ (b) translates into Voronov's result \cite{vor:q} mentioned above; however, we will give a direct proof of this result, not using any super-geometry. More precisely, given vertical and horizontal $\VB$-algebroid structures, the proof will give an explicit relationship between their compatibility (or lack thereof) and the super-commutator of the two differentials. 
\bigskip

If $D=T\g$ is the tangent bundle of a Lie algebra, 
viewed as a double Lie algebroid with $A=\g,\ B=0,\ \CC=\g^*$, then the three Weil algebras are 
\[ \W(T\g)=\wedge\g^*\otimes \vee\g^*,\ \ \W((T\g)')=\wedge \g\otimes \vee \g,\ \ 
\W((T\g)'')=\wedge \g^*\otimes  \wedge \g.\]
Here $\W(T\g)$ is the standard Weil algebra \cite{we:oe1}, with $\d_h$ the Chevalley-Eilenberg differential for the $\g$-module $\vee\g^*$ and $\d_v$ the Koszul differential. The differential $\d_h'$ on $\W((T\g)')$ is the Koszul differential, and the Gerstenhaber bracket is a natural extension of the Lie bracket of the semi-direct product $\g\ltimes \g$ for the adjoint action. The differential $\d_v''$ on $\W((T\g)'')$ is the Chevalley-Eilenberg differential for the $\g$-module $\wedge\g$, and the Gerstenhaber bracket extends the pairing  between $\g^*$ and $\g$ (this is a special case of Kosmann-Schwarzbach's \emph{big bracket} \cite{kos:jac,kos:poi1}).
More generally, if $D=TA$ is the tangent bundle of a Lie algebroid $A$, the double complex 
$\W(TA)$ coincides with the \emph{Weil algebra of the Lie algebroid $A$}, as defined by Mehta \cite{meh:sup} using  super-geometry, and by Abad-Crainic \cite{aba:wei} in a classical framework. 
\medskip

Returning to a general double vector bundle, let $\wedge_AD$ and $\wedge_B D$ be the exterior bundles of $D$ viewed as vector bundles over $A$ and $B$, respectively. By considering the homogeneity of sections in the $A$-direction, one can define  distinguished subspaces of  of \emph{linear sections} of these bundles.  In Section \ref{sec:wedge}, we show that these have descriptions in terms of the Weil algebras:
\[  \Gamma_\lin(\wedge^\bullet_A D,A)=\W^{\bullet,1}(D''),\ \ \ 
\Gamma_\lin(\wedge^\bullet_B D,B)=\W^{1,\bullet}(D').\]

%
As a consequence of Theorem I above, a double-linear Poisson structure on $D$ determines a degree $1$ differential on these spaces, while a $\VB$-algebroid structure on $D$ over $A$ 
(respectively over $B$) determines a Gerstenhaber bracket. 
For the cotangent and tangent bundles of a vector bundle $V\to M$ (with the $\DVB$ structures as in Section \ref{subsec:examples}), some of these spaces 
have well-known interpretations:
\[ \W^{1,\bullet}(T^*V)=\mf{X}_\lin^\bullet(V),\  \ \W^{1,\bullet}(TV)=\Omega^\bullet_\lin(V). 
\]
Here $\mf{X}_\lin(V)$ are linear multi-vector fields with  the Schouten bracket, while 
$\Omega_\lin(V)$ are linear differential forms with the de Rham 
differential.
If $V$ is a Lie algebroid over $M$, one also has  horizontal differentials on $\W(T^*V)$ and on $\W(TV)$, 
coming from the $\VB$-algebroid structures of $T^*V$ over $V^*$ and $TV$ over $TM$, respectively. In Section \ref{sec:Applications}, we will see that
\[ \W^{1,\bullet}(T^*V)\cap \ker(\d_h)=\mf{X}_{\on{IM}}^\bullet(V),\ \ \ \W^{1,\bullet}(TV)\cap \ker(\d_h)=\Omega_{\on{IM}}^\bullet(V),\]
the space of \emph{infinitesimally multiplicative} multi-vector fields  \cite{igl:uni}
and  infinitesimally multiplicative differential forms \cite{bur:mul}, respectively. On the other hand, the Lie algebroid structure of $V$ over $M$ also induces a $\VB$-algebroid structure on 
$T^*V^*$ over $V^*$, and the corresponding 
 differential $\d_v$ on $\W^{1,\bullet}(T^*V^*)=\mf{X}_\lin(V^*)$ is the Poisson differential for the resulting Poisson structure on $V^*$,  identifying this space with the  \emph{deformation  complex} of Crainic-Moerdijk \cite{cra:def}. (This may also be seen as a consequence of a result of Cabrera-Drummond \cite{cab:hom} for the $\VB$-algebroid $T^*V^*$.)  Further applications relate the Weil algebra to the Fr\"olicher-Nijenhuis \cite{fro:th} and Nijenhuis-Richardson \cite{nij:def} brackets, to \emph{matched pairs of Lie algebroids} 
\cite{mok:mat}, and to the notion of \emph{representations up to homotopy} \cite{aba:rep,gra:dla,gra:vba}. 

Let us finally remark that there is no principal difficulty in generalizing the constructions of this paper to $n$-fold vector bundles, however, the theory requires careful 
bookkeeping due to the presence of multiple cores and a more complicated structure group 
(see e.g. \cite{gra:dua}).\medskip

After completion of this work, Madeleine Jotz Lean informed us that she had independently obtained 
a geometric construction of the Weil algebra for `split' double Lie algebroids, using the ideas from  \cite{gra:dla,jot:lie2}. 

\bigskip
\noindent{\bf Acknowledgments.} It is a pleasure to thank Francis Bischoff, Henrique Bursztyn, Madeleine Jotz Lean, Yvette Kosmann-Schwarzbach, Kirill Mackenzie, and Luca Vitagliano for fruitful discussions and comments related to this paper. We also thank the referees for numerous helpful suggestions.
\bigskip

\section{Double vector bundles}\label{sec:dvb}
\subsection{Definitions}
The concept of a double vector bundle ($\DVB$) was introduced by Pradines \cite{pra:rep,pra:fib}
in terms of local charts, and later reformulated as manifolds with `commuting' vector bundle structures \cite{mac:dou}.  We shall work with an elegant approach due to Grabowski-Rotkiewicz \cite{gra:hig}, who observed that vector bundle structures on manifolds $V$ are completely determined by their scalar multiplications $\kappa_t\colon V\to V,\ t\in\R$, and vector bundle morphisms $V\to V'$ are exactly the smooth maps intertwining scalar multiplications. Similarly, consider a manifold $D$ with two vector bundle structures, 
referred to as vertical and horizontal, respectively. The corresponding scalar multiplications are denoted $\kappa_t^v,\,\kappa_s^h$. Then $D$ is called a \emph{double vector bundle} if 
\begin{equation}\label{eq:kappacommute}
\kappa_t^v\,\kappa_s^h=\kappa_s^h\kappa_t^v
\end{equation}
for all $t,s\in \R$. A  \emph{morphism of double vector bundles} ($\DVB$ morphism) $\varphi\colon D\to D'$
 is a smooth map  intertwining both the horizontal and the vertical scalar multiplications. The condition \eqref{eq:kappacommute} implies a list of compatibilities with the horizontal and vertical addition operations, such as the \emph{interchange property} \cite{mac:dou}, however, we will not need these in this paper. 

Let $A=\kappa_0^v(D)$ and $B=\kappa_0^h(D)$ be the base submanifolds for the two vector bundle structures; each of these is a vector bundle 
over the submanifold $M=\kappa_0^h\kappa_0^v(D)$. The situation is depicted 
by a diagram
\begin{equation}\label{eq:dvb}
\xymatrix{ {D} \ar[r] \ar[d] & B \ar[d]\\
	A \ar[r] & M}
\end{equation}
One calls $A,B$ the \emph{side bundles}, and $M$ the \emph{base} of the double vector bundle. 
 The projections to the side bundles combine into a $\DVB$ morphism 
 \begin{equation}\label{eq:varphi}
  \varphi\colon D\to A\times_M B,\end{equation}
 where $A\times_M B$ has the double vector structure given by the two obvious scalar multiplications. (See also Section \ref{subsec:examples} below.)  As stated in \cite[Section 4]{gra:hig}, 
 this map is a surjective submersion. (Proof: At points 
 $m\in M\subset D$, the tangent map $T_m\varphi\colon T_mD\to  T_m(A\times_M B)=T_mA\oplus T_mB$ is surjective, with right inverse given by the tangent maps to the inclusions $A\hra D,\ B\hra D$.  Hence $\varphi$ is a submersion on a neighborhood of $M$ in $D$, and 
 by equivariance with respect to $\kappa_t^h \kappa_s^v$ it is a submersion everywhere.)
 It follows that the preimage of the base under this map  is a submanifold 
 \[ \core(D)=\varphi^{-1}(M)\subset D\]
 called the \emph{core} of the double vector bundle. It admits an alternative characterization as the submanifold on which the two scalar multiplications coincide \cite{gra:hig}:
 \begin{equation}\label{eq:coincide}
 \core(D)=\{d\in D|\ \forall t\in\R\colon \kappa_t^h(d)=\kappa_t^v(d)\}.
\end{equation}
The restriction of $\kappa_t^h$ (or, equivalently, of $\kappa_t^v$) is the scalar multiplication 
 for a vector bundle structure on $\core(D)\to M$. The core may be regarded as a subbundle of $D$, for each of the two vector bundle structures.  
  From now on,  we will reserve the notation  
  \[ \CC=\core(D)^*\]  
  for the \emph{dual bundle} of the core. 
  
 \begin{remark}
 We would prefer the letter $C$, since we will make extensive use of a cyclic symmetry interchanging the bundles $A, B$, and $\on{core}(D)^*$; see  Section \ref{subsec:triality} below. However, since $C$ is commonly used to denote the core itself, 
this might cause confusion with the existing literature.  
 \end{remark}

  
\subsection{Examples}\label{subsec:examples}
Here are some examples of double vector bundles:
\begin{enumerate}
\item 
\label{it:ex1}
If $A,B,\CC$ are vector bundles over $M$, then $A\times_M B\times_M \CC^*$ is a double vector bundle
\[
\xymatrix{ {A\times_M B\times_M \CC^*} \ar[r] \ar[d] & B \ar[d]\\
	A \ar[r] & M}
	\]
with core given by $\CC^*$. The horizontal and vertical scalar multiplications are given by 
$\kappa_t^h(a,b,\ggamma)=(ta,b,t\ggamma)$ and $\kappa_t^v(a,b,\ggamma)=(a,tb,t\ggamma)$, respectively. In particular, any vector bundle $V\to M$ can be regarded as a double vector bundle in three ways, 
by playing the role of $A,B$ or $\CC^*$. 
\item\label{it:ex2}
 If $V\to M$ is any vector bundle, then its tangent bundle and cotangent bundle are double vector bundles 
\[ \xymatrix{ {TV} \ar[r] \ar[d] & TM \ar[d]\\
	V \ar[r] & M }         \ \ \ \            \ \ \ \     \xymatrix{ {T^*V} \ar[r] \ar[d] & V^* \ar[d]\\
		V \ar[r] & M }\]  
with  $\core(TV)=V$ (thought of as the vertical bundle of $TV|_M$) and 
$\core(T^*V)=T^*M$. 
The $\DVB$ structure on $TV$ appeared in \cite{pra:fib}; the $\DVB$ structure on $T^*V$ was first discussed in \cite{mac:lie}. 
\item\label{it:ex3}
Suppose $V\to M$ is a subbundle of a vector bundle 
$W\to Q$. Then the normal and conormal bundle of $V$ in $W$ are double vector bundles 
\[ \xymatrix{ \nu(W,V) \ar[r] \ar[d] & \nu(Q,M)\ar[d]\\
V	\ar[r] & M}\ \ \ \ \  
	\xymatrix{ \nu^*(W,V) \ar[r] \ar[d] & 	(W|_M/V)^* \ar[d]\\
V	\ar[r] & M }
	\] 
with 
$\core(\nu(W,V))=W|_M/V$ and 
$ \core(\nu^*(W,V))=\nu^*(Q,M)$. 
%
%
\item\label{it:ex4}
Let $M_1,M_2$ be submanifolds of a manifold $Q$, 
with clean intersection \cite{ho:an3}. Then there is a \emph{double normal bundle}
with base $M=M_1\cap M_2$, 
\[  \xymatrix{ \nu(Q,M_1,M_2) \ar[r] \ar[d] & \nu(M_1,M) \ar[d]\\
	\nu(M_2, M)\ar[r] & M}\]  
with core $ TQ|_M/(TM_1|_M+TM_2|_M)$. (Note that the core is trivial if and only if the intersection is transverse.) The double normal bundle can be defined as an iterated normal bundle $\nu\big(\nu(Q,M_2),\nu(M_1,M) \big)$, or more symmetrically, as follows: Give $C^\infty(Q)$ the bi-filtration by the order of vanishing on the two submanifolds, and let $\S$ be the associated bi-graded algebra. Then the double normal bundle is recovered as the character spectrum $\nu(Q,M_1,M_2)=\Hom_{\on{alg}}(\S,\R)$. In other words, $\S$ is the algebra $\S(D)$ of double-polynomial functions on $D=\nu(Q,M_1,M_2)$. 
%
More details on this example will be given in a forthcoming work. 
\end{enumerate}
Given a double vector bundle $D$ as in \eqref{eq:dvb}, one can switch the roles of the horizontal and vertical scalar multiplications. This defines the \emph{flip}
 \begin{equation}\label{eq:dvb_flip}
 \xymatrix{ \on{flip}(D) \ar[r] \ar[d] & A \ar[d]\\
 	B \ar[r] & M}
 \end{equation}
A more interesting way of obtaining new double vector bundles is by taking dual bundles,  horizontally or vertically. We will denote these by $D^h$ and $D^v$, respectively: 
\[
\xymatrix{ {D}^h \ar[r] \ar[d] & B \ar[d]\\
	\CC \ar[r] & M}\ \ \ \ \ \ \ \ \ \ \ \ \ \ \xymatrix{ {D}^v \ar[r] \ar[d] & \CC \ar[d]\\
	A \ar[r] & M}\]
Also common are the notations $D\duer B,\ D\duer A$ (e.g., \cite{mac:gen}) or 
$D^*_B,\ D^*_A$ (e.g., \cite{gra:vba}).
The horizontal and vertical scalar multiplication on $D^h$ 
are characterized by the property 
\[ \l \kappa_t^h(\phi),d \r=t\l\phi,d\r=
\l \kappa_t^v(\phi),\kappa_t^v(d)\r,\]
for $\phi\in D^h,\ d\in D$ in the same fiber over $B$. 
Similarly, for $\psi\in D^v,\ d\in D$ in the same fiber over $A$, 
\[ \l \kappa_t^h(\psi),\kappa_t^h(d)\r=t\l\psi,d\r=
 \l \kappa_t^v(\psi), d\r.\]
One finds $\core(D^h)=A^*,\ \core(D^v)=B^*$. Clearly, $\on{flip}(D^v)=\on{flip}(D)^h$. In Example \ref{subsec:examples}\eqref{it:ex2}, $T^*V$ is the vertical dual of $TV$, while in Example \ref{subsec:examples}\eqref{it:ex3}, $\nu^*(W,V)$ is the vertical dual of $\nu(W,V)$.

\subsection{Splittings}
Let $D$ be a double vector bundle over $M$, with side bundles $A,B$ and with $\core(D)=\CC^*$.
A \emph{splitting} (or \emph{decomposition}) of $D$ is a $\DVB$ isomorphism 
\[ D\to A\times_M B\times_M \CC^*,\]
inducing the identity on $A,B,\CC^*$. 
Here the $\DVB$ structure on the right hand side is as in Example \ref{subsec:examples}\eqref{it:ex1}. 

\begin{example}\label{ex:TV1}
Let $V\to M$ be a vector bundle, and $TV$ its tangent bundle regarded as a $\DVB$. 
A splitting  of $TV$  is equivalent to a linear connection $\nabla$ on $V$. 
(Cf. \cite[Example 2.12]{gra:vba}.) 
\end{example}

\begin{theorem}\label{th:splittings}
	Every double vector bundle admits a splitting. 
	\end{theorem}
This result was stated in \cite{gra:vba} with a reference to \cite{gra:hig}; a detailed proof was given in the Ph.D. thesis of del Carpio-Marek \cite{del:geo}. (The recent paper \cite{jot:mult} by Heuer and Jotz Lean generalizes this result to $n$-fold vector bundles.) 
In the appendix, we will present a somewhat shorter argument. 

Combining the existence of splittings with local trivializations of $A,B,\CC^*$, we see in particular that every double vector bundle $D$ is a  fiber bundle over its base manifold $M$, with bundle projection $q=\kappa_0^h\kappa_0^v\colon D\to M$. Its fibers $D_m=q^{-1}(m)$ are \emph{double vector spaces} (i.e., double vector bundles over a point). 


\subsection{The associated principal bundle}\label{subsec:associated}
Given non-negative integers $n_1,n_2,n_3\in \Z_{\ge 0}$, put 
$A_0=\R^{n_1},\ B_0=\R^{n_2},\ \CC_0=\R^{n_3}$, 
and let 
$D_0$ be the double vector space 
\begin{equation}\label{eq:d0}
 D_0=A_0\times B_0\times \CC_0^*\end{equation}
with  $ \kappa_t^h(a,b,\ggamma)=(a,tb,t\ggamma)$ and 
$\kappa_t^v(a,b,\ggamma)=(ta,b,t\ggamma)$. (Cf. Example \ref{subsec:examples}\eqref{it:ex1}.) 
For any vector space $V_0$, we denote by $\GL(V_0)$ its general linear group; it comes with standard representations on $V_0$ and on the dual space $V_0^*$. Thus, $\GL(A_0)\times \GL(B_0)\times \GL(\CC_0)$ has a standard action on the double vector space \eqref{eq:d0}.

 \begin{lemma}\label{lem:gaugegroup} \cite{gra:dua}
 The group of $\DVB$ automorphisms of $D_0=A_0\times B_0\times \CC_0^*$ is 
 a semi-direct product
 \begin{equation}\label{eq:semidirect}
  \Aut(D_0)=\big(\GL(A_0)\times \GL(B_0)\times \GL(\CC_0)\big)\ltimes (A_0^*\otimes B_0^*\otimes \CC_0^*).\end{equation}
 with the standard action of $\GL(A_0)\times \GL(B_0)\times \GL(\CC_0)$, 
 and with $\omega\in A_0^*\otimes B_0^*\otimes \CC_0^*$ acting as
 \begin{equation}\label{eq:omegaaction} (a,b,\ggamma)\mapsto (a,b,\ggamma+\omega(a,b)).\end{equation}
 \end{lemma}

 Given a double vector bundle $D$, take $n_1,n_2,n_3$ to be the ranks of the bundles $A,B,E$. An isomorphism  of double vector spaces $D_m\to D_0$ will be called a
 \emph{frame of $D$ at $m\in M$}.  Clearly, any two frames are related by the action of $\Aut(D_0)$.  We define the \emph{frame bundle} of $D$  to be the principal $\Aut(D_0)$-bundle $ P\to M$ whose fibers $P_m$ are the set of frames at $m$. 

\begin{remark}
In \cite{lan:dou}, the semi-direct product \eqref{eq:semidirect} 
is regarded as a \emph{double Lie group}, and a general  theory of double principal bundles for double Lie groups is developed. 
\end{remark}
Many constructions with double vector bundles may be expressed 
in terms of bundles associated to $P$. In particular, $D$ is itself an associated bundle for the action \eqref{eq:omegaaction},
\begin{equation}\label{eq:associated} D=(P\times  D_0)/\Aut(D_0).\end{equation}
(Pradines' original definition \cite{pra:rep} of double vector bundles was in terms of local trivializations with $\on{Aut}(D_0)$-valued transition functions.)  A splitting of $D$ amounts to a reduction of the structure group  to $\GL(A_0)\times \GL(B_0)\times \GL(\CC_0)\subset \Aut(D_0)$; the fibers $Q_m$ of the reduction $Q\subset P$ are all those $\DVB$ isomorphisms $D_m\to D_0$ that also preserve the splittings. 

Let $D_0^-$ be equal to $D_0$ as a double vector space, but with the new action of $\Aut(D_0)$, 
where  $\omega\in A_0^*\otimes B_0^*\otimes \CC_0^*$
acts by $(a,b,\ggamma)\mapsto (a,b,\ggamma-\omega(a,b))$, while 
$\GL(A_0)\times \GL(B_0)\times \GL(\CC_0)$ acts in the standard way. The resulting double vector bundle
\[ D^-=(P\times D_0^-)/\Aut(D_0).\]
will feature in some of the constructions below. 
\begin{lemma}\label{lem:minus}
There is a canonical $\DVB$ isomorphism $ D^-\to D$ that is the identity on the side bundles  but minus the identity on the core.   
\end{lemma}
\begin{proof}
The isomorphism is induced by the $\Aut(D_0)$-equivariant isomorphism 
of double vector spaces $D_0^-\to D_0,\ (a,b,\ggamma)\mapsto (a,b,-\ggamma)$.
\end{proof}
\subsection{Triality of double vector bundles}\label{subsec:triality}
By cyclic permutation of the roles of $A_0,B_0,\CC_0$, the action \eqref{eq:omegaaction}  of $\Aut(D_0)$ on $D_0=A_0\times B_0\times \CC_0^*$ gives rise to similar actions on $D_0'=B_0\times \CC_0\times A_0^*$ and 
$D_0''=\CC_0\times A_0\times B_0^*$. The bilinear pairings 
\begin{equation}\label{eq:dvbpair1}
D_0\times_{B_0} D_0'\to \R,\ \big((a,b,\ggamma),(b,\cc,\alpha)\big)\mapsto \alpha(a)-\ggamma(\cc),\end{equation}
and similar maps given by cyclic permutation, are $\on{Aut}(D_0)$-equivariant. 
Taking associated bundles, we obtain three double vector bundles 
\[
\xymatrix{ D \ar[r] \ar[d] & B \ar[d]\\
	A \ar[r] & M}\ \ \ \ \ \ \ \xymatrix{ D' \ar[r] \ar[d] & \CC \ar[d]\\
	B \ar[r] & M}\ \ \ \ \ \ \ 
	\xymatrix{ D'' \ar[r] \ar[d] & A \ar[d]\\
		\CC \ar[r] & M}\]
with bilinear pairings 
\begin{equation}\label{eq:dvbpairings}
D\times_B D'\to \R,\ \ \ \ D'\times_\CC D''\to \R,\ \ \ \ D''\times_A D\to \R
\end{equation}
The bundles $D',D''$ are closely related to the horizontal and vertical duals:

\begin{proposition}\label{prop:dual}
There are canonical $\DVB$ isomorphisms 
\[ D^h\cong \on{flip}(D')^-,\ \ \ D^v\cong \on{flip}(D'')^-\]
that are the identity on 
the side bundles  and on the core. 
\end{proposition}
\begin{proof}
We give the proof for $D^h$ (the argument for $D^v$ is similar). 	It suffices to consider the double vector space $D_0$. Write $D$ as an associated bundle \eqref{eq:associated}. Then 
\[ D^h=(P\times D_0^h)/\on{Aut}(D_0)\] 
where the $\on{Aut}(D_0)$-action on $D_0^h=\CC_0\times B_0\times A_0^*$ is given by 
the standard action of $\GL(A_0)\times \GL(B_0)\times \GL(\CC_0)$, while  $\omega\in A_0^*\otimes B_0^*\otimes \CC_0^*$ acts as 
\[  (\cc,b,\alpha)\mapsto (\cc,b,\alpha-\omega(b,\cc)).\]
This action is dictated by invariance of the duality 
pairing 
\[
 D_0\times_{B_0} D_0^h\to \R,\ \big((a,b,\ggamma),(\cc,b,\alpha)\big)\mapsto \alpha(a)+\ggamma(\cc).
 \]
The $\DVB$-isomorphism  $D_0^h\to \on{flip}(D_0'),\ (\cc,b,\alpha)\mapsto (\cc,b,-\alpha)$ is  $\on{Aut}(D_0)$-equivariant, and induces a $\DVB$-isomorphism $D^h\to \on{flip}(D')$ that is the identity on the sides but minus the identity on the core. Now use
Lemma \ref{lem:minus}. 
\end{proof}


\begin{remark}
Using the isomorphisms from Proposition \ref{prop:dual}, the second pairing in \eqref{eq:dvbpairings} translates 
into Mackenzie's pairing $D^v\times_\CC D^h\to \R$ \cite[Theorem 3.1]{mac:sym}. 		
We also recover the result of Mackenzie \cite{mac:sym} and Konieczna and Urba\'nski \cite{kon:dou}, giving a canonical $\DVB$ isomorphism 
\[ ((D^h)^v)^h\cong ((D^v)^h)^v\]
that is the identity on the side bundles and on the core; indeed, by iteration of Proposition
\ref{prop:dual} we see that both are identified with $D^-$. As a special case, if $D=T^*V$ we 
have that $D^v=TV$, $(D^v)^h=TV^*$, $((D^v)^h)^v=T^* V^*$, 
and we recover the canonical $\DVB$ isomorphism 
\begin{equation}\label{eq:mackenziexu} T^*(V^*)\cong \on{flip}(T^* V)^-\end{equation}
of Mackenzie-Xu \cite{mac:lie}.
\end{remark}

\begin{example} 
Consider  $D=TV$ as a double vector bundle with sides $A=V$ and $B=TM$. The natural pairing between tangent and cotangent vectors identifies $D^v=T^*V$, while the tangent prolongation of the pairing $V\times_M V^*\to \R$ identifies 
$D^h=TV^*$. The three double vector bundles $D,D',D''$ are
therefore, 
\[
\xymatrix{ TV \ar[r] \ar[d] & TM\ar[d]\\
	V \ar[r] & M}\ \ \ \ \ \ \ \xymatrix{\on{flip}(T(V^*))^- \ar[r] \ar[d] & V^* \ar[d]\\
	TM \ar[r] & M}\ \ \ \ \ \ \ 
	\xymatrix{ \on{flip}(T^*V)^- \ar[r] \ar[d] & V \ar[d]\\
		V^* \ar[r] & M}\]
Put differently, there is a canonical $\DVB$ isomorphism $(TV)'\to \on{flip}(T(V^*))$ that is the identity 
on the side bundles and minus identity on the core $V^*$, 
and a canonical $\DVB$ isomorphism 
$(TV)''\to \on{flip}(T^*V)$ that is the identity on the side bundles and minus the identity on the 
core $T^*M$. 
\end{example}


\section{Double-linear functions}\label{sec:symmetric}
For any vector bundle $V\to M$, the fiberwise linear 
functions on $V$ are identified with the sections of the dual bundle  $V^*$. We will similarly associate to any double vector bundle $D$ a vector bundle whose space of 
sections are the functions on $D$ that are 
\emph{double-linear}, i.e.,  linear for both scalar multiplications. Throughout this discussion, we will find it convenient to present
$D$ as an associated bundle $D=(P\times D_0)/\Aut(D_0)$ with 
$D_0=A_0\times_M B_0\times_M \CC_0^*$.

\subsection{Double-linear functions}\label{subsec:double_linear}

Consider the $\Aut(D_0)$-action on 
\[ \wh{\CC}_0=(A_0^*\otimes B_0^*)\oplus \CC_0,\]
 where $\GL(A_0)\times \GL(B_0)\times \GL(\CC_0)$ acts in the standard way, while elements 
$\omega\in A_0^*\otimes B_0^*\otimes \CC_0^*\cong \Hom(\CC_0,A_0^*\otimes B_0^*)$ act
as 
\[ (\nu,\cc)\mapsto (\nu-\omega(\cc),\cc).\]
The projection $\wh{\CC}_0\to \CC_0$ is $\Aut(D_0)$-equivariant, with kernel $A_0^*\otimes B_0^*$. 
Taking associated bundles, we obtain a vector bundle  
\[ \wh{\CC}=(P\times \wh{\CC}_0)/\Aut(D_0)\]
with an exact sequence of vector bundles over $M$,  
\begin{equation}\label{eq:hatCexactsequence}
 0\lra A^*\otimes B^*\stackrel{i_{\wh{\CC}}}{\lra} \wh{\CC}\lra \CC\lra 0.
\end{equation}
Using cyclic permutation of $A_0,B_0,\CC_0$, we obtain three such bundles $\wh{A},\ \wh{B},\ \wh{\CC}$, 
with inclusion maps 
\begin{equation}\label{eq:3inclusions}
i_{\wh{\CC}}\colon A^*\otimes B^*\to \wh{\CC},\ \ \ \ 
i_{\wh{A}}\colon B^*\otimes \CC^*\to \wh{A},\ \ \ \  
i_{\wh{B}}\colon \CC^*\otimes A^*\to \wh{B},\end{equation}
and exact sequences similar to  \eqref{eq:hatCexactsequence}. In Section \ref{subsec:geometricinterpretation} below we will identify the bundles $\wh{A},\wh{B}$ with those introduced by Gracia-Saz and Mehta \cite{gra:vba}; the corresponding exact sequences appear as Equation (26) in that reference.

\begin{proposition}
The space of sections of $\wh{\CC}$ is canonically isomorphic to the space of double-linear functions on $D$. Under this identification, the quotient map to $\CC$ is given by restriction of double-linear functions to $\core(D)=\CC^*$, and the inclusion map $i_{\wh{\CC}}$ is given by the multiplication of pull-backs of 
linear functions on $A$ and on $B$. 
\end{proposition}
\begin{proof}
It suffices to prove these claims for the double vector space  $D_0$. Using a Taylor expansion, 
we see  that the double-linear functions on $D_0=A_0\times B_0\times \CC_0^*$ are 
$\wh{\CC}_0=(A_0^*\otimes B_0^*)\oplus \CC_0$, where $\CC_0$ is interpreted as linear functions on $\CC_0^*$ and $A_0^*\otimes B_0^*$ as linear combinations of products of linear functions on $A_0,B_0$. 
\end{proof}
In terms of this interpretation through double-linear functions, the exact sequence  \eqref{eq:hatCexactsequence} was discussed by Chen-Liu-Sheng \cite{che:dou}
as the dual of the \emph{$\DVB$ sequence}. A central result of their paper is 
that the double vector bundle may be recovered from this sequence: 
\begin{proposition}\cite{che:dou}\label{prop:chen}
	The double vector bundle $D$ is the sub-double vector bundle of 
	\[ \wh{D}=A\times_M B\times_M\wh{\CC}^*\] 
	consisting of all $(a,b,\wh{\ggamma})\in A\times_M B\times_M\wh{\CC}^*$ such that $i_{\wh{\CC}}^*(\wh{\ggamma})=a\otimes b$.  A splitting of $D$ is equivalent to a splitting of 
	the exact sequence \eqref{eq:hatCexactsequence}.
\end{proposition}
An alternative way of proving this result is to verify the analogous statement for the double vector space $D_0=A_0\times B_0\times \CC_0^*$.

\begin{remark}
A direct consequence is that every double vector bundle $D$ comes with a map 
\begin{equation}\label{eq:chatstarprojection} D\to \wh{\CC}^*\end{equation}
given by the inclusion $D\hookrightarrow \hat D$ followed by projection to $\wh{\CC}^*$.  This map is a $\DVB$-morphism if the vector bundle $\wh{\CC}^*$ is regarded as a double vector bundle (with zero sides). 
In terms of the associated bundle construction, this is induced by the map 
\[ D_0=A_0\times B_0\times \CC_0^*\to \wh{\CC}_0^*=
(A_0\otimes B_0)\oplus \CC_0^*,\ \ (a,b,\ggamma)\mapsto a\otimes b+\ggamma.\]
\end{remark}

\begin{remark}\label{rem:quotient}
The inclusion $D\hra \wh{D}$ dualizes to a surjective $\DVB$ morphism
\[ \wh{D}'=B\times_M \wh{\CC}\times_M A^*\to D'.\] 
Replacing $D'$ with $D$, this shows that every double vector bundle also arises as a \emph{quotient} of a split double vector bundle.
\end{remark}

\begin{remark}\label{rem:splittings}
Since a splitting of $D$ is equivalent to a splitting of $D',D''$, we see that a splitting of 
$D$ is equivalent to a splitting of any one of the three vector bundle maps $\wh{A}\to A,\ \wh{B}\to B$ or $\wh{\CC}\to \CC$. 
\end{remark}

\subsection{The three pairings}
In what follows, we will denote elements of the bundles  $\wh{A},\wh{B},\wh{\CC}$ by 
 $\wh{a},\, \wh{b},\,  \wh{\cc}$, and their images in $A,B,\CC$ by $a,\, b,\, e$. 
%
%
\begin{proposition}
	There are canonical bilinear pairings 
	\begin{align}
	\l\cdot,\cdot\r_{\CC^*}\colon& \wh{B}\times_M \wh{A}\to \CC^*,\nonumber\\  
	\l\cdot,\cdot\r_{A^*}\colon& \wh{\CC}\times_M \wh{B}\to A^*,\label{eq:three_pairings}\\ 
	\l\cdot,\cdot\r_{B^*}\colon& \wh{A}\times_M \wh{\CC}\to B^*, \nonumber 
	\end{align}
	with the properties
	\begin{align}\label{eq:properties1}	
	\big\l  \wh{b},\ i_{\wh{A}}(\mu) \big\r_{\CC^*}&=\mu(b),\ \ \ 
	\mu\in B^*\otimes \CC^*,\ \wh{b}\in \wh{B},\\ 
	\big\l i_{\wh{B}}(\nu),\ \wh{a}  \big\r_{\CC^*}&=-\nu(a),\label{eq:properties2}	\ \ \ \ \ 
	\nu\in \CC^*\otimes A^*,\ \ \wh{a}\in\wh{A},
	\end{align}
	and similar properties obtained by cyclic permutations of $A,B,\CC$. 
	The pairings are related by the identity
		\begin{equation}\label{eq:identityc}\big\l  \wh{b}, \wh{a} \big\r_{\CC^*}(\cc)+ \big\l\wh{\cc},\wh{b}\big\r_{A^*}(a)+ \big\l \wh{a}, \wh{\cc}\big\r_{B^*}(b)=0.\end{equation}
\end{proposition}
\begin{proof}
Using the associated bundle construction, it suffices to define the corresponding pairings for the double vector space $D_0$, and check that they are $\Aut(D_0)$-equivariant.  We have 
\[ \wh{A}_0=(B_0^*\otimes \CC_0^*)\oplus A_0,\ \  \ \wh{B}_0= ( \CC_0^*\otimes A_0^*)\oplus B_0, \ \ \ \wh{\CC}_0= ( A_0^*\otimes B_0^*)\oplus \CC_0.\] 
Put
\begin{equation}\label{eq:zeropairings}
\l\cdot,\cdot\r_{\CC_0^*}\colon \wh{B}_0\times \wh{A}_0\to \CC_0^*,\ \ 
	 \l (\nu,b)  ,\,  (\mu,a)\r_{\CC_0^*}=\mu(b)-\nu(a);\end{equation}
This is clearly equivariant for the actions of $\GL(A_0)\times \GL(B_0)\times \GL(\CC_0)$. 
For the action of $\omega\in A_0^*\otimes B_0^*\otimes \CC_0^*$, observe that 
in the pairing between
\[ \omega.(\mu,a)=(\mu-\omega(a),a),\ \ \omega.(\nu,b)=(\nu-\omega(b),b),\]
 the terms involving $\omega$ cancel. The properties \eqref{eq:properties1} and \eqref{eq:properties2}	hold by definition. Furthermore, given 
$\wh{a}=(\mu,a)\in\wh{A}_0,\ \wh{b}=(\nu,b)\in\wh{B}_0,\ \wh{\cc}=(\rho,\cc)\in\wh{\CC}_0$ the three terms in \eqref{eq:identityc} are 
$\mu(\cc,a)-\nu(b,\cc) $,\ $\rho(a,b)-\mu(\cc,a)  $ and $\nu(b,\cc)-\rho(a,b) $, hence their sum is zero.  
\end{proof}
As we shall explain in Remark \ref{rem:warp} below, the pairing $ \l\cdot,\cdot\r_{\CC^*}$ is equivalent to Mackenzie's notion of 
`warp' \cite{mac:jac}.\medskip

Replacing $D$ with $\on{flip}(D)$ reverses the role of $A$ and $B$. Hence, the three inclusion maps 
\eqref{eq:3inclusions} are unchanged, but the three pairings \eqref{eq:three_pairings} all change sign. 
On the other hand, for $D^-$ we have: 
\begin{proposition}\label{prop:d-}
The bundles $\wh{A},\wh{B},\wh{C}$ for $D$ are canonically isomorphic to those for $D^-$. 
Under this identification, replacing $D$ with $D^-$ changes the signs of 
the inclusion maps \eqref{eq:3inclusions} and also of the pairings \eqref{eq:three_pairings}. 
\end{proposition} 
\begin{proof}
Let $\wh{A}^-,\wh{B}^-,\wh{\CC}^-$ be the corresponding bundles for $D^-$. 
The $\Aut(D_0)$-equivariant isomorphism 
\[ \wh{\CC}_0^-\to \wh{\CC}_0,\  
(\nu,c)\mapsto (-\nu,c)\] 
gives the desired isomorphism $\wh{\CC}^-\to \wh{\CC}$, and similar for $\wh{A}^-,\ \wh{B}^-$. One readily checks that these isomorphisms give sign changes for the inclusions and pairings. 
\end{proof}

\subsection{Geometric interpretations}\label{subsec:geometricinterpretation}
The bundles $\wh{A},\wh{B},\wh{\CC}$ and the pairings between them have various geometric interpretations, in terms of functions and vector fields on $D$. 

We begin by recalling analogous interpretations for vector bundles $V\to M$. The space $\mf{X}(V)_{[r]}$ of vector fields on $V$ that are homogeneous of degree $r$ for the scalar multiplication (i.e., $\kappa_t^*X=t^r X$ for $t\neq 0$) is  trivial if $r<-1$, while the \emph{core} and \emph{linear} vector fields 
\begin{equation}\label{eq:2ident}
 \mf{X}(V)_{[-1]}:=\mf{X}_\core(V),\ \ \ \mf{X}(V)_{[0]}=:\mf{X}_\lin(V)
\end{equation}
are identified with sections of $V$ (via the \emph{vertical lift}, taking a section  $\sigma\in \Gamma(V)$ to the corresponding fiberwise constant vector field $\sigma^\sharp$),  
and infinitesimal automorphisms of $V$, respectively.   On the other hand, $C^\infty(V)_{[0]}=C^\infty(M)$ and $C^\infty(V)_{[1]}=\Gamma(V^*)$. 

The pairing $V\times_M V^*\to \R$ is realized as the map $\mf{X}(V)_{[-1]}\otimes 
C^\infty(V)_{[1]}\to C^\infty(V)_{[0]}$ given by Lie derivative, $X\otimes f\mapsto \L_X f$. 
We can also take a dual  viewpoint (`Fourier transform'), 
using the identifications
$ C^\infty(V^*)_{[0]}=C^\infty(M),\ C^\infty(V^*)_{[1]}=\Gamma(V),\ 
\mf{X}(V^*)_{[-1]}=\Gamma(V^*)$. Here, we realize the pairing  $V\times_M V^*\to \R$ as 
\emph{minus} the Lie derivative, $h\otimes Z\mapsto -\L_Z h$. (Working with multi-vector fields, it is convenient to think of these pairings as Schouten brackets\footnote{Recall that the Schouten bracket on multi-vector fields on a manifold $Q$ makes $\mf{X}^\bullet(Q)[1]$ into a graded super-Lie algebra, in such a way that the bracket extends the usual Lie bracket of vector fields and satisfies 
	$\Lie{X,f}=\L_X(f)$ for vector fields $X$ and functions $f$.  In particular we have 
	\[ \Lie{X,Y}=-(-1)^{(k-1)(l-1)}\Lie{Y,X},\ \ \ \ \  \Lie{X,\Lie{Y,Z}}=\Lie{\Lie{X,Y},Z}+(-1)^{(k-1)(l-1)}\Lie{Y,\Lie{X,Z}}\]
	for $X\in \mf{X}^k(D),\ Y\in \mf{X}^l(D),\ Z\in \mf{X}^m(D)$. The map $X\mapsto \Lie{X,\cdot}$ is by graded derivation of the wedge product: 
	\[ \Lie{X,Y\wedge Z}=\Lie{X,Y}\wedge Z+(-1)^{(k-1)l} Y\wedge \Lie{X,Z}.\]
}  between 1-vector fields and 
0-vector fields.)


For a double vector bundle, let $\mf{X}(D)_{[k,l]}$ be the space of vector fields on $D$ that are homogeneous of degree $k$ horizontally and of degree $l$ vertically. Similar notation will be used for smooth functions, differential forms, and so on.  Let $\Gamma(D,A)$ be the sections of $D$ as a vector bundle over $A$, and $\Gamma_\lin(D,A)$ the subspace of sections that are homogeneous of degree $0$ horizontally, i.e., such that the corresponding map $A\to D$ is $\kappa_t^h$-equivariant. Linear sections of $D$ over $B$ are defined similarly. We have 
\begin{equation}\label{eq:dadb}
 \Gamma_\lin(D,A)\cong \mf{X}(D)_{[0,-1]},\ \ \ \ 
\Gamma_\lin(D,B)\cong \mf{X}(D)_{[-1,0]},\ \ \ 
\Gamma(\CC^*)\cong \mf{X}(D)_{[-1,-1]}.
\end{equation}
To verify \eqref{eq:dadb}, note that vector fields  $X\in \mf{X}(D)_{[k,l]}$ with $k=-1$ or $l=-1$ annihilate  
$C^\infty(D)_{[0,0]}=C^\infty(M)$, and are thus vertical for the bundle projection 
$D\to M$. Hence, it suffices to check for the double vector space  $D_0=A_0\times B_0\times \CC_0^*$. But 
\[\mf{X}(D_0)_{[0,-1]}\cong B_0\oplus (A_0^*\otimes \CC_0^*),\ \  \mf{X}(D_0)_{[-1,0]}= A_0\oplus (B_0^*\otimes \CC_0^*),\ \ \ 
\mf{X}(D_0)_{[-1,-1]}\cong  \CC_0^*,\]
where elements of a vector space  are seen as constant vector fields on the vector space, and elements of the dual space as linear functions.  Indeed, with this interpretation the elements of $A_0,B_0,\CC_0^*$ have homogeneity bidegrees $(-1,0),(0,-1),(-1,-1)$ respectively, while  elements of $A_0^*,B_0^*,\CC_0$ have homogeneity bidegrees $(1,0),(0,1),(1,1)$.

\begin{proposition}\label{prop:alternatives}
 The space of sections of $\wh{\CC}$ is canonically isomorphic to 
\begin{enumerate}
\item the space $C^\infty(D)_{[1,1]}$ of double-linear functions on $D$,
\item the space $\Gamma_\lin(D',B)$ of linear sections of $D'$ over $B$, 
\item the space $\Gamma_\lin(D'',A)$ of linear sections of $D''$ over $A$,
\item the space $\mf{X}(D')_{[0,-1]}$ 
of vector fields on $D'$ of homogeneity $(0,-1)$, 
\item the space $\mf{X}(D'')_{[-1,0]}$ of vector fields on $D''$ of homogeneity $(-1,0)$. 
\end{enumerate}
Similar descriptions hold for sections of $\wh{A},\wh{B}$. 
\end{proposition}
\begin{proof}
It suffices to prove these descriptions for the double vector spaces $D_0=A_0\times B_0\times \CC_0^*$. 
We have already remarked that $\wh{\CC}_0$ is the space of double-linear functions on $D_0$. 
For (b), note that sections of  $D_0'=B_0\times \CC_0\times A_0^*$ over $B_0$ are  smooth functions $B_0\to \CC_0\times A_0^*$;  such a function defines a linear section if and only if its first component is a constant map $B_0\to \CC_0$, while its second component is a linear map $B_0\to A_0^*$. Hence, we obtain $\Gamma_\lin(D_0',B_0)=(A_0^*\otimes B_0^*)\oplus \CC_0=\wh{\CC}_0$. The proof of (c) is similar, and  by the counterparts of \eqref{eq:dadb} for $D_0',D_0''$ the properties (b),(c) are equivalent to (d),(e). 
\end{proof}

\begin{remark}\label{rem:warp}
	 In the work of Gracia-Saz and Mehta  \cite[Section 2.4]{gra:vba}, 
the isomorphisms $\Gamma(\wh{A})\cong \Gamma_\lin(D,B)$ and $\Gamma(\wh{B})\cong \Gamma_\lin(D,A)$ are used as the definition of $\wh{A}, \wh{B}$. With these identifications, the pairing $\l\cdot,\cdot\r_{\CC^*}\colon \Gamma_\lin(D,A)\times \Gamma_\lin(D,B)\to \Gamma(\CC^*)$ becomes Mackenzie's \emph{warp} of two linear sections \cite{mac:jac}. 
\end{remark}
Using these geometric interpretations, the three inclusion maps \eqref{eq:3inclusions}
are realized as the bilinear maps
\begin{align}i_{\wh{\CC}}\colon\ \ 
C^\infty(D)_{[1,0]}\times C^\infty(D)_{[0,1]}\to C^\infty(D)_{[1,1]},&\ \ \ \ (f,g)\mapsto fg\nonumber \\
i_{\wh{A}}\colon \ \ C^\infty(D)_{[0,1]}\times \mf{X}(D)_{[-1,-1]}\to \mf{X}(D)_{[-1,0]},&\ \  \ \ (g,Z)\mapsto gZ\label{eq:3is} \\i_{\wh{B}}\colon \ \ 
\mf{X}(D)_{[-1,-1]} \times 
C^\infty(D)_{[1,0]}\to \mf{X}(D)_{[0,-1]},&\ \ \ \ (Z,f)\mapsto fZ \nonumber 
\end{align}
while the three pairings 
\eqref{eq:three_pairings} are 
\begin{align}\l\cdot,\cdot\r_{\CC^*}\ \ \mf{X}(D)_{[0,-1]}\times
\mf{X}(D)_{[-1,0]} \to \mf{X}(D)_{[-1,-1]},&\ \ \ \  (X, Y)\mapsto \Lie{X,Y}\nonumber \\ \l\cdot,\cdot\r_{A^*}\ \
C^\infty(D)_{[1,1]}\times 
\mf{X}(D)_{[0,-1]}\to C^\infty(D)_{[1,0]},&\ \ \ \ 
(h, X)\mapsto -\L_X h\label{eq:3pair}\\ \l\cdot,\cdot\r_{B^*}\ \ 
\mf{X}(D)_{[-1,0]} \times C^\infty(D)_{[1,1]}    \to C^\infty(D)_{[0,1]},&\ \ \ \
(Y, h)\mapsto \L_Y h.\nonumber 
\end{align}
The identity \eqref{eq:identityc} just amounts to $\L_{[X,Y]}= [\L_X,\L_Y]$. \smallskip
To verify \eqref{eq:3is} and \eqref{eq:3pair}, it is enough to consider the double vector space $D_0$, but there it follows by a routine check from the definitions.

\subsection{Applications to vector bundles I}\label{subsec:TV1}
We illustrate the concepts above with  various double vector bundles associated to a vector bundle $V\to M$. In this context, we will encounter the \emph{jet bundle} $J^1(V)$ and the \emph{Atiyah algebroid} $\on{At}(V)$ (often denoted by $\mf{D}(V)$, or similar). For $\sigma\in \Gamma(V)$, we denote by $j^1(\sigma)\in \Gamma(J^1(V))$ its jet prolongation.
The jet bundle comes with a quotient map 
$J^1(V)\to V$ taking sections of the form $f j^1(\sigma)$ to $f\sigma$; this defines a
short exact sequence 
\begin{equation}\label{eq:jet} 
0\to T^*M\otimes V\stackrel{i_{J^1(V)}}{\lra}  J^1(V)\to V\to 0
\end{equation}
with $i_{J^1(V)}(\d f\otimes \sigma)= j^1(f\sigma)-f j_1(\sigma)$. On the other hand,
the Atiyah algebroid comes with a short exact sequence 
\begin{equation}\label{eq:atiyah} 
0\to V\otimes V^*\stackrel{i_{\on{At}(V)}}{\lra} \on{At}(V)\stackrel{\a}{\lra} TM\to 0
\end{equation}
where $\a$ is the anchor. We shall find it convenient to use the identification 
$\Gamma(\on{At}(V))\cong \mf{X}_\lin(V)$ (cf.~ \eqref{eq:2ident}) to interpret sections $\delta$ of the Atiyah algebroid in terms of the corresponding linear vector field $\wt{\a}(\delta)$ on $V$; its restriction to the zero section is $\a(\delta)$. From this perspective, 
$\wt{\a}\big(i_{\on{At}(V)}(\sigma\otimes \tau)\big)=\phi_\tau\sigma^\sharp$, where 
$\phi_\tau\in C^\infty(V)$ is the linear function defined by 
$\tau\in \Gamma(V^*)$, and $\sigma^\sharp\in\mf{X}(V)_{[-1]}$ denotes the vertical lift of $\sigma\in \Gamma(V)$. 
The representation of $\on{At}(V)$ on $V$ is given by the Lie bracket, 
$(\nabla_\delta \sigma)^\sharp=\Lie{\wt{\a}(\delta),\sigma^\sharp}$, and the dual 
 representation $\nabla^*$ on $V^*$, defined as
\begin{equation}\label{eq:star}\L_{\a(\delta)}\l\tau,\sigma\r=\l\tau,\nabla_\delta \sigma\r+
\l \nabla^*_\delta \tau,\,\sigma\r,\end{equation}
is realized by the Lie derivative of $\wt{\a}(\delta)$ on linear functions, $\phi_\tau\mapsto 
\L_{\wt{\a}(\delta)}\phi_\tau$.

\subsubsection{Tangent bundle of $V$}
For $D=TV$ we have  $ A=V,\ B=TM,\ \CC=V^*$. One finds that 
\begin{equation}\label{eq:aha}
\wh{A}=J^1(V),\ \ \wh{B}=\on{At}(V),\ \ \wh{\CC}=J^1(V^*).
\end{equation}
In terms of $\Gamma(\wh{A})\cong \mf{X}(D)_{[-1,0]},\ \Gamma(\wh{B})\cong \mf{X}(D)_{[0,-1]},\ 
\Gamma(\wh{\CC})\cong C^\infty(D)_{[1,1]}$, these identifications  are given by 
\[ j^1(\sigma)\mapsto (\sigma^\sharp)_T,\ \ 
\delta\mapsto \wt{\a}(\delta)^\sharp,\ \ 
j^1(\tau)\mapsto 
(\phi_\tau)_T,\]
where 
$X\mapsto X_T$ is the \emph{tangent lift} of a (multi-)vector field.\footnote{The vertical lift 
of vector fields extends to an algebra morphism on multi-vector fields. On the other hand, the usual tangent lift $X\mapsto X_T$ of vector fields extends uniquely to a linear map on multivector fields, in such a way that $(X\wedge Y)_T=X_T\wedge Y^\sharp+X^\sharp 
\wedge Y_T$. One has the Schouten bracket relations 
$\Lie{X_T,Y_T}=\Lie{X,Y}_T,\ \Lie{X_T,Y^\sharp}=\Lie{X,Y}^\sharp,\ \Lie{X^\sharp,Y^\sharp}=0$. See e.g. \cite{cou:lie,gra:tan}.}
Using  \eqref{eq:3is} and \eqref{eq:3pair}, we obtain the three inclusions
\[ i_{\wh{\CC}}(\tau\otimes \d f)=i_{J^1(V^*)}(\d f\otimes \tau),\ \ \ 
i_{\wh{A}}(\d f\otimes \sigma)=i_{J^1(V)}(\d f\otimes \sigma),\ \ \ 
i_{\wh{B}}(\sigma\otimes \tau)=i_{\on{At}(V)}(\sigma\otimes \tau)
\]
and the three pairings 
\[ \l \delta,\, j^1(\sigma)\r_{V}=\nabla_\delta \sigma,\ \ \ 
\l j^1(\tau),\, \delta\r_{V^*}=-\nabla^*_\delta \tau,\ \ \ 
\l j^1(\sigma),\ j^1(\tau)\r_{T^*M}=\d \l\tau,\sigma\r
\]
for $\sigma\in \Gamma(V),\ \tau\in \Gamma(V^*),\ \delta\in \Gamma(\on{At}(V))$. 
As a sample computation, note that \eqref{eq:3is} gives $i_{\wh{\CC}}(\tau\otimes \d f)=(\phi_\tau)^\sharp f_T =(f\phi_\tau)_T-f^\sharp (\phi_\tau)_T\in C^\infty(TV)_{[1,1]}$ (here the vertical lift $f^\sharp$ of a function is simply the pullback, while the vertical lift $f_T$ is the exterior differential, regarded as a function on the tangent bundle). This coincides with the image of 
$i_{J^1(V^*)}(\d f\otimes \tau)=j^1(f\tau)-f j^1(\tau)$. 
The  exact sequences \eqref{eq:hatCexactsequence} are just the standard exact sequences for the jet bundles and the Atiyah algebroid.

\begin{remark} The $\CC^*=V$-valued pairing between $\wh{A}=J^1(V)$ and 
$\wh{B}=\on{At}(V)$ was observed by Chen-Liu in \cite[Section 2]{che:omn}. \end{remark}


Let us also note that by Remark \ref{rem:splittings}, a splitting of $D=TV$ is equivalent to a splitting of 
any one of the exact sequences for $J^1(V^*),\ J^1(V)$ or $\on{At}(V)$; in turn, these are equivalent to a 
linear connection on the vector bundle $V$.

%

\subsubsection{Cotangent bundle of $V$}
We will use the following notations for 
cotangent bundles $T^*Q$. Given $X\in \mf{X}(Q)$, let $\phi_X\in C^\infty(T^*Q)$ be the corresponding linear function, defined by the pairing with covectors. The \emph{standard Poisson structure} on $T^*Q$ is described by the condition 
$\{\phi_{X},\phi_{Y}\}=\phi_{[X,Y]}$ for any two such vector fields. For any $H\in C^\infty(T^*Q)$, 
the derivation $\{H,\cdot\}$ is its \emph{Hamiltonian vector field}; in particular,  
$X_{T^*}=\{\phi_X,\cdot\}$ is  the \emph{cotangent lift} of the vector field $X$. One has the identity  
$(fX)_{T^*}=f X_{T^*}-\phi_X (\d f)^\sharp$ for $f\in C^\infty(Q)$. 


For $D=T^*V$, we have that $A=V,\ B=V^*,\ \CC=TM$ with
\[  \wh{A}=J^1(V),\ \ \wh{B}=J^1(V^*),\ \  \wh{\CC}=\on{At}(V).\]   
In terms of the identifications of their spaces of sections with 
$\mf{X}(D)_{[-1,0]},\ \mf{X}(D)_{[0,-1]},\ C^\infty(D)_{[1,1]}$, these isomorphisms  are given by 
\[ j^1(\sigma)\mapsto \{\phi_{\sigma^\sharp},\cdot\},\ \ 
j^1(\tau)\mapsto \{(\phi_\tau)^\sharp,\cdot\},\ \ 
\delta\mapsto \phi_\delta.\]
Using \eqref{eq:3is} and \eqref{eq:3pair}, we find that the three pairings are the same 
as for $TV$ (with the order of the two entries interchanged), while each of the three inclusion maps 
changes sign. This is consistent with  $T^*V=\on{flip}(TV)^-$, see Proposition \ref{prop:d-}. The three exact sequences \eqref{eq:hatCexactsequence} are the standard exact sequences for the jet bundles and the Atiyah algebroid, up to a sign change of the three inclusion maps.

\section{The Weil algebra $\W(D)$}\label{sec:weilalgebra}
Throughout this section, we consider a fixed double vector bundle $D$ over $M$, with side bundles $A,B$ and with $\CC=\core(D)^*$. It will be convenient to regard $D$ as an associated bundle $(P\times D_0)/\Aut(D_0)$ with $D_0=A_0\times B_0\times \CC_0^*$. 
\subsection{Overview}
Recall (e.g., \cite{cat:sup})
that a graded supermanifold $\M$ is given by a base manifold $M$ together with a \emph{structure sheaf} of algebras, admitting local trivializations $C^\infty(U)\times \vee \mathsf{E}$ for $U\subset M$. Here $\mathsf{E}$ is a $\Z$-graded vector space, and the symmetric algebra is defined using the super-sign conventions. 
One formally thinks of global sections of the structure sheaf as the algebra of functions $C^\infty(\M)$  
on the `space' $\M$, even though the latter is not an actual topological space. If $V\to M$ is a vector bundle, the algebra 
$\Gamma(\wedge V^*)$ is regarded as the algebra of functions on the graded manifold $V[1]$, where the 
$[1]$ signifies a degree shift: Here $\mathsf{E}=(\R^k[1])^*$ with $k=\on{rank}(V)$. 

The definition of graded supermanifold has a straightforward generalization to \emph{bigraded} supermanifold.  Given a 
double vector bundle $D$, there is a bigraded supermanifold $D[1,1]$ (with a parity shift for both vector bundle directions), defined in terms of a structure sheaf.  In Voronov's work  \cite{vor:q}, the structure sheaf was obtained from local coordinates on $D$, compatible with a given splitting, by declaring coordinates for the two side bundle directions  to be odd. The corresponding bigraded vector space is the direct sum  $\mathsf{E}=A_0^*\oplus B_0^*\oplus \CC_0$, where the three summands reside in bidegrees $(1,0),\ (0,1),\ (1,1)$ respectively. In particular, one obtains an algebra of functions $C^\infty(D[1,1])$ as the global sections of the structure sheaf. One goal of this section is to give a  coordinate-free description of this algebra, independent of a choice of splittings.  We will call it a \emph{Weil algebra} since it generalizes the Weil algebra \cite{we:oe1} from equivariant cohomology. 


As pointed out to us by a referee, the work of Grabowski-J\'{o}\'{z}wikowski-Rotkievicz \cite{gra:dua2} suggests
analogous constructions in more general contexts.\footnote{Note however that the `Weil algebra bundles' in \cite{gra:dua2} involve a different notion of Weil algebra, as in \cite[Chapter 8]{kol:nat}.} 
\medskip

\subsection{The algebra of double-polynomial functions}
A smooth function on a double vector bundle $D$ will be called a (homogeneous) \emph{double-polynomial of bidegree $(r,s)$} if it is homogeneous of degree $r$ for the horizontal scalar multiplication, and of degree $s$ for the vertical scalar multiplication. The space of such functions is 
denoted $\S^{r,s}(D)=C^\infty(D)_{[r,s]}$; their direct sum over all $r,s\ge 0$ is denoted $\S(D)$.  

\begin{lemma}
The space $\S^{r,s}(D)$ of  double-polynomial functions on $D$ of bidegree $(r,s)$ is the space of sections of a vector bundle 
\[ S^{r,s}(D)\to M.\]
\end{lemma}
\begin{proof}
The space $S^{r,s}(D_0)$ of double-polynomial functions of bidegree $(r,s)$ on the double vector space $D_0$ is a subspace of the space of ordinary polynomials of degree at most $r+s$, and in particular is finite-dimensional.  Clearly, this space is $\Aut(D_0)$-invariant, and the  sections of the vector bundle $S^{r,s}(D)=(P\times S^{r,s}(D_0))/\Aut(D_0)$ are the double-polynomial functions of bidegree $r,s$. 
\end{proof}
We hence obtain a bigraded algebra bundle $S(D)=\bigoplus_{r,s} S^{r,s}(D)$. 
By applying a similar construction to the double vector bundles $D'$ and $D''$, we also have the bigraded algebra bundles $S(D')\to M$ and $S(D'')\to M$. 

\begin{proposition}\label{prop:sym}
The algebra bundle $S(D)$ is the bundle of bigraded commutative algebras 
\[ S(D)=(\vee A^*\otimes\vee B^*\otimes \vee \wh{\CC})/\sim\]
where the generators $\alpha\in A^*,\, \beta \in B^*,\, \wh{\cc}\in \wh{\CC}$ have bidegrees $(1,0),\ (0,1),\ (1,1)$, respectively. Here the kernel of the quotient map is the ideal generated by elements of the form $\alpha\beta-i_{\wh{\CC}}(\alpha\otimes\beta)$ for $(\alpha,\beta)\in A^*\times_M B^*$.
\end{proposition}
\begin{proof}
The claim is immediate for the double vector space 
$D_0=A_0\times B_0\times \CC_0^*$; the general case  follows 
 by the associated bundle construction. 	
\end{proof}
The construction of $S(D)$ is functorial: a $\DVB$ morphism  $D_1\to D_2$,  with base map $F\colon M_1\to M_2$, defines algebra morphisms $S(D_2)_{F(m)}\to S(D_1)_m$, hence a
comorphism of bigraded algebra bundles \[ S(D_1)\da S(D_2).\] The induced pullback on sections is simply the  pullback of double-polynomial functions. If $M_1=M_2=M$, with base map the identity, then the comorphism of algebra bundles may be seen as an ordinary morphism of algebra bundles $S(D_2)\to S(D_1)$. 
For example, the presentation of $S(D)$ as a quotient of $S(\wh{D})=\vee A^*\otimes \vee B^*\otimes \vee \wh{\CC}$ is functorially  
 induced by the inclusion $D\hra \wh{D}$.

\begin{remark}Just as vector bundles may be recovered from their algebra of polynomial functions as a character spectrum $V\cong \Hom_{\on{alg}}(\S(V),\R)$, double vector bundles are recovered as  $D\cong \Hom_{\on{alg}}(\S(D),\R)$. 
\end{remark}

\subsection{Definition and basic properties of $W(D)$}
The Weil algebra bundle is obtained from the description of 
$S(D)$, given in Proposition \ref{prop:sym}, by replacing commutativity with super-commutativity:
\begin{definition}
The \emph{Weil algebra bundle} $W(D)$ is the bundle of bigraded super-commutative algebras 
given as 
\[ W(D)=(\wedge A^*\otimes \wedge B^*\otimes \vee \wh{\CC})/\sim,\]
 where the generators  $\alpha\in A^*,\ \beta\in B^*,\ \wh{\cc}\in \wh{\CC}$ have bidegrees $(1,0),\ (0,1),\ (1,1)$ respectively. Here  the kernel of  
the  quotient map is the ideal generated by elements of the form \[ \alpha\beta-i_{\wh{\CC}}(\alpha\otimes \beta)\] 
with $(\alpha,\beta)\in A^*\times_MB^*$. 
\end{definition}
In the definition above, $\otimes$ denotes the usual tensor product of superalgebras.  
For $x\in W^{p,q}(D)$ we write 
\[ |x|=p+q\] 
for the total degree; thus super-commutativity means 
$x_1 x_2=(-1)^{|x_1||x_2|} x_2 x_1$. The bigraded algebra of sections $\W(D)=\Gamma(W(D))$ is called the \emph{Weil algebra} of $D$. In super-geometric terms, it is the algebra of smooth functions on the supermanifold $D[1,1]$. Similar to the construction of $S(D)$, a $\DVB$ morphism $D_1\to D_2$
with base map $F\colon M_1\to M_2$  induces a comorphism of bigraded superalgebra bundles $W(D_1)\da W(D_2)$, hence a morphism of bigraded superalgebras $\W(D_2)\to \W(D_1)$. 
\medskip

The definition gives a number of straightforward properties of $W(D)$:
\begin{enumerate}

\item In degree $p\le 1,\, q\le 1$, $W^{p,q}(D)$ coincides with $S^{p,q}(D)$:
\[ W^{0,0,}(D)=M,\ \ \ \ W^{1,0}(D)=A^*,\ \ \ \ W^{0,1}(D)=B^*,\ \ \ \ W^{1,1}(D)=\wh{\CC}.\]
\item 
A choice of splitting $D\cong A\times_M B\times_M \CC^*$  gives an algebra bundle isomorphism 
\[ W(D)\cong \wedge A^*\otimes \wedge B^*\otimes \vee \CC.\]

\item 
$W(D)=(P\times W(D_0))/\Aut(D_0)$. Since $W(D_0)=\wedge A_0^*\otimes \wedge B_0^*\otimes 
\vee \CC_0$, this may be used as an alternative definition of $W(D)$. 

\end{enumerate}
Replacing $D$ with $D'$ and $D''$, we have three bigraded algebra bundles $W(D),W(D'),W(D'')$ over $M$, where the roles of $A,B$, and $\CC$ are cyclically permuted. In particular, 
\[ W^{1,1}(D)=\wh{\CC},\ \ \ \ W^{1,1}(D')=\wh{A},\ \ \ \ W^{1,1}(D'')=\wh{B}.\]
The pairings \eqref{eq:three_pairings} between these bundles extend to 
\begin{align}
\l\cdot,\cdot\r_{\CC^*}&\colon 
W^{p,1}(D'')\times_M W^{1,q}(D')\to  \wedge^{p+q-1} \CC^*, \nonumber \\ 
\l\cdot,\cdot\r_{A^*}&\colon W^{p,1}(D)\times_M
W^{1,q}(D'') \to  \wedge^{p+q-1} A^*, \label{eq:three_parings2}\\
\l\cdot,\cdot\r_{B^*}&\colon 
W^{p,1}(D')\times_M W^{1,q}(D)\to\  \wedge^{p+q-1} B^*. \nonumber
\end{align}
Here $ \l\cdot,\cdot\r_{\CC^*}$ is the unique extension of the given pairing such that 
\[ \l \alpha,\,\wh{a}\r_{\CC^*}= -\alpha(a),\ \ \l\wh{b},\,\beta\r_{\CC^*}=\beta(b) \]
for the cases $p=0,q=1$ and $p=1,q=0$, and such that the following bilinearity property holds:   
\[ \l\lambda x,\,y\r_{\CC^*}=\lambda \l x,y\r_{\CC^*},\ \  
\l x,\,y\lambda\r_{\CC^*}=\l x,y\r_{\CC^*}\lambda\]
for $\lambda\in \wedge \CC^*
,\ x\in W^{\bullet,1}(D'')
,\ y\in  W^{1,\bullet}(D')$ (with the same base points). 
 The discussion for the pairings $\l\cdot,\cdot\r_{A^*},\ \l\cdot,\cdot\r_{B^*}$ is similar. 
In Section \ref{sec:wedge}, we will give geometric interpretations of these pairings.  

\begin{remark}
The description of the Weil algebra bundle for $D^-$ is obtained from that for $D$ 
by replacing the sign of the inclusion map $i_{\wh{\CC}}$. That is, $W(D^-)$ has the same generators, 
but the defining relation becomes 
$\alpha\beta=-i_{\wh{\CC}}(\alpha\otimes\beta)$. The map on generators $\alpha\mapsto \alpha,\ 
\beta\mapsto \beta,\ \wh{\cc}\mapsto -\wh{\cc}$ extends to an isomorphism of algebra bundles 
$W(D^-)\to W(D)$. 
\end{remark}

\subsection{Derivations}\label{subsec:contractions}
For a vector bundle $V\to M$, the graded bundle $\on{Der}(\wedge V^*)$ of \emph{fiberwise} superderivations of $\wedge V^*$ is the free $\wedge V^*$-module generated by contractions. Thus  
\[ \on{Der}(\wedge V^*)=\wedge V^*\otimes V\]
as a bundle of graded super-Lie algebras, where the elements $1\otimes v$ have degree $-1$. 

Given a double vector bundle $D$, we are interested in the structure of  the bigraded bundle 
 \[ \on{Der}(W(D))=\bigoplus_{r,s} \on{Der}^{r,s}(W(D)).\]
Here  $\on{Der}^{r,s}(W(D))\to M$ is the bundle of fiberwise superderivations of bidegree $(r,s)$ 
of the algebra bundle $W(D)\to M$: its space of sections 
consists of bundle maps $\delta\colon 
W(D)\to W(D)$ of bidegree $(r,s)$ with the superderivation property 
\[ \delta(xy)=\delta(x)y+(-1)^{|\delta||x|}x\delta (y)\]
for homogeneous elements $x,y$, where $|\delta|=r+s$ and $|x|$ are the total degrees of $\delta$ and $x$.  
The following result describes the structure of $\on{Der}(W(D))$ as a $W(D)$-module and as a bundle of graded Lie algebras. 

\begin{theorem}
Let $m\in M$ and $\wh{a}\in \wh{A}_m$,  $\wh{b}\in \wh{B}_m$, and $\ggamma\in \CC^*_m$.  There are 
unique \emph{contraction operators}
\[ \iota_h(\wh{a})\in \on{Der}^{-1,0}(W(D))_m,\ \ 
\iota_v(\wh{b})\in \on{Der}^{0,-1}(W(D))_m,\ \ \ 
\iota(\ggamma)\in \on{Der}^{-1,-1}(W(D))_m\]
such that 
\begin{equation} \label{eq:contractions4}
\iota_h(\wh{a})v=\l \wh{a},v\r_{B^*},\ \ \ 
\iota_v(\wh{b})u=(-1)^{|u|}\l u,\wh{b}\r_{A^*},\ \ \ 
\iota(\ggamma)\cc=-\ggamma(\cc)\end{equation}
for all $u\in W^{1,\bullet}(D)_m,\ \ v\in W^{\bullet,1}(D)_m,\ \ \cc\in W^{1,1}(D)_m=\wh{\CC}_m$. The contraction operators satisfy the commutation relations
\begin{equation}\label{eq:commrel}
[\iota_v(\wh{b}),\,\iota_h(\wh{a})]=-\iota\big(\l\wh{b},\wh{a}\r_{\CC^*}\big),\ \ \wh{a}\in\wh{A},\ \wh{b}\in \wh{B},
\end{equation}
while all other commutations of contractions are zero.  The $W(D)_m$-module  $\on{Der}(W(D))_m$ is generated by the three types of  contraction operators, subject to the relations  
\begin{equation}
\iota_h\big(i_{\wh{A}}(\beta\otimes \ggamma)\big)=\beta\iota(\ggamma),\ \ \ \ 
\iota_v\big(i_{\wh{B}}(\ggamma\otimes \alpha)\big)=-\alpha\iota(\ggamma),\ \ \ \ \alpha\in A^*_m,\ \beta\in B^*_m,\ \ggamma\in \CC^*_m.
\end{equation}
\end{theorem}
\begin{proof}
For degree reasons, the proposed expressions for the contractions determine the formulas on generators 
of $W(D)_m$. Specifically,  $\iota_h(\wh{a})$ is given on  generators $\alpha\in A^*_m,\ \beta\in B^*_m,\ \wh{\cc}\in \wh{\CC}_m$ 
by
$\alpha\mapsto \alpha(a),\ \beta\mapsto 0,\ \wh{\cc}\mapsto \l \wh{a},\wh{\cc}\r_{B^*}$, while 
$\iota_v(\wh{b})$ is given by $\alpha\mapsto 0,\ \beta\mapsto \beta(b),\ \wh{\cc}\mapsto \l\wh{\cc},\wh{b}\r_{A^*}$, and $\iota_h(\ggamma)$ is given by $\alpha\mapsto 0,\ \beta\mapsto 0,\ \wh{\cc}\mapsto -\ggamma(\cc)$. One readily checks that these formulas are compatible with the defining relation of the Weil algebra, and hence extend to a derivation on all of $W(D)_m$. Furthermore, the super-commutation relation between these contractions operators are verified by evaluating on generators. 

The three types of contraction operators define a $W(D)_m$-module morphism 
\begin{equation}\label{eq:quotmap}
W(D)_m\otimes (\wh{A}_m\oplus \wh{B}_m\oplus \CC^*_m)\to \on{Der}(W(D))_m\end{equation} 
whose kernel contains elements of the form 
\begin{equation}\label{eq:contains}
1\otimes i_{\wh{A}}(\beta\otimes \ggamma)-\beta\otimes \ggamma,\ \ \ \ 
1\otimes i_{\wh{B}}(\ggamma\otimes \alpha)+\alpha\otimes \ggamma\end{equation}
with $\alpha\in A^*_m,\ \beta\in B^*_m,\ \ggamma\in \CC^*_m$.  We have to show that \eqref{eq:quotmap} is surjective, with kernel the submodule generated by elements of the form
\eqref{eq:contains}. 

It suffices to prove this for the double vector space $D_0=A_0\times B_0\times \CC_0^*$. Here $W(D_0)$ is simply a tensor product $\wedge A_0^*\otimes \wedge B_0^*\otimes \vee \CC_0$, and hence $\on{Der}(W(D_0))=W(D_0)\otimes (A_0\oplus B_0\oplus \CC_0^*)$. 
Since
$\wh{A}_0=A_0\oplus (B_0^*\otimes \CC_0^*)$ and $\wh{B}_0=B_0\oplus (\CC_0^*\otimes A_0^*)$, it is immediate that the module map \eqref{eq:quotmap} (with $D$ replaced by $D_0$) 
is surjective. Its kernel contains elements of the form \eqref{eq:contains}; hence it also contains the 
$W(D_0)$-submodule generated by elements of this form. But this submodule is a 
complement to the submodule $W(D_0)\otimes (A_0\oplus B_0\oplus \CC_0^*)$, and is therefore the entire kernel
of \eqref{eq:quotmap}. 
\end{proof}

In particular, we see that the bundle 
$\on{Der}^{r,s}(W(D))$ is zero if $r<-1$ or $s<-1$,  while
\begin{equation}\label{eq:corcontractions}
 \on{Der}^{-1,0}(W(D))=\wh{A},\ \ \ \ 
\on{Der}^{0,-1}(W(D))=\wh{B},\ \ \ \ 
 \on{Der}^{-1,-1}(W(D))=\CC^*.
\end{equation}

\begin{proposition}\label{prop:contractionpairing}
The horizontal contractions extend to an isomorphism of left $\wedge B^*$-modules 
\begin{equation}\label{eq:iotah} \iota_h\colon W^{\bullet,1}(D')\to \on{Der}^{-1,-1+\bullet}(W(D)),\ x\mapsto \iota_h(x)\end{equation}
such that 
\begin{equation}\label{eq:da1} \iota_h(x)z=\l x,z\r_{B^*},\ \ \ 
x\in  W^{\bullet,1}(D'),\ z\in W^{1,\bullet}(D)
.\end{equation}
The vertical contractions extend to an isomorphism of left $\wedge A^*$-modules
\begin{equation}\label{eq:iotav} \iota_v\colon W^{1,\bullet}(D'')\to \on{Der}^{-1+\bullet,-1}(W(D)),\ y\mapsto \iota_v(y)\end{equation}
given by 
\begin{equation}\label{eq:da2}
\iota_v(y)z=-(-1)^{(|y|+1)(|z|+1)}\,\l z,y\r_{A^*},\ \ \ y\in W^{1,\bullet}(D''),\ \ z\in W^{\bullet,1}(D)\end{equation}
\end{proposition}
The sign in \eqref{eq:iotav} comes from the fact that 
we are using the left $\wedge A^*$-module structures, whereas the pairing is 
bilinear for the right $\wedge A^*$-module structure in the second argument. 
Note that $\iota_h(\ggamma)=\iota_v(\ggamma)=\iota(\ggamma)$ for $\ggamma\in \CC^*$.

\begin{proof}
The proposed formulas determine $\iota_h(x),\ \iota_v(y)$ on generators. To show that 
the formula  \eqref{eq:da2} for $\iota_v(y)$ gives a well-defined $\wedge A^*$-module homomorphism,  it suffices to check that the right hand side 
is  linear in the argument $y$ for the left $\wedge A^*$-module structure and linear in the argument $z$ for the right $\wedge A^*$-module structure. 
Indeed, replacing $y$ with $\alpha y$ changes the 
right hand side to 
 \[ (-1)^{|y|(|z|+1)}\l z,\,\alpha y\r_{A^*}
 =(-1)^{|y|\,|z|}\l z,\,y\r_{A^*}\alpha 
 =(-1)^{(|y|+1)(|z|+1)}\alpha\l z,\,y\r_{A^*}.\]
Similarly, 
replacing 
$z$ with $z\alpha$ for $\alpha\in A^*$ changes the right hand side to 
\[ (-1)^{(|y|+1)|z|}\l z\alpha,y\r_{A^*}
=(-1)^{|y|\,|z|}\alpha \l z,\,y\r_{A^*}
=(-1)^{(|y|+1)(|z|+1)}\l z,\,y\r_{A^*}\alpha\]
as required. The argument for $\iota_h(x)$ is similar. 
 \end{proof}

\subsection{Applications to vector bundles II}\label{subsec:TV2}
Continuing the discussion from Section \ref{subsec:TV1}, we have the following description of the Weil algebras and contraction operators for the tangent bundles and cotangent bundles of vector bundles $V\to M$.
\subsubsection{Tangent bundle of $V$}
Consider $D=TV$, so that $A=V,\ B=TM,\ C=V^*$. 
The Weil algebra $\W(TV)$ is generated by functions  $f\in C^\infty(M)$ (bidegree $(0,0)$) and 
their de Rham differentials $\d f$  (bidegree 
$(0,1)$), together with sections $\tau\in \Gamma(V^*)$ (bidegree $(1,0)$) and their 1-jets 
$j^1(\tau)\in \Gamma(J^1(V^*))$ (bidegree $(1,1)$), subject to relations of $C^\infty(M)$-linearity and the relation that 
\[ \tau\ \d f=j^1(f\tau)-f j^1(\tau).\]
Here we used that $i_{\wh{\CC}}(\tau\otimes \d f)=i_{J^1(V^*)}(\tau\otimes \d f)$. 
The contraction operators are computed from the pairings, for example: 
\[ \iota_v(\delta) j^1(\tau)=\l j^1(\tau),\delta\r_{V^*}=-\nabla_\delta^*\tau.\] 

%


\subsubsection{Cotangent bundle of $V$}
Consider $D=T^*V$, so that $A=V,\ B=V^*,\ C=TM$. 
The Weil algebra $\W(T^*V)$ is generated by functions  $f\in C^\infty(M)$ (bidegree $(0,0)$), sections $\sigma\in \Gamma(V)$ (bidegree $(0,1)$), sections $\tau\in \Gamma(V^*)$ (bidegree $(1,0)$),  and infinitesimal automorphisms $\delta\in \Gamma(\on{At}(V))$ (bidegree $(1,1)$), subject to relations of $C^\infty(M)$-linearity  and the relation that  
the product $\tau \sigma$ in the Weil algebra, for $\tau\in \Gamma(V^*)$ and $\sigma\in \Gamma(V)$, equals the section $i_{\wh{\CC}}(\tau\otimes \sigma)=
-i_{\on{At}(V)}(\sigma\otimes \tau)=-\phi_\tau\,\sigma^\sharp
\in \Gamma(\on{At}(V))$. Again, the various contraction operators may be computed from the pairings.


\section{Linear and core sections of $\wedge_A D$}\label{sec:wedge}
We have already encountered the linear and core sections of a double vector bundle $D$ over its side bundles. We shall now consider the generalization to the exterior algebra bundles, and relate it to the Weil algebra bundles. Throughout, $D$ will denote a double vector bundle with sides $A,B$ and with 
$\core(D)=\CC^*$.

\subsection{Linear and core sections of $\wedge^\bullet_A D\to A $}\label{subsec:lincor}
Given a double vector bundle $D$, we denote by 
\[ \wedge^n_A D\to A\] 
its exterior powers as a vector bundle over $A$. The horizontal scalar multiplications $\kappa_t^h\colon D\to D$ are vector bundle endomorphisms of $D\to A$,
hence they extend to algebra bundle endomorphisms $\wedge^\bullet\kappa_t^h$ of 
$\wedge^\bullet_A D\to A$. 
A section $\sigma\colon A\to \wedge^n_A D$ is \emph{homogeneous of degree $k$} if it satisfies 
\[ \sigma(\kappa_t^h(a))=t^k\ (\wedge^n\kappa_t^h)(\sigma(a))\]
for all $t\in\R$; the space of such sections is denoted $\Gamma(\wedge^n_A D, A)_{[k]}$. 


%
\begin{definition}
The spaces of \emph{core sections} and  \emph{linear sections} of $\wedge^n_A D$ over $A$ are defined as 
follows: 
\begin{align*}
\Gamma_\core(\wedge^n_A D, A)&= \Gamma(\wedge^n_A D, A)_{[-n]},\\
\Gamma_\lin(\wedge^n_A D, A)&= \Gamma(\wedge^n_A D, A)_{[-n+1]}.
\end{align*}
The spaces 
$\Gamma_\core(\wedge^n_B D, B)$ and $\Gamma_\lin(\wedge^n_B D, B)$
are defined similarly. 
\end{definition}
The core sections $\Gamma_\core(\wedge^\bullet_A D, A)$ are a super-commutative graded algebra under the wedge product, and $\Gamma_\lin(\wedge^\bullet_A D, A)$ is a graded module over this algebra.
The significance of these spaces is clarified by the following result. 
\begin{proposition}\label{prop:linses}
The space $\Gamma(\wedge^n_A D, A)_{[k]}$ is zero if $k<-n$, and for $k=-n$ is given by  
 \begin{equation}\label{eq:coreiso}
  \Gamma_\core(\wedge^n_A D, A)\cong \Gamma(\wedge^n \CC^*).
  \end{equation}
  The space of linear sections fits into a short exact sequence
\begin{equation}\label{eq:ses}
  0\to \Gamma(\wedge^n \CC^*\otimes A^*)\to \Gamma_\lin(\wedge^n_A D, A)\to \Gamma(\wedge^{n-1} \CC^*\otimes B)\to 0.
 \end{equation}
\end{proposition}
\begin{proof}
Recall that when $D$ is viewed as a vector bundle over $A$, its restriction to the submanifold $M$ is the direct sum $D=\CC^*\oplus B $. Hence, the restriction of sections to $M\subset A$ gives a map
\[ \Gamma(\wedge_A D,A)\to \Gamma(\wedge(\CC^*\oplus B))=\Gamma(\wedge \CC^*\otimes \wedge B).\]
 We claim that the restriction of core sections gives the isomorphism \eqref{eq:coreiso}, while restriction of linear sections gives a map from $\Gamma_\lin(\wedge^n_A D, A)$ onto $\Gamma(\wedge^{n-1} \CC^*
\otimes B)$, with kernel $\Gamma(\wedge^n \CC^*\otimes A^*)$ spanned by products of core sections with linear functions on the base $A$. Using the associated bundle construction, 
it suffices to prove these claims for the double vector space $D_0=A_0\times B_0\times \CC_0^*$. We have 
\[ \Gamma(\wedge^n_{A_0} D_0,A_0)=C^\infty(A_0,\bigoplus_{i+j=n} \wedge^j \CC_0^*\otimes  \wedge^i B_0).\]
The elements of $\wedge^{j} \CC_0^*\otimes\wedge^{n-j} B_0 \to M$ (regarded as constant sections of 
$\wedge^n_{A_0} D_0$) are homogeneous of degree $-j$. To obtain a section that is homogeneous of  degree $k$, we must multiply by a polynomial on $A_0$ of degree $k+j$. Thus, 
\[ \Gamma(\wedge^n_{A_0} D_0,A_0)_{[k]}=\bigoplus_j
 \wedge^{j} \CC_0^*\otimes 
\wedge^{n-j} B_0\otimes
\vee^{k+j} A_0^*\]
where the sum is over all $j$ with $0\le j\le n$ and $k+j\ge 0$. In particular, this space is zero if $k<-n$, 
and is equal to $\wedge^n \CC_0^*$ for $k=-n$. 
 Specializing to $k=-n+1$ this shows 
\begin{equation}\label{eq:two}
 \Gamma_\lin(\wedge^n_{A_0} D_0,A_0)=(\wedge^n \CC_0^*\otimes A_0^*) \oplus 
(\wedge^{n-1} \CC_0^*\otimes B_0).\end{equation}
Hence, the map $\Gamma_\lin(\wedge^n_{A_0} D_0, A_0)\to 
\wedge^{n-1} \CC_0^*\otimes
B_0 $ is surjective, with kernel  $\wedge^n \CC_0^*\otimes A_0^*$. 
\end{proof}
\medskip

\subsection{Interpretation in terms of Weil algebras}\label{subsec:wedgeweil}
The linear and core sections of $\wedge_A D\to A$ are graded subspaces of Weil algebras, as follows.

\begin{proposition}\label{prop:weilident}
There is a canonical isomorphism $\Gamma_\core(\wedge^\bullet_A D, A)\cong \W^{\bullet,0}(D'')=\Gamma(\wedge^\bullet \CC^*)$ as graded algebras, and an isomorphism of left  modules over this algebra, 
\[ \Gamma_\lin(\wedge^\bullet_A D, A)\cong \W^{\bullet,1}(D'').\]
Similarly, there is a canonical isomorphism of graded algebras, 
$\Gamma_\core(\wedge^\bullet_B D, B)\cong \W^{0,\bullet}(D')=\Gamma(\wedge^\bullet \CC^*)$ 
and an isomorphism of right modules over this algebra, 
\[ \Gamma_\lin(\wedge^\bullet_B D, B)\cong\W^{1,\bullet}(D').\]
\end{proposition}
\begin{proof}
It suffices to prove the claim for the 
double vector space $D_0=A_0\times B_0\times \CC_0^*$. But 
$W^{\bullet,0}(D_0'')=\wedge^\bullet \CC_0^*,\ \ \ W^{0,\bullet}(D_0')=\wedge^\bullet \CC_0^*$
as graded algebras. Furthermore, the isomorphism of graded left $\wedge \CC_0^*$-modules
\[ W^{\bullet,1}(D_0'')=(\wedge^\bullet \CC_0^*\otimes A_0^*)\oplus (\wedge^{\bullet-1} \CC_0^*\otimes B_0)\]
is exactly the description of linear sections of $\wedge^n_{A_0}D_0$, see 
\eqref{eq:two}. Similarly for 
\[ W^{1,\bullet}(D_0')=(B_0^*\otimes\wedge^\bullet \CC_0^*) \oplus (A_0\otimes \wedge^{\bullet-1}\CC_0^*).\]
as graded right $\wedge \CC_0^*$-modules. 
\end{proof}
 With these identifications, the pairing 
$\l\cdot,\cdot\r_{\CC^*}\colon 
W^{p,1}(D'')\times_M W^{1,q}(D') \to  \wedge^{p+q-1} \CC^*$ (cf.~ \eqref{eq:three_parings2})
translates into a $\Gamma(\wedge \CC^*)$-bilinear pairing 
\begin{equation}\label{eq:pairingtaketwo}
 		\l\cdot,\cdot\r_{\CC^*}\colon 
 		\Gamma_\lin(\wedge^p_A D,A)\times 
 		\Gamma_\lin(\wedge^q_B D,B)
 \to  \Gamma(\wedge^{p+q-1} \CC^*).\end{equation}

\subsection{Application to vector bundles III}\label{subsec:TV3}
Let $V\to M$ be a vector bundle. 
\subsubsection{Linear Multi-vector fields on vector bundles}\label{ex:corelin}\label{ex:TV}
Consider $TV$ as a double vector bundle as in Example \ref{subsec:examples}\ref{it:ex2}. The \emph{core $n$-vector fields}
$ \mf{X}^n_\core(V)\equiv \Gamma_\core(\wedge^n_V TV,V)$
are the sections of $\wedge^n V$, regarded as vertical `fiberwise constant' multi-vector fields on $V$:
\[ \mf{X}^n_\core(V)=\Gamma(\wedge^n V).\]
The \emph{linear $n$-vector fields}
$\mf{X}^n_\lin(V)= \Gamma_\lin(\wedge^n_V TV,V)$
may be defined by their property that the evaluation on linear 1-forms 
on $V$ is a linear function on $V$ \cite{gra:tan}. The short exact sequence \eqref{eq:ses} specializes to 
\begin{equation}\label{eq:exactsequenceX}
 0\to \Gamma(\wedge^{n}V\otimes V^*)\to \mf{X}^n_\lin(V)\to \Gamma(\wedge^{n-1}V
 \otimes TM)\to 0;
 \end{equation}
here the inclusion of $\Gamma(\wedge^{n}V\otimes V^*)$ is as the subspace of linear $n$-vector fields on $V$ that are tangent  to the fibers of $V\to M$. In local vector bundle coordinates, with $x_i$ the coordinates on the base and $y_j$ the coordinates on the fiber, the linear $n$-vector fields on $V$ are of the form 
\[ \sum a^j_{j_1\cdots j_n}(x)\, y_j\, \f{\p}{\p y_{j_1}}\wedge \cdots \wedge
\f{\p}{\p y_{j_n}}+\sum b_{i\, j_1\cdots j_{n-1}}(x)
\f{\p}{\p y_{j_1}}\wedge \cdots \wedge \f{\p}{\p y_{j_{n-1}}}\wedge \f{\p}{\p x_i} 
.\]
The Schouten bracket of multi-vector fields defines a graded Lie algebra structure on $\mf{X}_\lin^{\bullet}(V)[1]$, with a representation on $\mf{X}_\core^\bullet(V)=\Gamma(\wedge^\bullet V)$.
These are the \emph{multi-differentials} in the work of Iglesias-Ponte, Laurent-Gengoux, and Xu \cite{igl:uni}.

\subsubsection{Linear differential forms on vector bundles}
\label{ex:linforms}
Consider next the double vector bundle $T^*V$ from Example \ref{subsec:examples}\ref{it:ex2}. 
The space $\Omega^n_\core(V)=\Gamma_\core(\wedge^n_V T^*V,V)$ is just $\Omega^n(M)$, viewed as 
the space of basic $n$-forms on $V$ via pullback. The space 
\[ \Gamma_\lin(\wedge^n_V T^*V,V)=\Omega^n_\lin(V)\] 
of \emph{linear $n$-forms} on $V$ consists of n-forms $\alpha$ with $\kappa_t^*\alpha=t\alpha$ where $\kappa_t$ is scalar multiplication by $t$ on $V$. 
(Note that the homogeneity of $n$-forms on $V$ relative to pullback $\kappa_t^*$ 
is not the same as homogeneity as sections of $\wedge^n_V T^*V$ over $V$.) 
In local bundle bundle coordinates, with $x_i$ the coordinates on the base and $y_j$ the coordinates on the fiber, the 1-forms $\d x_i$, seen as local sections of $\wedge^\bullet_V T^*V$, have homogeneity degree $-1$ while the $\d y_j$ have homogeneity $0$.   A general linear $n$-form is locally given by an expression
\[ \sum a_{j\, i_1\cdots i_{n-1}}(x) \,\d x_{i_1}\wedge \cdots \wedge \d x_{i_{n-1}}\wedge \d y_j 
+\sum b_{j\, i_1\cdots i_n}(x) y_j  \d x_{i_1}\wedge \cdots \wedge \d x_{i_n}.\]
The short exact sequence \eqref{eq:ses} becomes 
\begin{equation}\label{eq:sesforms} 
0\to \Gamma(\wedge^n T^*M\otimes V^*) \to \Omega^n_\lin(V) \to \Gamma(
 \wedge^{n-1}T^*M\otimes
V^* ) \to 0;\end{equation}
here the inclusion of $\Gamma(\wedge^n T^*M\otimes V^*)$ is as the space of linear $n$-forms  on $V$ that are horizontal for the projection to $M$, while the projection to $\Gamma(\wedge^{n-1}T^*M
\otimes 
V^* ) $ is given by contraction with sections 
of $V$ (regarded as the space $\mf{X}_\core(V)$ of fiberwise constant vector fields on $V$). 
The exact sequence \eqref{eq:sesforms} has a canonical splitting \cite{bur:mul}: every element of $\Omega^n_\lin(V)$ 
decomposes uniquely as $\nu+\d \mu$ where $\nu\in \Gamma(\wedge^n T^*M\otimes V^*) $ 
and $\mu\in \Gamma(\wedge^{n-1} T^*M\otimes V^*)$.  Using the Mackenzie-Xu isomorphism \eqref{eq:mackenziexu} we obtain a similar interpretation 
\[ \Gamma_\lin(\wedge^q_{V^*}T^*V,V^*)=\Omega_\lin^q(V^*).\]
Equation \eqref{eq:pairingtaketwo} defines an $\Omega(M)$-bilinear
 pairing 
\begin{equation}\label{eq:formpairing} 
\l\cdot,\cdot\r_{T^*M}\colon \Omega_\lin^p(V)\times \Omega_\lin^q(V^*)\to \Omega^{p+q-1}(M),\end{equation}
(using the right $\Omega(M)$-module structure in the second argument), given in low degrees by 
\[\l\phi_\tau,\phi_\sigma\r_{T^*M}=0,\ \ \ 
 \l \phi_\tau,\d\phi_\sigma\r_{T^*M}=-\l\d \phi_\tau,\phi_\sigma\r_{T^*M}=\l\tau,\sigma\r,\ \ \ \ \l\d\phi_\tau,\d\phi_\sigma\r_{T^*M}=
 \d\l\tau,\sigma\r\]
for $\tau\in \Gamma(V^*),\ \sigma\in \Gamma(V)$ (with $\phi_\tau,\phi_\sigma$ the corresponding linear functions). 

\subsection{Multi-vector fields on $D$}\label{subsec:multivector}
The linear and core sections of $\wedge_A D\to A$ and $\wedge_B D\to B$, and their pairings,  have a simple interpretation in terms of the space $\mf{X}^\bullet(D)$ of multi-vector fields on  $D$. 
Using the discussion from Example \ref{ex:TV}, the sections of $\wedge^n_A D$ are identified with $n$-vector fields
on $D$ that are homogeneous of degree $-n$ with respect to the vertical vector bundle structure. A similar description applies to sections of $\wedge^n_B D$. As in Section \ref{subsec:geometricinterpretation}, we let $\mf{X}^n(D)_{[k,l]}$ denote the space 
of $n$-vector fields that are homogeneous of degree $k$ horizontally and of degree $l$ vertically. 
This space is trivial if $k<-n$ or $l<-n$, while $ \mf{X}^n(D)_{[-n,-n]}\cong \Gamma(\wedge^n \CC^*)$ is identified with $\Gamma_\core(\wedge^n_A D,A)$ and also with 
$\Gamma_\core(\wedge^n_B D,B)$. Furthermore, 
we have canonical isomorphisms  
\[ 
\Gamma_\lin(\wedge^p_A D,A)\cong \mf{X}^p(D)_{[1-p,-p]},\ \ \ \ \ \ 
\Gamma_\lin(\wedge^q_B D,B)\cong \mf{X}^q(D)_{[-q,1-q]}.\]
%
The first isomorphism is compatible with the left module structure over $\Gamma(\wedge \CC^*)$, the second  isomorphism with the right module structure, 
realized as wedge product of the corresponding multivector fields from the left or right, respectively.   

\begin{proposition}\label{prop:pairingschouten}
With the above identifications, the pairing 
\eqref{eq:pairingtaketwo} is  given by the Schouten bracket 
\[ \l x,y\r_{\CC^*}=\Lie{x,y}\] 
for all
$x\in \mf{X}^p(D)_{[1-p,-p]}$ and $y\in \mf{X}^q(D)_{[-q,1-q]}$.
%
%
\end{proposition}
\begin{proof}
The Schouten bracket of elements $\lambda\in \mf{X}^n(D)_{[-n,-n]}$
with 
$x\in \mf{X}^p(D)_{[1-p,-p]}$ or with 
$y\in \mf{X}^q(D)_{[-q,1-q]}$ is zero, 
for degree reasons. Hence, the derivation property of the Schouten bracket shows that  $\Lie{x,y}$  is 
$\Gamma(\wedge \CC^*)$-bilinear, for the left module structure on 
$\mf{X}^p(D)_{[1-p,-p]}$ and the right module structure on 
$\mf{X}^q(D)_{[-q,1-q]}$. The pairing $\l x,y\r_{\CC^*}$ has the same bilinearity property. It therefore suffices to prove the formula  for $p,q\le 1$. 
If $p=q=1$, we are dealing with the pairing of vector fields $X\in \mf{X}(D)_{[0,-1]}=\Gamma(\wh{B})$ and 
$Y\in\mf{X}(D)_{[-1,0]}=\Gamma(\wh{A})$, and the claim was already noted in Section \ref{subsec:geometricinterpretation}. If $p=0,\,q=1$ we have $x=\alpha\in \Gamma(A^*)
,\ y=\wh{a}\in \Gamma(\wh{A})$ with the pairing $\l\alpha,\wh{a}\r_{\CC^*}=-\alpha(a)$. 
After identification of $\wh{a}$ with a vector field $Y\in \mf{X}^1(D)_{[-1,0]}$ and $\alpha$ with a function $f\in C^\infty(D)_{[1,0]}$, this coincides with $\L_Y f=-\Lie{f,Y}$, as required. 
Similarly, for $p=1,\ q=0$ we have $x=\wh{b}\in \Gamma(\wh{B})$ and $y=\beta\in \Gamma(B^*)$, 
with pairing $\l\wh{b},\beta\r_{A^*}=\beta(b)$, which, after identification of $\wh{b}$ with a vector field 
$X\in\mf{X}^1(D)_{[0,-1]}$ and $\beta$ with a function $g\in C^\infty(D)_{[0,1]}$, coincides with 
$-\L_X g=-\Lie{X,g}$.  
\end{proof}


\section{Poisson double vector bundles}\label{sec:PDVB}

\subsection{Reminder on Poisson vector bundles}\label{subsec:reminder}
Given a vector bundle $p\colon V\to M$, one knows that the following structures are equivalent: 
\begin{itemize}
\item[(i)] a linear Poisson structure $\pi$ on $V\to M$, 
\item[(ii)] a degree $-1$ Poisson bracket  $\{\cdot,\cdot\}$ on the algebra of polynomial functions on $V$, 
\item[(iii)] a Lie algebroid structure  on the dual bundle, $V^*\Ra M$,
\item[(iv)] a degree $-1$ Gerstenhaber  bracket  (\emph{Schouten bracket}) on $\Gamma(\wedge V^*)$,
\item[(v)] a degree $1$ differential $\d_{CE}$ on $\Gamma(\wedge V)$.
\end{itemize}
Here (and from now on) we write $A\Ra M$ to indicate a Lie algebroid over $M$; the notation (which we learned from \cite{bur:vec}) suggests the differentiation of a Lie groupoid $G \rra M$ when source and target become `infinitesimally close'.
Let us briefly recall how these equivalences come about.  Given a linear Poisson tensor $\pi$ on $V$, the corresponding Poisson bracket $\{\cdot,\cdot\}$ on $C^\infty(V)$ restricts to a bracket on the space of polynomial functions on $V$, and 
is uniquely determined by this restriction. The Poisson bivector $\pi$ being linear is equivalent to the bracket of linear functions being again linear, thus to $\{\cdot,\cdot\}$ having degree $-1$. Hence,  (i)$\Leftrightarrow$ (ii). The Poisson bracket is in fact already determined by its restriction to linear and basic functions on $V$. Using the identification of  linear functions with sections $\sigma \in \Gamma(V^*)$,\footnote{We shall directly regard $\sigma$ as a linear function, rather than using our earlier notation $\phi_\sigma$.} this gives the 
 equivalence (ii)$\Leftrightarrow$(iii), where the Lie bracket and anchor are expressed by 
\begin{equation}\label{eq:homogeneity}
 \Lie{\sigma_1,\sigma_2}=\{\sigma_1,\sigma_2\},\ \ \ 
p^*(\L_{\a(\sigma)}f)=\{\sigma,p^*(f)\}.\end{equation}
This Lie algebroid bracket extends to a Schouten bracket on the algebra $\Gamma(\wedge V^*)$, with $\Lie{\sigma,p^*f}=p^*(\L_{\a(\sigma)}f)$ as the bracket between generators of degrees $1$ and $0$,  hence (iii)$\Leftrightarrow$(iv). The Chevalley-Eilenberg differential $\d_{CE}$ 
on  $\Gamma(\wedge V)$ is the unique degree $1$ derivation such that
\begin{equation}\label{eq:anchorformula}
 \iota(\sigma)\d_{CE} f=\L_{\a(\sigma)}(f)
 \end{equation}
for $f\in C^\infty(M)$ and $\sigma\in \Gamma(V^*)$, and such that 
\begin{equation}\label{eq:sigmasigma}
[\iota(\sigma_1),[\iota(\sigma_2),\d_{CE}]]=\iota(\Lie{\sigma_1,\sigma_2})
\end{equation}
for all $\sigma_1,\sigma_2\in \Gamma(V^*)$;
%
one can recover the bracket and anchor from these identities, giving the equivalence (iii)$\Leftrightarrow$(v). 

\subsection{Results for Poisson double vector bundles}
We are interested in the counterparts of these correspondences for double vector bundles. 
A Poisson bivector field $\pi\in\mf{X}^2(D)$ on a double vector bundle is called \emph{double-linear} if it is linear for both vector bundle structures, i.e., homogeneous of bidegree $(-1,-1)$. Following Mackenzie \cite{mac:ehr}, a double vector bundle with a double-linear Poisson bivector field $\pi$ is called a \emph{Poisson double vector bundle}. 
\begin{theorem}\label{th:PVB}
Let $D$ be a double vector bundle with sides $A,B$ and with $\on{core}(D)=\CC^*$. The following are equivalent. 
\begin{enumerate}
\item[(i)] a double-linear Poisson structure $\pi$ on $D\to M$, 
\item[(ii)] a bidegree $(-1,-1)$ Poisson bracket $\{\cdot,\cdot\}$ on the algebra $\S(D)$ of double polynomials, 
\item[(iii)] a $\VB$-algebroid structure on $D'$ over $B$,
\item[(iv)] a $\VB$-algebroid structure on $D''$ over $A$,
\item[(v)] 
a Lie algebroid structure on $\wh{\CC}$, together with representations on  $A^*$ and $B^*$, 
and an invariant  bilinear pairing 
$A^*\times_M B^*\to \R$, with certain compatibility conditions (cf.~ Theorem \ref{th:dvbdata} below),
\item[(vi)] a bidegree $(-1,-1)$ Gerstenhaber bracket on the Weil algebra $\W(D)$,
\item[(vii)] a bidegree $(0,1)$ differential $\d'_v$ on the Weil algebra $\W(D')$, 
\item[(viii)] a bidegree $(1,0)$ differential $\d''_h$ on the Weil algebra $\W(D'')$.
\end{enumerate}
\end{theorem}
Some of these equivalences are already known: Given 
$\pi$, the corresponding Poisson bracket $\{\cdot,\cdot\}$ on $C^\infty(D)$ restricts to the subalgebra $\S(D)$ of double-polynomial functions, and is uniquely determined by this restriction 
(since the differentials of functions in this subalgebra span the cotangent bundle everywhere). 
The bivector $\pi$ being double-linear means precisely that this Poisson bracket has bidegree $(-1,-1)$, hence  (i)$\Leftrightarrow$(ii). The equivalence with (iii), (iv) is due to Mackenzie  \cite{mac:ehr} (see also \cite{bur:vec}): Regarding $D$ as a vector bundle over $B$, and using the nondegenerate pairing $D\times_B D'\to \R$ from 
\eqref{eq:dvbpairings}, the Poisson structure $\pi$ determines 
a Lie algebroid structure $D'\Ra B$. The bivector field $\pi$ being linear in the vertical direction $D\to A$ implies that the horizontal scalar multiplication on $D'$ is by Lie algebroid morphisms, 
which shows that $D'$ is a $\VB$-algebroid. 
Similarly, from the pairing $D''\times_A D\to \R$ we obtain a $\VB$-algebroid structure on $D''\Ra A$.
We depict these $\VB$-algebroid structures on $D',\,D''$ by 
 \begin{equation}\label{eq:liealg1}
 \xymatrix{ {D'} \ar[r] \ar@{=>}[d] & \CC \ar@{=>}[d]\\
B\ar[r] & M}\ \ \ \ \ \ \ 
 \xymatrix{ {D''} \ar@{=>}[r] \ar[d] & A \ar[d]\\
	\CC\ar@{=>}[r] & M}
 \end{equation}
where the double arrow indicates Lie algebroid directions. 
In particular, we see that the bundle $\CC$ becomes a Lie algebroid $\CC\Ra M$. (The Lie algebroid structures on $\CC$ coming from the $\VB$-algebroid structures on $D',\ D''$  coincide; indeed, we will see below that both are induced from a $\VB$-algebroid structure $\wh{\CC}\ra M$.)  The characterizations (v), (vi), (vii) will be consequences of 
Theorems \ref{th:dvbdata}, \ref{th:Gerstenhaber}, and  \ref{th:uniquedifferential} below, while 
(viii) is obtained by applying (vii) to the flip of $D$.

\subsection{Examples of Poisson double vector bundles}
As a preparation for the general situation, let us consider some special cases.

\begin{example}
Any Poisson vector bundle $V\to M$ can be seen as a Poisson double vector bundle 
with zero side bundles,  thus $A=B=M,\ \CC=V^*$.   In this case, 
\[ \xymatrix{ {D=V} \ar[r] \ar[d] & M\ar[d]\\
M\ar[r] & M}
\ \ \ \ \ \ \ 
 \xymatrix{ {D'=V^*} \ar[r] \ar@{=>}[d] & V^* \ar@{=>}[d]\\
M\ar[r] & M}\ \ \ \ \ \ \ 
 \xymatrix{ {D''=V^*} \ar@{=>}[r] \ar[d] & M \ar[d]\\
  V^*\ar@{=>}[r] & M}\]
\end{example}

\begin{example}
Suppose $D$ is a Poisson double vector bundle for which the side bundle $A$ is zero. 
Then 
\[ \xymatrix{ {D=B\times_M \CC^*} \ar[r] \ar[d] & B \ar[d]\\
M\ar[r] & M}\ \ \ \ \ \ \ 
 \xymatrix{ {D'=B\times_M \CC} \ar[r] \ar@{=>}[d] & \CC \ar@{=>}[d]\\
B\ar[r] & M}\ \ \ \ \ \ \ 
 \xymatrix{ {D''=\CC\times_M B^*} \ar@{=>}[r] \ar[d] & M \ar[d]\\
  \CC\ar@{=>}[r] & M}\]
Here  $D'\Rightarrow B$ is the \emph{action Lie algebroid} for a representation of $\CC\Ra M$ on $B$, and the second diagram describes $D''\Rightarrow M$ as the \emph{semi-direct product Lie algebroid} for the dual $\CC$-representation on $B^*$. 
\end{example}

\begin{example}\label{ex:bilinearform}
Suppose $D$ is a \emph{vacant} Poisson double vector bundle, that is, with a zero core:
\[
\xymatrix{ {D=A\times_M B} \ar[r] \ar[d] & B \ar[d]\\
A\ar[r] & M}\ \ \ \ \ \ \ 
 \xymatrix{ {D'=A^*\times_M B} \ar[r] \ar@{=>}[d] & M \ar@{=>}[d]\\
B\ar[r] & M}\ \ \ \ 
 \xymatrix{ {D''=A\times_M B^*} \ar@{=>}[r] \ar[d] & A \ar[d]\\
  M\ar@{=>}[r] & M}\]
We claim that  double-linear Poisson structures $\pi$ on $D=A\times_M B$ are 
equivalent to bilinear pairings $(\cdot,\cdot)\colon A^*\times_M B^*\to \R$. To see this, 
note that the bigraded algebra $\S(D)$  is generated by 
$\S^{0,0}(D)=C^\infty(M)$,  $\S^{1,0}(D)=\Gamma(A^*)$,\ and $\S^{0,1}(D)=\Gamma(B^*)$. Given $\pi$, it follows for degree reasons that the only non-trivial Poisson bracket  of generators are between $\alpha\in \Gamma(A^*)$ and $\beta\in \Gamma(B^*)$;  the resulting pairing $(\alpha,\beta)=\{\alpha,\beta\}$ is $C^\infty(M)$-linear by the derivation property. 
Conversely, given  the pairing, we define a bi-derivation by letting $\{\alpha,\beta\}=(\alpha,\beta)$, and setting all other brackets between generators equal to zero. This bi-derivation satisfies the Jacobi identity, since triple brackets between generators are always zero.  This proves the claim. The Lie algebroid structure on $D'$ is that of an action Lie algebroid for the translation action of $A^*$ on $B$, given by the map $\Gamma(A^*)\to \Gamma(B)=\mf{X}(B)_{[-1]},\ \alpha\mapsto (\alpha,\cdot)$, and similarly for $D''$. 
Since $\CC=0$, we have $\wh{\CC}=A^*\otimes B^*$. The sections of this bundle have a Lie bracket, coming from its identification with double-linear functions on $D$:\[ [\alpha_1\beta_1,\,\alpha_2\beta_2]\equiv\{\alpha_1\beta_1,\,\alpha_2\beta_2\}= (\alpha_1,\beta_2)\,\alpha_2 \beta_1-(\alpha_2,\beta_1) \alpha_1\beta_2;\]thus $\wh{\CC}$ becomes a Lie algebroid with zero anchor. Likewise, the Poisson bracket of such functions with $\alpha\in\Gamma(A^*)$ or $\beta\in \Gamma(B^*)$ defines representations of this Lie algebroid on $A^*,B^*$, respectively. 
\end{example}

\begin{example}\label{ex:decomposed}
Suppose that $\CC\Ra M$ is a Lie algebroid, together  with representations $\nabla^{A^*},\nabla^{B^*}$ on $A^*,B^*$, and that $(\cdot,\cdot)\colon A^*\times_M B^*\to \R$ is a bilinear pairing that is $\CC$-invariant in the sense that   \[ (\nabla_{\cc}^{A^*}\alpha,\beta)+(\alpha,\nabla_{\cc}^{B^*}\beta)=\L_{\a(\cc)}(\alpha,\beta)\]
 for $\alpha\in \Gamma(A^*),\ \beta\in \Gamma(B^*),\ \cc\in \Gamma(\CC)$. Then $D=A\times_M B\times_M \CC^*$ becomes a Poisson double vector bundle, with the non-zero brackets on generators given as 
 \[ 
 \{\alpha,\beta\}=(\alpha,\beta),\ \ \ 
 \{\cc,\alpha\}=\nabla^{A^*}_\cc\alpha,\ \ \ 
 \{\cc,\beta\}=\nabla^{B^*}_\cc \beta,\ \ \ 
 \{\cc_1,\cc_2\}=\Lie{\cc_1,\cc_2},\ \ \ 
 \{\cc,f\}=\L_{\a(\cc)} f,\]
 for $f\in C^\infty(M)=\S^{0,0}(D)$, $\alpha\in \Gamma(A^*)=\S^{1,0}(D)$, $\beta\in \Gamma(B^*)=\S^{0,1}(D)$, and $\cc,\cc_1,\cc_2\in \Gamma(\CC)\subseteq \S^{1,1}(D)$. The Jacobi identity and the biderivation property of $\{\cdot,\cdot\}$ follow from the definition of Lie algebroids
 and their representations, together with the invariance of the pairing $(\cdot,\cdot)$.
\end{example}


\subsection{The Lie algebroid structure on $\wh{\CC}$}
In this section, we will concentrate on the characterization (v) of Poisson double vector bundles from Theorem \ref{th:PVB}. Suppose $D$ is a Poisson double vector bundle, with corresponding Poisson bracket $\{\cdot,\cdot\}$. Recall that the algebra $\S(D)=\bigoplus \S^{r,s}(D)$ is
generated by 
\[ \S^{1,1}(D)=\Gamma(\wh{\CC}),\ \ 
\S^{1,0}(D)=\Gamma(A^*),\ \ 
\S^{0,1}(D)=\Gamma(B^*),\ \ \ 
\S^{0,0}(D)=C^\infty(M);
\]
we will use these identifications without further comment, and for example think of $\alpha\in \Gamma(A^*)$ as a function on $D$. The Poisson bracket gives bilinear maps 
$\S^{r,s}(D)\times \S^{r',s'}(D)\to 
\S^{r+r'-1,s+s'-1}(D) $, and is uniquely determined by the resulting maps on generators, 
\begin{align}
\Lie{\cdot,\cdot}\colon & \ \ \S^{1,1}(D)\times \S^{1,1}(D)\to \S^{1,1}(D),\ \ \ \Lie{\wh{\cc}_1,\wh{\cc}_2}=\{\wh{\cc}_1,\wh{\cc}_2\},
\label{eq:cc}\\
\a\colon & \ \ \S^{1,1}(D)\times \S^{0,0}(D)\to \S^{0,0}(D),\ \ \ \a(\wh{\cc},f)=\{\wh{\cc},f\},
\label{eq:cf}\\
\nabla^{A^*}\colon & \ \ \S^{1,1}(D)\times \S^{1,0}(D)\to \S^{1,0}(D),\ \ \ \ \nabla_{\wh{\cc}}^{A^*}\alpha=\{\wh{\cc},\alpha\},
\label{eq:ca}\\
\nabla^{B^*}\colon & \ \ \S^{1,1}(D)\times \S^{0,1}(D)\to \S^{0,1}(D),\ \ \ \ \nabla_{\wh{\cc}}^{B^*}\beta=\{\wh{\cc},\beta\},
\label{eq:cb}\\
(\cdot,\cdot)\colon & \ \ 
\S^{1,0}(D)\times \S^{0,1}(D)\to \S^{0,0}(D),\ \ \ \ (\alpha,\beta)=\{\alpha,\beta\}\label{eq:ab}
\end{align}
(the other brackets between generators are zero, for degree reasons).

\begin{lemma}\label{lem:reps}
The formulas \eqref{eq:cc}--\eqref{eq:ab} define a Lie algebroid structure on 
$\wh{\CC}$, with representations on $A^*,B^*$, and with an invariant bilinear pairing $(\cdot,\cdot)\colon A^*\times_M B^*\to \R$. 
\end{lemma}
\begin{proof}
The derivation property of the Poisson bracket shows that 
\eqref{eq:cf} is $C^\infty(M)$-linear in the first argument and satisfies a Leibniz rule in the second argument; hence that $\a(\wh{\cc})=\a(\wh{\cc},\cdot)$ comes from a bundle map 
$\a\colon \wh{\CC}\to TM$. The Jacobi identity for $\{\cdot,\cdot\}$  implies that \eqref{eq:cc} is a Lie bracket on $\Gamma(\wh{\CC})$, and  the derivation property for $\{\cdot,\cdot\}$  
shows that $\Lie{\cdot,\cdot}$ satisfies the Leibniz rule for the anchor map $\a$, hence that $\wh{\CC}$ is a Lie algebroid. Further applications of the Jacobi identity and derivation property of $\{\cdot,\cdot\}$ show that $\nabla^{A^*},\nabla^{B^*}$ are representations of the Lie algebroid $\wh{\CC}$ on $A^*,B^*$, and that the pairing $(\cdot,\cdot)$ is $\wh{\CC}$-invariant. 
\end{proof}

\begin{remark}\label{rem:laremarks}
The Lie algebroid structure $\wh{\CC}\Ra M$, and its action on $A^*,B^*$ were first observed by Gracia-Saz and Mehta, \cite[Section 4.3]{gra:vba} in terms of  the $\VB$-algebroid  $D'\Rightarrow B$ and the identification  $\Gamma_\lin(D',B)=\Gamma(\wh{\CC})$. 
In this approach, the Lie algebroid bracket of $\wh{\CC}$ is the restriction of 
the bracket $\Lie{\cdot,\cdot}_{D'}$ to linear sections, the representation on $A^*$ 
is the bracket of linear sections with 
$\Gamma_\core(D',B)=\Gamma(A^*)$. The representation of $\wh{\CC}$ on $B^*$ (respetively, the pairing $(\cdot,\cdot)$ between $A^*$ and $B^*$) are given by the anchor $\a_{D'}$ on linear (respectively, core) sections, applied to $\Gamma(B^*)$ viewed as linear functions on $B$. 

Similarly, we can describe the data \eqref{eq:cc}--\eqref{eq:ab} in terms of the $\VB$-algebroid structure  $D''\Rightarrow A$. 
\end{remark}

The Lie algebroid representations of $\wh{\CC}$ on $A^*,B^*$ and the bilinear form satisfy certain compatibility conditions. Recall from Example \ref{ex:bilinearform} that the pairing $(\cdot,\cdot)\colon A^*\times_M B^*\to \R$ defines a Lie algebroid structure on $A^*\otimes B^*$, with zero anchor, and that this Lie algebroid comes with natural representations on $A^*,B^*$.  The data for $\wh{\CC}$ must `extend' these data for its subbundle $i_{\wh{\CC}}(A^*\otimes B^*)$:

\begin{theorem}\label{th:c-liealgebroid}\label{th:dvbdata}
Let $D$ be a double vector bundle.  A Lie algebroid structure on the bundle $\wh{\CC}\to M$, together with 
Lie algebroid representations  on $A^*$ and $B^*$
and an invariant bilinear pairing $(\cdot,\cdot)\colon A^*\times_M B^*\to \R$,
defines a double-linear Poisson structure on $D$ if and only if the following compatibility conditions are satisfied:
\begin{enumerate}
\item[(i)] The image of $i_{\wh{\CC}}\colon \Gamma(A^*\otimes B^*)\hra \Gamma(\wh{\CC})$ is a Lie algebra ideal
(in particular, $i_{\wh{\CC}}(A^*\otimes B^*)$ is a Lie subalgebroid of $\wh{\CC}$), 
\item[(ii)] the $\wh{\CC}$-representations
on $A^*,B^*$ extend those of its Lie subalgebroid  $i_{\wh{\CC}}(A^*\otimes B^*)$,
\item[(iii)]  
the  $\wh{\CC}$-representation on $i_{\wh{\CC}}(A^*\otimes B^*)$ is the tensor product of those  on $A^*,\, B^*$. 
\end{enumerate}
\end{theorem}
Condition (i) determines a Lie algebroid structure on $\CC$, 
in such a way that 
\[ 0\to A^*\otimes B^*\stackrel{i_{\wh{\CC}}}{\lra} \wh{\CC}\to \CC\to 0\]
is an exact sequence of Lie algebroids.  
\begin{proof}
Throughout, we denote by $\alpha,\alpha_1$ sections of $A^*$, by $\beta,\beta_1$ sections of $B^*$, and by $\wh{\cc},\wh{\cc}_1$ sections of $\wh{\CC}$. Suppose first that a double-linear Poisson structure on $D$ is given, determining the Lie algebroid structure on $\wh{\CC}$, representations on $A^*,B^*$, and a pairing $(\cdot,\cdot)$. On the level of sections, 
the inclusion $i_{\wh{\CC}}$ is just the multiplication map $\alpha\otimes \beta\mapsto \alpha\beta$. Thus (i), (iii) amount to the derivation property 
\[ \{\wh{\cc},\alpha\beta\}=\{\wh{\cc},\alpha\}\beta+\alpha\{\wh{\cc},\beta\},\]
while (ii) corresponds to 
$\{\alpha\beta,\alpha_1\}=-(\alpha_1,\beta)\alpha$ and $\{\alpha\beta,\beta_1\}=(\alpha,\beta_1)\beta$. 

Conversely, suppose (i),(ii),(iii) are satisfied.  Recall from Proposition \ref{prop:chen} that $D$ is a sub-double vector bundle of $\wh{D}=A\times_M B\times_M\wh{\CC}^*$. The formulas of Example 
\ref{ex:decomposed} (with $\CC$ replaced by $\wh{\CC}$) define a double-linear Poisson structure 
on $\wh{D}$. 
We will show that $D$ is a Poisson submanifold of $\wh{D}$, and hence is a Poisson double vector bundle. The ideal $\ca{I}\subset \S(\wh{D})$ of double-polynomial functions vanishing on $D$ is generated by functions  of the form 
\begin{equation}\label{eq:alphabeta}
 \alpha \beta-i_{\wh{\CC}}(\alpha\otimes \beta)\end{equation}
 with $\alpha\in \Gamma(A^*),\beta\in \Gamma(B^*)$. 
 To show that $\ca{I}$ is an ideal for the Poisson bracket, it suffices to show that the Poisson bracket of functions \eqref{eq:alphabeta} with any of the generators lies in the ideal. 
For $\wh{\cc}\in \Gamma(\wh{\CC})$ we have 
\begin{equation}\label{eq:lat1}
 \{\wh{\cc},\ \alpha \beta-i_{\wh{\CC}}(\alpha\otimes \beta)\}
=(\nabla_{\wh{\cc}}^{A^*}\alpha)\beta-i_{\wh{\CC}}(\nabla_{\wh{\cc}}^{A^*}\alpha\otimes \beta)
 +\alpha \nabla_{\wh{\cc}}^{B^*}\beta-i_{\wh{\CC}}(\alpha\otimes \nabla_{\wh{\cc}}^{B^*}\beta)\in\ca{I},
 \end{equation}
where we used  (i) and (iii). For $\alpha_1\in \Gamma(A^*)$ 
we compute 
\begin{equation}\label{eq:lat2}  \{\alpha_1,\alpha\beta -i_{\wh{\CC}}(\alpha\otimes \beta)\}
=\alpha(\alpha_1,\beta)+\nabla_{i_{\wh{\CC}}(\alpha\otimes \beta)}^{A^*}\alpha_1=0,\end{equation}
where we used (ii). A similar argument applies to generators $\beta_1\in \Gamma(B^*)$. Finally, for $f\in C^\infty(M)$
\begin{equation}\label{eq:lat3}
 \{f,\alpha\beta -i_{\wh{\CC}}(\alpha\otimes \beta)\}
=\L_{\a(i_{\wh{\CC}}(\alpha\otimes \beta))}(f)=0
\end{equation}
since $\a\circ i_{\wh{\CC}}=0$ by (i). 
\end{proof}

\begin{remark} As pointed out by one of the referees, the necessity of the conditions (i),(ii),(iii) may also be found in Section 4 of the paper \cite{gra:vba}. The sufficiency is not noted there, but can be proved using similar techniques. 	Luca Vitagliano has remarked that Theorem \ref{th:dvbdata} can also be obtained as a consequence of \cite[Theorem 2.33]{esp:inf}. \end{remark}

Let us note the following consequences of the discussion above:
\begin{proposition}\label{prop:observation}
Let $D$ be a Poisson double vector bundle, with side bundles $A,B$ and with 
$\CC=\core(D)^*$. Then 
\begin{enumerate}
\item $A\times_MB$ inherits a double-linear Poisson structure, with $\varphi\colon D\to A\times_M B$ a Poisson map.
\item The subbundle $\core(D)$ is a Poisson-Dirac submanifold of $D$: every smooth function on the core 
extends to a smooth function on $D$ with Hamiltonian vector field tangent to the core.
\item $\wh{D}=A\times_M B\times_M \wh{\CC}^*$ acquires a 
double-linear Poisson structure, 
such that $D\subset \wh{D}$ is a Poisson submanifold.
\end{enumerate}
\end{proposition}

\begin{proof} 
For (a), observe that the image of the pullback map $\varphi^*\colon \S(A\times_M B)\to \S(D)$ is
the subalgebra generated by $\Gamma(A^*),\Gamma(B^*)$; by the bracket relations 
\eqref{eq:cc}--\eqref{eq:ab} it is a Poisson subalgebra. Part (c) is contained in the proof of 
Theorem \ref{th:dvbdata}. For (b), note that it is enough to prove the analogous statement for $\wh{D}$, since $D$ is a Poisson submanifold and 
$\core(D)=D\cap \core(\wh{D})$. But functions on $\wh{\CC}^*$  
extend canonically to functions on $\wh{D}$, by taking the pullback under $\wh{D}\to \wh{\CC}^*$. The
vanishing ideal of $\wh{\CC}^*$ is generated by $\Gamma(A^*),\Gamma(B^*)\subset C^\infty(\wh{D})$, and is preserved under Poisson brackets with pullbacks of functions on $\wh{\CC}^*$. This means that the Hamiltonian vector fields of the latter are tangent to $\wh{\CC}^*$. 
\end{proof}

\subsection{Gerstenhaber brackets}
Our next aim is to interpret double-linear Poisson structures on $D$ in terms of a `Gerstenhaber' bracket on the Weil algebra $\W(D)$, as in item (vi) of Theorem \ref{th:PVB}. We make the following definitions. Let $\A$ be a bigraded commutative superalgebra. A \emph{bidegree $(-1,-1)$ Gerstenhaber bracket} on $\A$ is a bilinear map $\Ger{\cdot,\cdot}\colon \A\times \A\to \A$ of bidegree $(-1,-1)$, such that $\A[1,1]$ (i.e., the space $\A$ with bidegrees  shifted
down by $(1,1)$) becomes a bigraded super-Lie algebra, and for
$x\in \A^{p,q}$ 
the map $\Ger{x,\cdot}$ is a superderivation of bidegree $(p-1,q-1)$ of the algebra structure on $\A$: In particular
\[ \Ger{x,y}=-(-1)^{|x| |y|}\Ger{y,x},\ \ \ 
\Ger{x,y z}=\Ger{x,y} z+(-1)^{|x||y|}y \Ger{x,z}.\]
From now on, we we will omit the explicit mention of the bidegree $(-1,-1)$, taking the degree shifts to be understood. 
(This deviates from  the work of Huebschmann \cite{hue:dif}, where the degree shift is  taken to be $(1,0)$.) Note also that we will reserve the symbol $\Ger{\cdot,\cdot}$ for Gerstenhaber brackets on bigraded superalgebras, to avoid confusion with various other Lie brackets and commutators.  

\begin{theorem}\label{th:Gerstenhaber}
A double-linear Poisson structure $\pi$ on a double vector bundle $D$ is equivalent to a Gerstenhaber bracket $\Ger{\cdot,\cdot}$ on the Weil algebra $\W(D)$. 
\end{theorem}
\begin{proof}
The argument is similar to that for Theorem \ref{th:dvbdata}, hence we will be brief. 
Note that for 
$r,s\le 1$,
the spaces  $\W^{r,s}(D)$ coincide with $\S^{r,s}(D)$:
\[ \W^{1,1}(D)=\Gamma(\wh{\CC}),\ \ 
\W^{1,0}(D)=\Gamma(A^*),\ \ 
\W^{0,1}(D)=\Gamma(B^*),\ \ \ 
\W^{0,0}(D)=C^\infty(M).
\]  
A Gerstenhaber bracket on $\W(D)$ gives bilinear maps 
$\W^{r,s}(D)\times \W^{r',s'}(D)\to \W^{r+r'-1,s+s'-1}(D)$. 
The following formulas define a Lie algebroid structure on $\wh{\CC}$, together with representations of this Lie algebroid on $A^*,B^*$, and a bilinear pairing $(\cdot,\cdot)$ between $A^*$ and $B^*$:
\begin{equation}\label{eq:lotsofbrackets1}
\Ger{\wh{\cc}_1,\wh{\cc}_2}=\Lie{\wh{\cc}_1,\wh{\cc}_2},\ \ \ \ \Ger{ \wh{\cc},f}=\L_{\a(\wh{\cc})}(f),
\end{equation}
\begin{equation}\label{eq:lotsofbrackets2}
 \Ger{ \wh{\cc},\alpha}=\nabla_{\wh{\cc}}^{A^*}\alpha,\ \ \ \ 
\Ger{ \wh{\cc},\beta}=\nabla_{\wh{\cc}}^{B^*}\beta,
\ \  \ \ \Ger{ \alpha,\beta}=-(\alpha,\beta).
\end{equation}
Note that these formulas are obtained from 
\eqref{eq:cc}--\eqref{eq:ab} by replacing Poisson brackets with Gerstenhaber brackets, except for an extra minus sign in the last formula.   The various compatibilities in Theorem \ref{th:c-liealgebroid} follow from the Jacobi identity and derivation property of  $\Ger{\cdot,\cdot}$.  

 Conversely, given a double-linear Poisson structure $\pi$ and the associated 
data from Theorem \ref{th:c-liealgebroid}, we obtain a Gerstenhaber bracket as 
 follows. Consider the Poisson double vector bundle $\wh{D}=A\times_M B\times_M \wh{\CC}^*$ (cf.~ Example \ref{ex:decomposed}).  
%
%
On the super-commutative algebra $\W(\wh{D})$, we define a super-biderivation $\Ger{\cdot,\cdot}$  of bidegree $(-1,-1)$, by taking \eqref{eq:lotsofbrackets1} and \eqref{eq:lotsofbrackets2} to be the nontrivial  brackets between generators. This super-biderivation satisfies the super-Jacobi identity, as one checks on generators. Finally, by essentially the same argument as for the Poisson bracket on $\S(D)$, this Gerstenhaber bracket descends to $\W(D)=\W(\wh{D})/\sim$. In repeating the 
calculations \eqref{eq:lat1} -- \eqref{eq:lat3}, the second equation encounters a minus sign 
since $\Ger{\alpha_1,\alpha\beta}=- (\alpha_1,\beta)\,\alpha$ in contrast to 
$\{\alpha_1,\alpha\beta\}=(\alpha_1,\beta)\,\alpha$; this is compensated by the extra sign in the last equation of 
\eqref{eq:lotsofbrackets2}.  
\end{proof}
\subsection{Differentials}\label{subsec:differentials}
Suppose $D$ is a Poisson double vector bundle. The corresponding $\VB$-algebroid structure $D'\Ra B$ (dual to the Poisson vector bundle $D\to B$) gives a Chevalley-Eilenberg differential $\d_{D'}$ on $\Gamma(\wedge^\bullet_B D,B)$. Since $\d_{D'}$ commutes with the scalar multiplication $\kappa_t^v$ on $D$ and on $B$, it restricts to a differential on core and linear sections of $\wedge_B D$ over $B$. Since $\d_{D'}$ is a derivation with respect to the wedge product, we see that 
the core sections become a differential graded algebra, 
and the linear sections a differential graded module over this differential graded algebra. 

On the other hand, recall  from Proposition \ref{prop:weilident} that the linear and core sections of 
$\wedge_B D$ over $B$ are identified with $\W^{1,\bullet}(D')$ and $\W^{0,\bullet}(D')$, respectively. 
Hence, a bidegree $(0,1)$ differential on $\W(D')$ again restricts to differentials on the core and linear 
sections, making them into differential graded algebras and differential graded modules, respectively. Hence, to prove the characterization of double-linear Poisson structures on $D$ in terms of differentials 
on $\W(D')$, it suffices to show that we can reverse these constructions. 

\begin{proposition}
Suppose $\Gamma_{\on{core}}(\wedge^\bullet_B D,B)$ and $\Gamma_\lin(\wedge^\bullet_B D,B)$ are equipped with differentials $\d$, for which they are a
differential graded algebra and  differential graded module respectively. Then there are unique extensions of 
$\d$ to 
\begin{enumerate}
	\item a bidegree $(0,1)$ differential $\d_v'$ on the algebra $\W(D')$, 
	\item a degree $1$ differential $\d_{D'}$ on the algebra $\Gamma(\wedge_B D,B)$,
\end{enumerate}
as superderivations for the algebra structures. 
\end{proposition}
\begin{proof}
	\begin{enumerate}
		\item By definition, $\W(D)$ is generated by 
		elements in bidegree $(i,j)$ with $0\le i,j\le 1$. 
		Put $\d_v'x=\d x$ whenever $x$ is one of these generators. To show that this definition extends as a superderivation , we have to verify that 
		it is compatible with the relations between generators.  The defining relation of $\W(D')$ 
		(aside from super-commutativity and $C^\infty(M)$-linearity) is that for $\beta\in \Gamma(B^*)=\Gamma_\lin(\wedge^0_B D,B)
		$ and $\ggamma\in\CC^*=\Gamma_\core(\wedge^1 B D,B)$, the linear section $i_{\wh{A}}(\beta\otimes \ggamma)$ of 
		$\wh{A}=\Gamma_\lin(\wedge^1_B D,B)$ coincides with the product $\beta\ggamma$. Thus, we need that 
\[ \d(i_{\wh{A}}(\beta\otimes \ggamma))
=(\d\beta)\ \ggamma-\beta\,(\d\ggamma).\]
		But the linear section $i_{\wh{A}}(\beta\otimes \ggamma)$ is simply the product of the linear function $\beta$ with the core section $\ggamma$; hence the desired identity follows from the assumption that linear sections are a differential graded module over the core sections. 
		 
		\item The algebra of sections of $\wedge_B D$ over $B$ has a subalgebra 
		\[ \Gamma_{\on{pol}}(\wedge_BD,B)=\bigoplus_{m,n\ge 0}\Gamma(\wedge^n_B D,B)_{[-n+m]}\]
		of \emph{polynomial sections},  in the notation of Section \ref{subsec:lincor}. 
		It is a super-commutative bigraded algebra, with bigrading given by $m,n$, and with the 
		for the $\Z_2$-grading given as the mod 2 reduction of $n$. 
		It is generated by its components in degree $m,n\le 1$, which coincide with those for $\W(D')$, 
		and the relations between these generators are the same as for $\W(D')$, with the exception that 
		the relation $\beta_1\beta_2=-\beta_2\beta_1$ for 
		$\beta_1,\beta_2\in \Gamma(B^*)$ gets replaced with $\beta_1\beta_2=\beta_2\beta_1$. The same argument as for $\W(D')$ shows that $\d$ extends to a superderivation $\d_{D'}$ of $\Gamma_{\on{pol}}(\wedge_B D,B)$.
		By \cite[Theorem 3.15]{gra:vba}, the latter determines a Lie algebroid structure on $D'$ over $B$, which, in turn, extends the differential to all sections of $\wedge_B D$ over $B$. 
	\end{enumerate}	
\end{proof}

\begin{remark}\label{rem:cabdru}
In \cite{gra:vba}, the double complex $ \Gamma_{\on{pol}}(\wedge_BD,B)$ is denoted $\Omega^{\bullet,\bullet}(D')$. 
As explained there, it may be regarded as the space of double-polynomial functions on the supermanifold $D'[1,0]$, using a parity change only in the vertical vector bundle direction.  Note also that in the notation of Cabrera-Drummond \cite{cab:hom} $ \Gamma_{\on{pol}}(\wedge^\bullet_BD,B)_{[k-\bullet]}$ is the space of $k$-homogeneous cochains for the
$\VB$-algebroid $D'\Ra B$. 
\end{remark}

\section{Relationships between brackets, differentials, and pairings}\label{sec:relations}
Throughout, $D$ will denote a Poisson double vector bundle, with side bundles $A,B$ and with $\core(D)^*=\CC$. Equivalently, the double vector bundles $D',D''$   are $\VB$-algebroids $D'\Ra B$, $D''\Ra A$, respectively. In the last section, we saw how 
these structures are equivalent to a Gerstenhaber bracket $\Ger{\cdot,\cdot}$ on $\W(D)$, a vertical differential $\d_v'$ on $\W(D')$, and a horizontal differential $\d_h''$ on $\W(D'')$. 
These data  interact in many ways,
via  the contraction operators and the various pairings between the Weil algebras.

\subsection{Differential and contractions on $\W(D')$}
Thinking of $D$ as the dual bundle of $D'$ over $B$, using once again the duality pairing \eqref{eq:dvbpair1},  we have the usual contraction operators $\iota_{D'}(\sigma)$ on $\Gamma(\wedge_B D,B)$, for 
any $\sigma\in \Gamma(D',B)$. On the other hand,  sections $\wh{\cc}\in \Gamma(\wh{\CC})\cong \Gamma_\lin(D',B)$ define contraction operators  $\iota_v'(\wh{\cc})$ of bidegree $(0,-1)$ of $\W(D')$, while  sections $\alpha\in \Gamma(A^*)$ define contraction operators $\iota_v'(\alpha)$ of bidegree $(-1,-1)$ 
on $\W(D')$. Recall again the identifications 
$\W^{1,\bullet}(D')\cong \Gamma_\lin(\wedge_B D,B)$ and $\W^{0,\bullet}(D')\cong \Gamma_\core (\wedge_B D,B)$. Checking on generators, one verifies:
\begin{lemma}
\begin{enumerate}
	\item The isomorphism $\Gamma(\wh{\CC})\cong \Gamma_\lin(D',B)$ identifies the contraction operators 
	$\iota_v'(\wh{\cc})\colon \W^{p,\bullet}(D')\to \W^{p,\bullet}(D')$ for $p=0,1$ 
	with the operator $\iota_{D'}(\wh{\cc})$ on core sections and linear sections of $\wedge_B D$. 
	\item The isomorphism  
	$\Gamma(A^*)\cong \Gamma_\core(D',B)$ identifies
	$\iota_v'(\alpha)\colon \W^{1,\bullet}(D')\to \W^{0,\bullet}(D')$ with the operator 
	$\iota_{D'}(\alpha)\colon \Gamma_\lin(\wedge_B D,B)\to \Gamma_\core(\wedge_B D,B)$.
\end{enumerate}	
\end{lemma}
Using these facts, we obtain: 
\begin{proposition}\label{th:uniquedifferential}
	The derivation $\d'_v$ satisfies 
	\begin{align}
	[\iota'_v(\wh{\cc}_1),[\iota'_v(\wh{\cc}_2),\d'_v]]&=\iota'_v(\Ger{ \wh{\cc}_1,\wh{\cc}_2}),  \label{eq:w1}\\
	[\iota'_v(\wh{\cc}),[\iota'_v(\alpha),\d'_v]]&=-\iota'_v(\nabla_{\wh{\cc}}^{A^*}\alpha),\label{eq:w2}\\
	[\iota'_v(\alpha_1),[\iota'_v(\alpha_2),\d'_v]]&=0,\label{eq:w3}
	\end{align}
	for $\wh{\cc},\wh{\cc}_1,\wh{\cc}_2\in \Gamma(\wh{\CC}),\ \alpha,\alpha_1,\alpha_2\in \Gamma(A^*)$. Furthermore, 
	\begin{equation}\label{eq:beta1}
	\iota'_v(\wh{\cc})\d'_v f=\L_{\a(\wh{\cc})}f,\ \ \ \
	\iota'_v(\wh{\cc})\d'_v\beta=
	\nabla_{\wh{\cc}}^{B^*}\beta,\ \ \ \ 
	\end{equation}
	\begin{equation}\label{eq:beta2}
	\iota'_v(\alpha)\d'_v f=0,\ \ \ \ 
	\iota'_v(\alpha)\d'_v \beta=-(\alpha,\beta),
	\end{equation}
	for all all $\wh{\cc}\in \Gamma(\wh{\CC}),\ \alpha\in \Gamma(A^*),\ 
	f\in C^\infty(M),\ \beta\in \Gamma(B^*)$. 
\end{proposition}
\begin{proof}
Equations \eqref{eq:w1}--\eqref{eq:w3} are equalities of derivations; hence it suffices to check that both sides agree 
on $ \W^{0,\bullet}(D'), \W^{1,\bullet}(D')$. Since the identifications of these spaces with core and linear sections of $\wedge_B D$ takes $\iota_v',\d'_v$ to the contractions and Lie algebroid differential of $\Gamma(\wedge_B D,B)$, and since the right hand sides can be expressed in terms of Lie algebroid brackets (see Remark \ref{rem:laremarks}), the three equalities  \eqref{eq:w1}--\eqref{eq:w3} correspond to the formula \eqref{eq:sigmasigma} for the bracket of Lie algebroids. Similarly, \eqref{eq:beta1}, \eqref{eq:beta2} correspond to the formula \eqref{eq:anchorformula} for the anchor of a Lie algebroid. 
\end{proof}

\subsection{Relationship between Gerstenhaber bracket and differential}
\label{subsec:relations1}
The  Gerstenhaber bracket on $\W(D)$ restricts to a bracket on 
$\W^{1,1+\bullet}(D)$, making the latter into a graded Lie algebra with a representation 
$x\mapsto \Ger{x,\cdot}$ on $\W^{0,\bullet}(D)=\Gamma(\wedge B^*)$. Likewise, 
$\W^{1+\bullet,1}(D)$ is a graded Lie algebra with a representation on $\W^{\bullet,0}(D)=\Gamma(\wedge A^*)$ by graded superderivations $x\mapsto \Ger{x,\cdot}$. 

\begin{proposition}\label{prop:differential-gerstenhabera}
For $x\in \W^{1,\bullet}(D)$ and $y\in \W^{0,\bullet}(D)=\Gamma(\wedge B^*)=\W^{\bullet,0}(D')$, 
\begin{equation}\label{eq:13a}
\iota'_v(x)\d'_v y=\Ger{ x,y}.
\end{equation}
Similarly, for $x\in \W^{1,\bullet}(D)$ and $y\in \W^{0,\bullet}(D)=\Gamma(\wedge B^*)=\W^{\bullet,0}(D')$, 
\begin{equation}\label{eq:12a}
\iota''_h(x)\d''_h y=\Ger{ x,y}.
\end{equation}
\end{proposition}
\begin{proof}
For $x\in 	\W^{1,q}(D)$, both $\Ger{x,\cdot}$ and $[\iota_v'(x),\d_v']$ define superderivations 
of degree $q-1$ on $\Gamma(\wedge B^*)$. Since $\iota_v'(x)$ vanishes on 
$ \W^{\bullet,0}(D')$, Equation \eqref{eq:13a} amounts to the equality of these two superderivations. It suffices to verify this on generators $y$ of $\Gamma(\wedge B^*)$, 
given by $f\in C^\infty(M)$ or $\beta\in \Gamma(B^*)$. Furthermore, since 
$\iota_v'$ is a left $\Gamma(\wedge B^*)$-module morphism, we only need to consider the cases that $x$ is a generator $\W^{1,\bullet}(D)$, given by $\alpha\in \Gamma(A^*)$ or 
$\wh{\cc}\in \Gamma(\wh{\CC})$. These verifications are as follows, using 
 \eqref{eq:beta2}:
\begin{align*}
\Ger{\alpha,f}&=0=\iota'_v(\alpha)\d'_v f,\ \ \ \ \  \Ger{ \alpha,\beta}=-(\alpha,\beta) =\iota'_v(\alpha)\d'_v\beta,\\
\Ger{ \wh{\cc},f}&=\L_{\a(\wh{\cc})}f =\iota'_v(\wh{\cc})\d'_v f,\ \ \ \Ger{ \wh{\cc},\beta}=\nabla^{B^*}_{\wh{\cc}}\beta=\iota'_v(\wh{\cc})\d'_v\beta.
\end{align*}
Equation \eqref{eq:12a} may be proved along similar lines, or obtained from 
\eqref{eq:13a} by using the flip operation. 
\end{proof}
Note that \eqref{eq:13a} 
can be written in various other ways:
\begin{equation}\label{eq:theotherway}
\Ger{x,y}=-(-1)^{(|x|+1)|y|} \l \d_v'y,\,
x\r_{B^*}=-(-1)^{(|x|+1)|y|}
\iota_h(\d_v' y)x.
\end{equation}

\begin{proposition}\label{prop:differential-gerstenhaber}
For $x_1,x_2\in \W^{1,\bullet}(D)$, 
\begin{equation}\label{eq:13}  [\iota'_v(x_1),[\iota'_v(x_2),\d'_v]]=(-1)^{|x_2|}\iota'_v(\Ger{ x_1,x_2})\end{equation}
For  $x_1,x_2\in \W^{\bullet,1}(D)$, 
\begin{equation}\label{eq:12}
  [\iota''_h(x_1),[\iota''_h(x_2),\d''_h]]=(-1)^{|x_2|}\iota''_h(\Ger{ x_1,x_2}).\end{equation}
\end{proposition}
\begin{proof}
 Equation \eqref{eq:13} holds for $x_i=\wh{\cc}_i\in \Gamma(\wh{\CC})$ by \eqref{eq:w1} and for $x_1=\wh{\cc}\in \Gamma(\wh{\CC}),\ x_2=\alpha\in \Gamma(A^*)$ by \eqref{eq:w2}. Since both sides change sign by $(-1)^{(|x_1|+1)(|x_2|+1)}$ under interchange of $x_1,\,x_2$, this establishes the identity for generators. The general case follows by induction: the statement for  $x_1,x_2$ implies that for $x_1,\beta x_2$ with 
$\beta\in\Gamma(B^*)$, as follows:
\begin{align*} [\iota'_v(x_1),[\iota'_v(\beta\,x_2),\d'_v]]&=
\big[\iota'_v(x_1),\beta  [\iota_v'(x_2),\d'_v]+(-1)^{|x_2|+1} (\d'_v\beta)\, \iota_v'(x_2)\big]\\
&=(-1)^{|x_1|+1}\beta [\iota'_v(x_1),[\iota'_v(\,x_2),\d'_v]]
+(-1)^{|x_2|+1}(\iota_v'(x_1)\d_v'\beta)\ \iota_v'(x_2)\\
&=(-1)^{|x_1|+|x_2|+1}\iota'_v(\beta\Ger{ x_1,x_2})
+(-1)^{|x_2|+1} \iota_v'(\Ger{x_1,\beta} x_2)\\
&=(-1)^{|x_2|+1} \iota_v'\Ger{x_1,\,\beta x_2}.
\end{align*}
The arguments for $\d''_h,\ \iota_h''(x)$ are analogous; alternatively, one can use the flip operation. 
\end{proof}
 For later reference, observe the following consequence of Proposition \ref{prop:differential-gerstenhaber}. 
\begin{corollary}\label{cor:cor}
For $\phi\in \W^{1,\bullet}(D'')$, and $x_i\in  \W^{\bullet,1}(D)$, 
\[ (-1)^{|x_2|}\iota''_h(x_1)\iota_h''(x_2)\d_h'' \phi=
 \iota_h''(\Ger{x_1,x_2})\phi-\Ger{x_1,\iota''_h(x_2)\phi}+(-1)^{|x_1||x_2|}
  \Ger{x_2,\iota_h(x_1)\phi}.
 \]
\end{corollary}
\begin{proof}
The left hand side $(-1)^{|x_2|}\iota''_h(x_1)\iota''_h(x_2)\d''_h \phi$ may be written as a sum of three terms, 
\[ (-1)^{|x_2|}[\iota''_h(x_1),[\iota''_h(x_2),\d''_h]] \phi
-\iota''_h(x_1)\,\d''_h\,\iota''_h(x_2)\phi
+(-1)^{|x_1|\,|x_2|}\iota''_h(x_2)\,\d''_h\,\iota''_h(x_1)\phi.\]
By Proposition \ref{prop:differential-gerstenhaber}, the first term is $\iota''_h(\Ger{x_1,x_2})\phi$. 
For the second term,  note that $\iota''_h(x_2)\phi\in \W^{0,\bullet}(D'')=\Gamma(\wedge A^*)$. But on sections of $\wedge A^*$, the composition $\iota''_h(x_1)\circ \d''_h$ acts as 
$\Ger{x_1,\cdot}$, again using Proposition \ref{prop:differential-gerstenhaber}. Hence, the second term is 
$-\Ger{x_1,\iota''_h(x_2)\phi}$. The third term is dealt with similarly. 
\end{proof}

\subsection{Relationships between the differentials} \label{subsec:relations2}
The differential $\d'_v$ on $\W(D')$ restricts to a differential on 
\[ \W^{1,\bullet}(D')=\Gamma_\lin(\wedge^\bullet_B D,B)
\equiv \Gamma_\lin(\wedge^\bullet_B (D')^v,B)
.\]

On the other hand, we also have the horizontal differential 
on 
\[ \W^{\bullet,1}(D'')=\Gamma_\lin(\wedge_A D,A)\equiv \Gamma_\lin(\wedge^\bullet_A (D'')^h,A)
\]
coming from the horizontal $\VB$-algebroid structure $D''\Ra A$
under the identification of $D$ with the flip of the horizontal dual $(D'')^h$, it is again the restriction of the Chevalley-Eilenberg differential. In other words, it is the space 
$\mathsf{CE}_{\VB}^\bullet(D'')$ in the notation of \cite{cab:hom}. The 
Lie algebroid differential $\d$ on $\Gamma(\wedge \CC^*)$ may be interpreted as $\d'_v$ on  
$\W^{0,\bullet}(D')$ or as $\d''_h$ on $\W^{\bullet,0}(D'')$, and both $\W^{1,\bullet}(D')$ and 
$\W^{\bullet,1}(D'')$ are differential graded modules over this algebra. 
\begin{proposition}\label{prop:pairing}
The $\wedge \CC^*$-valued pairing $\l\cdot,\cdot\r_{\CC^*}$ (cf.~ \eqref{eq:three_parings2}) 
satisfies 
\[ \d\l x,y\r_{\CC^*}=\l \d''_h x,y\r_{\CC^*}+(-1)^{|x|+1}\l x,\d'_v y\r_{\CC^*},\ \ \]
for all $x\in \W^{p,1}(D'')$ and $y\in \W^{1,q}(D')
$. Here $|x|=p+1$. 
\end{proposition}


Before proving the proposition, recall that for any Poisson manifold $(Q,\pi)$, the Schouten bracket 
  defines a degree $1$ differential on the graded algebra 
  $\mf{X}^\bullet(Q)$ of multi-vector fields on $Q$:
  \[ \Lie{ \pi,\cdot}\colon \mf{X}^p(Q)\to \mf{X}^{p+1}(Q).\]
If $Q=V$ is a Poisson vector bundle, thus  $\pi$ is homogeneous of degree $-1$, then this differential preserves the graded subalgebra 
$\mf{X}^\bullet_\core(V)$ of core multi-vector fields, as well as the module $\mf{X}^\bullet_\lin(V)$ of linear multi-vector fields. It is well-known that 
the identification 
$\mf{X}^\bullet_\core(V)=\Gamma(\wedge^\bullet V)$
intertwines the differential $\Lie{ \pi,\cdot}$ with the Lie algebroid differential for $V^*\Ra M$. 

\begin{proof}[Proof of Proposition \ref{prop:pairing}]
As explained in Section \ref{subsec:multivector}, 
\[  \mf{X}^q(D)_{[-q,1-q]}\cong\W^{1,q}(D') ,\ \
\mf{X}^p(D)_{[1-p,-p]}\cong\W^{p,1}(D''),\ \
\mf{X}^n(D)_{[-n,-n]}\cong\Gamma(\wedge^n \CC^*).
\]
%
By a check on generators, one finds that 
the differential $\d'_v$ is realized as $-\Lie{\pi,\cdot}$, while 
 $\d''_h,\ \d$ are realized as $\Lie{ \pi,\cdot}$.  Furthermore, according to 
Proposition \ref{prop:pairingschouten} the pairing between $x\in \mf{X}^p(D)_{[1-p,-p]}$ and $y\in\mf{X}^q(D)_{[-q,1-q]} $ is expressed in terms of the Schouten bracket as 
$\l x,y\r_{\CC^*}=\Lie{ x,y}$. 
The proposition thus translates into the Jacobi identity
\[ \Lie{ \pi,\Lie{ x,y}}=\Lie{ \Lie{ \pi,x},y}+(-1)^{|x|}\Lie{ x,\Lie{ \pi,y}}. \qedhere\]
\end{proof}
A consequence of Proposition \ref{prop:pairing} is the following result about contraction 
operators. 


\begin{proposition}\label{prop:contraction_differential}
For  $x\in  \W^{p,1}(D'')$, we have the following 
equality of superderivations of $\W(D')$,  
\begin{equation}\label{eq:dprime}
 [\d'_v,\iota'_h(x)]=\iota'_h(\d''_h x).\end{equation}
Similarly, for $y\in  \W^{1,q}(D')$ we have the equality of superderivations of $\W(D'')$,  
\begin{equation}\label{eq:dprimeprime}
 [\d''_h,\iota''_v(y)]=\iota''_v(\d'_v y).
 \end{equation}
\end{proposition}
\begin{proof}
Both sides of \eqref{eq:dprime} are superderivations of bidegree $(-1,p)$. Hence, they both vanish on 
$W^{0,\bullet}(D')$. On sections $y\in \W^{1,q}(D')$,  
the identity becomes 
\[\d'_v  \iota'_h(x)y+(-1)^{|x|}\iota'_h(x) \d'_v y=\iota'_h(\d''_h x) y.\]
After expressing the horizontal contractions in terms of the pairing $\l\cdot,\cdot\r_{\CC^*}$, this identity reads as
\[\d \l x,y\r_{\CC^*}+(-1)^{|x|}\l x,\,\d_v' y\r_{\CC^*}=\l \d_h''x,y\r_{\CC^*}.\]
which is just the statement of Proposition \ref{prop:pairing}. Similarly, the two sides of \eqref{eq:dprimeprime} are superderivations of bidegree $(q,-1)$. Applying 
both sides to $x\in \W^{p,1}(D'')$, the identity becomes 
\[ \d_h'' \iota_v''(y) x+(-1)^{|y|}\iota_v''(y)\d_h''x=\iota_v''(\d_v'y)x,\] 
which may be written
\[ (-1)^{(|x|+1)(|y|+1)}\d \l x,y\r_{\CC^*}
+(-1)^{|y|} (-1)^{|x|(|y|+1)}\l \d_h'' x,y\r_{\CC^*}=(-1)^{(|x|+1)|y|}\l x,\,\d_v' y\r_{\CC^*}.\]
After multiplying by the sign $ (-1)^{(|x|+1)(|y|+1)}$, this becomes 
\[ \d \l x,y\r_{\CC^*}-\l \d_h'' x,y\r_{\CC^*}=(-1)^{|x|+1}\l x,\,\d_v' y\r_{\CC^*}\]
which again is a reformulation of Proposition \ref{prop:pairing}.
\end{proof}

\subsection{Application to vector bundles IV}\label{subsec:TV4}
We continue the discussion from Sections \ref{subsec:TV1}, \ref{subsec:TV2} and \ref{subsec:TV3}. 

\subsubsection{Tangent bundle $TV$}\label{subsubsec:TV4a}
Recall from \ref{subsec:TV2} that $\W(TV)$ is $C^\infty(M)$-linearly generated by differentials $\d f$, 
together with sections $\tau\in \Gamma(V^*)$ and their jets $j^1(\tau)$, subject to the relation 
$j^1(f\tau)-f j^1(\tau)=\tau\ \d f$.  
The $\VB$-algebroid structure $TV\Ra V$ determines a vertical differential $\d_v$ on the Weil algebra. This differential is given on the generators by 
\[ \d_v(f)=\d f,\ \ \ \d_v(\d f)=0,\ \ \ \d_v(\tau)=-j^1(\tau),\ \ \ \d_v(j^1(\tau))=0.\]
On $\W^{0,\bullet}(TV)=\Gamma(\wedge T^*M)$, it agrees with the de Rham differential; on $\W^{1,\bullet}(TV)=
\Gamma_{\lin}(\wedge_V T^*V,V)$ it is the restriction of the de Rham differential to linear forms. 
For example, we may verify that 
\[ \iota_v(\delta)\d_v\tau=-\nabla_\delta^* \tau=-\l j^1(\tau),\delta\r_{V^*}=-\iota_v(\delta)j^1(\tau)
\]
as required; similarly 
$ \iota_v(\delta)\d_v f=\L_{\a(\delta)}f=-\l\d f,\delta\r_{V^*}=\iota_v(\delta) \d f$. 

\subsubsection{Cotangent bundle  $T^*V$}\label{subsubsec:TV4b}
The Weil algebra $\W(T^*V)$ is $C^\infty(M)$-linearly 
generated by sections $\tau\in \Gamma(V^*),\ \sigma\in \Gamma(V),\  \delta\in \Gamma(\on{At}(V))$,
subject to the relation that the product $\tau\sigma$ in the Weil algebra equals 
$-i_{\on{At}(V)}(\tau\otimes \sigma)$. Using the Poisson brackets between these generators (viewed as functions $\tau,\,\phi_{\sigma^\sharp},\,\phi_{\wt{\a}(\delta)}$ 
on $T^*V$, as in \ref{subsec:TV1}), we obtain the formulas for the 
Gerstenhaber bracket,  
\[ \Ger{\tau,\sigma}=\l\tau,\sigma\r,\ \ 
\Ger{\delta,\sigma}=\nabla_\delta \sigma,\ \ 
\Ger{\delta,\tau}=\nabla^*_\delta \tau,\ \ 
\Ger{\delta,f}=\L_{\a(\delta)}f,\ \ 
\Ger{\delta_1,\delta_2}=\Lie{\delta_1,\delta_2}.\] 
The first equality follows from $\{\tau,\phi_{\sigma^\sharp}\}=-\L_{\sigma^\sharp}\tau=-\l\tau,\sigma\r$, the second from 
$\{\phi_{\wt{a}(\delta)},\phi_{\sigma^\sharp}\}=\phi_{[\wt{\a}(\delta),\sigma^\sharp]}=\phi_{(\nabla_\delta\sigma)^\sharp}$, and so on.

\section{Double Lie algebroids}\label{sec:DLA}
\subsection{Definition, basic properties}
The concept of a \emph{double Lie algebroid} was introduced by Mackenzie in \cite{mac:not,mac:dou,mac:dou1}, as the infinitesimal counterpart to double Lie groupoids. 
It is given by a  double vector bundle with \emph{compatible} horizontal and vertical $\VB$-algebroid structures
 \begin{equation}\label{eq:doubleliealg}
 \xymatrix{ {D} \ar@{=>}[r] \ar@{=>}[d] & B \ar@{=>}[d]\\
A\ar@{=>}[r] & M}
 \end{equation}
To formulate the compatibility condition, 
recall that a vertical $\VB$-algebroid structure makes $D''$ into a Poisson double  vector bundle;
hence $D'$ becomes a $\VB$-algebroid horizontally, 
\[
 \xymatrix{ {D'} \ar@{=>}[r]\ar[d] & \CC\ar[d]\\
B\ar@{=>}[r] & M}
\]
Similarly, a horizontal $\VB$-algebroid structure on $D$ makes $D'$ into a double Poisson vector bundle, and hence $D''$ becomes a $\VB$-algebroid vertically, 
\[
 \xymatrix{ {D''} \ar[r]\ar@{=>}[d] & A\ar@{=>}[d]\\
\CC\ar[r] & M}
\]
The compatibility condition is that the two Lie algebroids $D'\Ra \CC$ and $D''\Ra \CC$, with their natural duality pairing, form a  \emph{Lie bialgebroid}, as defined by Mackenzie-Xu \cite{mac:lie}. 
In the formulation of Kosmann-Schwarzbach \cite{kos:exa}, this means that the Chevalley-Eilenberg differential $\d_{CE}$ on $\Gamma(\wedge_\CC D',\CC)$, defined by the identification of $D'\to \CC$ with the dual of the Lie algebroid $D''\Ra \CC$, is a derivation of the Schouten bracket $\Lie{\cdot,\cdot}$ for the Lie algebroid
$D'\Ra \CC$:
\begin{equation}\label{eq:matchedpair}
 \d_{CE} \Lie{ \lambda_1,\lambda_2}=\Lie{\d_{CE}\lambda_1,\lambda_2}+(-1)^{n_1-1}\Lie{ \lambda_1,\d_{CE}\lambda_2},
 \end{equation}
for $\lambda_i\in \Gamma(\wedge^{n_i}_\CC D',\CC),\ i=1,2$. 

\begin{examples}\label{ex:doublelie}
\begin{enumerate}
\item Mackenzie  arrived at the definition of a double Lie algebroid by applying the Lie functor 
to an $\LA$-groupoid. For instance, any Poisson Lie groupoid $G\rra M$ \cite{mac:lie}, with Lie algebroid $A\Ra M$,  gives rise to a double Lie algebroid structure on $T^*A$, by applying the Lie functor to its cotangent Lie algebroid $T^*G\rra A^*$. Similarly, for a \emph{double Lie groupoid} \cite{bro:det,mac:dou,mac:dou1}, applying the Lie functor twice produces a double Lie algebroid.
\item The tangent bundle of a Lie algebroid $V\Ra M$ is a double Lie algebroid \cite[Example 4.6]{mac:not}
\[
 \xymatrix{ {TV} \ar@{=>}[r]\ar@{=>}[d] & TM \ar@{=>}[d]\\
V \ar@{=>}[r] & M}
\]
One may regard $TV$ as being obtained by applying the Lie functor to the $\LA$-groupoid 
$V\times V\rra V$ (the pair groupoid of $V$).   
\item\label{it:matchepair}
\emph{Matched pairs of Lie algebroids}, due to Lu \cite{lu:poi} and Mokri \cite{mok:mat}, are a generalization of 
a matched pair of Lie algebras \cite{maj:mat} (also known as double Lie algebras \cite{lu:po}
or twilled extensions \cite{kos:poi}). Two Lie algebroids $A\Ra M,\ B\Ra M$, with actions of $A$ on $B$ and of $B$ on $A$,
are a matched pair if the brackets and actions define a  Lie algebroid structure on the direct sum $A\oplus B\Ra M$. Mackenzie \cite{mac:not} proved that matched pairs of Lie algebroids are equivalent to \emph{vacant} double Lie algebroids 
\[ \xymatrix{ D=A\times_M B \ar@{=>}[r] \ar@{=>}[d] & B\ar@{=>}[d]\\
A	\ar@{=>}[r] & M}
	\]	
i.e. such that $\core(D)=0$. 	
\end{enumerate}
\end{examples}

\subsection{Weil algebra of a double Lie algebroid}
The following theorem gives equivalent formulations of Mackenzie's definition of a double Lie algebroid 
in terms of the Weil algebras of the three double vector bundles $D,D',D''$. 

\begin{theorem}\label{th:doubleliealgebroid}
Let $D$ be a double vector bundle. The following are equivalent: 
\begin{enumerate}
\item A double Lie algebroid structure on $D$;
\item a Gerstenhaber bracket  on the bigraded superalgebra $\W(D')$, together with a 
differential $\d'_h$ of bidegree $(1,0)$ that is a derivation of the Gerstenhaber bracket;
\item a Gerstenhaber bracket  on the bigraded superalgebra $\W(D'')$, together with a 
differential $\d''_v$ of bidegree $(0,1)$ that is a derivation of the Gerstenhaber bracket; 
\item commuting differentials $\d_h,\d_v$ on $\W(D)$, of bidegrees $(1,0)$ and $(0,1)$, respectively,
\end{enumerate}
\end{theorem}
We will break up the proof into several steps. Consider first  equivalence (a) $\Leftrightarrow$ (c). 
\begin{lemma}
A double Lie algebroid structure on $D$ is equivalent to a Gerstenhaber bracket  on $\W(D'')$, together with a  differential $\d''_v$ of bidegree $(0,1)$ that is a derivation of the Gerstenhaber bracket.
\end{lemma}
\begin{proof}
As discussed in Section \ref{subsec:wedgeweil}, there are canonical identifications
\[ \W^{0,\bullet}(D'')\cong \Gamma_\core(\wedge^\bullet_\CC D',\CC),\ \ \ 
\W^{1,\bullet}(D'')\cong \Gamma_\lin(\wedge^\bullet_\CC D',\CC).\] 
These spaces generate $\W(D'')$ as an algebra, and also $\Gamma(\wedge_\CC D',\CC)$ as a 
module over $C^\infty(M)$. Furthermore, by definition, the restriction of the Gerstenhaber bracket on $\W(D'')$ to these spaces agrees with the Lie algebroid bracket for $D'\Ra \CC$, while the differential 
 $\d''_v$ coincides with the Chevalley-Eilenberg differential for $D''\Ra \CC$. Hence, $\d''_v$ being a 
 derivation of the Gerstenhaber bracket amounts to the defining compatibility condition of a double Lie algebroid. 
\end{proof}

Theorem \ref{th:PVB} shows that for a horizontal $\VB$-algebroid structure $D\Ra B$
on a double vector bundle $D$ is equivalent to a vertical differential $\d_v$ on $\W(D)$, and also to a horizontal differential $\d'_h$ on $\W(D')$. 
By Proposition \ref{prop:contraction_differential}, these are related by 
\begin{equation}\label{eq:dh} \d_v y=\d'_h y,\ \ \  \ \ [\d_v,\iota_h(x)]=\iota_h(\d'_h x),\end{equation}
for all $y\in \Gamma(\wedge^\bullet B^*)$ and all $x\in \W^{\bullet,1}(D')$. (The first identity uses 
that both $\W^{0,\bullet}(D)$ and $W^{\bullet,0}(D')$ are identified with sections of $\wedge B^*$.)
On the other hand, using Theorem \ref{th:PVB} again, a 
vertical $\VB$-algebroid structure $D\Ra A$ is 
equivalent to a horizontal differential $\d_h$ on $\W(D)$, and also to a Gerstenhaber bracket $\Ger{\cdot,\cdot}$ on $\W(D')$. 
According to Propositions \ref{prop:differential-gerstenhabera} and \ref{prop:differential-gerstenhaber}, these are related by 
\begin{equation}\label{eq:dv}
  \iota_h(x)\,  \d_h y=\Ger{x,y},\ \ \  [\iota_h(x_1),[\iota_h(x_2),\d_h]]=(-1)^{|x_2|}\iota_h(\Ger{ x_1,x_2}),\end{equation}
for all $y\in \Gamma(\wedge^\bullet B^*)$ and all $x,x_1,x_2\in \W^{\bullet,1}(D')$. 
Consider now the situation that both a horizontal and a vertical $\VB$-algebroid structure are given. 
\begin{lemma}\label{lem:ac}
Given a horizontal $\VB$-algebroid structure $D\Ra B$ and a vertical $\VB$-algebroid structure $D\Ra A$, the super-commutator $[\d_h,\d_v]$ satisfies 
\begin{equation}\label{eq:dhdv1}
\iota_h(x) [\d_h,\,\d_v] y= (-1)^{|x|+1}\big(\d'_h\Ger{ x,y}-\Ger{ \d'_h x,y}-(-1)^{|x|}\Ger{ x,\d'_h y}\big)
\end{equation}  
for $x\in \W^{\bullet,1}(D')$ and 
$y\in \Gamma(\wedge^\bullet B^*)$, as well as 
\begin{equation}\label{eq:dhdv2}
[\iota_h(x_1),[\iota_h(x_2),[\d_h,\d_v]]]=(-1)^{|x_1|}
 \iota_h\big(\d'_h \Ger{ x_1,x_2}-\Ger{ \d'_h  x_1,x_2}-(-1)^{|x_1|}\,\Ger{ x_1,\d'_h x_2}\big)
\end{equation}
for  $x_1,x_2\in \W^{\bullet,1}(D')$. 
\end{lemma}
\begin{proof}
The two identities follow from the calculations, using \eqref{eq:dh} and \eqref{eq:dv}, 
\begin{align*}  \iota_h(x) [\d_h,\d_v] y
&=\iota_h(x) \d_h\d_v\,y +\iota_h(x) \d_v\d_h y\\
&=\Ger{ x,\d_v y}+(-1)^{|x|}[\d_v,\iota_h(x)]\d_h y
-(-1)^{|x|} \d_v \iota_h(x)\d_h y\\
&=\Ger{ x,\d'_h  y}+(-1)^{|x|}
\iota_h(\d'_h  x) \d_h y- (-1)^{|x|}\d'_h \Ger{ x,y}\\
&=(-1)^{|x|+1}\big(\d'_h \Ger{ x,y}-\Ger{ \d'_h x,y}-(-1)^{|x|}\Ger{ x,\d'_h y}\big)
\end{align*}
and 
\begin{align*}
 \iota_h(\d'_h \Ger{ x_1,x_2})&=
 [\d_v,\iota_h(\Ger{ x_1,x_2})]\\
                   &=(-1)^{|x_2|} [\d_v,[\iota_h(x_1),[\iota_h(x_2),\d_h]]]\\
                   &=(-1)^{|x_2|}[[\d_v,\iota_h(x_1)],[\iota_h(x_2),\d_h]]
                   -(-1)^{|x_1|+|x_2|} [\iota_h(x_1),[[\d_v,\iota_h(x_2)],\d_h]]\\
                   &\ \ +(-1)^{|x_1|} [\iota_h(x_1),[\iota_h(x_2),[\d_v,\d_h]]]\\
                   &=(-1)^{|x_2|} [\iota_h(\d'_h  x_1),[\iota_h(x_2),\d_h]]
                    -(-1)^{|x_1|+|x_2|} [\iota_h(x_1),[\iota_h(\d'_h  x_2),\d_h]]\\
                   &\ \  +(-1)^{|x_1|} [\iota_h(x_1),[\iota_h(x_2),[\d_v,\d_h]]]\\
                   &= \iota_h(\Ger{ \d'_h  x_1,x_2})+(-1)^{|x_1|} 
                   \iota_h(\Ger{ x_1,\d'_h  x_2} ) +(-1)^{|x_1|} [\iota_h(x_1),[\iota_h(x_2),[\d_v,\d_h]]].
                    \end{align*}
\end{proof}

We now have all the tools we need to establish Theorem \ref{th:doubleliealgebroid}:

\begin{proof}[Proof of Theorem \ref{th:doubleliealgebroid}]
The equivalence (a) $\Leftrightarrow$ (c) was already established in Lemma \ref{lem:ac}. 
Consider now the implication (b) $\Ra$ (d). Using \eqref{eq:dhdv1} and \eqref{eq:dhdv2}, we see that if $\d'_h $ is a derivation of the Gerstenhaber bracket, then 
\[ \iota_h(x)[\d_h,\d_v]y=0,\ \ \ [\iota_h(x_1),[\iota_h(x_2),[\d_h,\d_v]]]=0\]
for all $x,x_1,x_2\in \W^{\bullet,1}(D')$ and 
$y\in \Gamma(\wedge^\bullet B^*)=\W^{\bullet,0}(D')$. The first equation shows that $[\d_h,\d_v]y=0$, so that 
$[\d_h,\d_v]$ vanishes on $\W^{\bullet,0}(D')$. Using the second equation, and induction on $q$, it then follows that $[\d_h,\d_v]$ vanishes on all $\W^{\bullet,q}(D')$, hence that $\d_h,\d_v$ super-commute. For the reverse implication  (d) $\Ra$ (b), suppose 
$[\d_h,\d_v]=0$.  Equations \eqref{eq:dhdv1} and \eqref{eq:dhdv2} tell us that 
\[ \d'_h \Ger{ x,y}-\Ger{ \d'_h x,y}-(-1)^{|x|}\Ger{ x,\d'_h y}=0,\ \ \ 
\d'_h \Ger{ x_1,x_2}-\Ger{ \d'_h  x_1,x_2}-(-1)^{|x_1|}\,\Ger{ x_1,\d'_h x_2}=0\]
for all $x,x_1,x_2\in \W^{\bullet,1}(D')$ and 
$y\in \W^{\bullet,0}(D')$. This means that $\d'_h $ acts as a derivation of the Gerstenhaber bracket 
on generators, and hence in general. We have thus shown  (b) $\Leftrightarrow$ (d). The equivalence 
(c) $\Leftrightarrow$ (d) follows by applying a flip, which interchanges the horizontal and vertical structures. 
\end{proof}

\subsection{The core of a double Lie algebroid}
It was pointed out in \cite[Section 4]{mac:not} that for any double Lie algebroid $D$, the core $\CC^*$ acquires the structure of a Lie algebroid. This fact may be seen as a consequence of the fact that the base of any Lie bialgebroid has a natural Poisson structure \cite[Proposition 3.6]{mac:lie}. It may also be obtained using the Weil algebras, as follows. 
Recall that $\W(D')$ has a vertical differential and $\W(D'')$ a horizontal differential, which are derivations of the Gerstenhaber brackets on these algebras. 

\begin{proposition}
The core $\CC^*$ of a double Lie algebroid $D$ has a Lie algebroid structure, with  bracket  
given in terms of the identification $\Gamma(\wedge \CC^*)=\W^{\bullet,0}(D'')$ by 
\[ \Lie{\ggamma_1,\ggamma_2}=\Ger{\ggamma_1,\d_v''\ggamma_2},\ \ \ \ 
\L_{\a(\ggamma)}(f)=\Ger{\ggamma,\d_v'' f},
\]
or in terms of the identification $\Gamma(\wedge \CC^*)=\W^{0,\bullet}(D')$ by 
\[ \Lie{\ggamma_1,\ggamma_2}=-\Ger{\ggamma_1,\d_h'\ggamma_2},\ \ \ \ 
\L_{\a(\ggamma)}(f)=-\Ger{\ggamma,\d_h' f}
.\]
\end{proposition}
\begin{proof}
Note that $\Ger{\ggamma_1,\d_v''\ggamma_2}=\Ger{\d_v''\ggamma_1,\ggamma_2}$, which implies skew-symmetry of the bracket $\Lie{\cdot,\cdot}$. The Jacobi identity for $\Lie{\cdot,\cdot}$ follows from that for the Gerstenhaber bracket:
\begin{align*} \Lie{\ggamma_1,\Lie{\ggamma_2,\ggamma_3}}&=
\Ger{\ggamma_1,\d_v''\Ger{\ggamma_2,\d_v''\ggamma_3}}=
\Ger{\ggamma_1,\Ger{\d_v''\ggamma_2,\d_v''\ggamma_3}}\\&=
\Ger{\Ger{\ggamma_1,\d_v''\ggamma_2},\d_v''\ggamma_3}
+\Ger{\d_v''\ggamma_2,\Ger{\ggamma_1,\d_v''\ggamma_3}}=\Lie{\Lie{\ggamma_1,\ggamma_2},\ggamma_3}+\Lie{\ggamma_2,\Lie{\ggamma_1,\ggamma_3}}.
\end{align*}
Furthermore, if $f\in C^\infty(M)$, 
\[ \Ger{\ggamma_1,\d_v''(f\ggamma_2)}=f\Ger{\ggamma_1,\d_v''\ggamma_2}+\Ger{\ggamma_1,\d_v'' f}\ggamma_2,
\]
so that $\L_{\a(\ggamma)}(f)=\Ger{\ggamma,\d_v''f}$ defines an anchor map for this bracket. The expression of the Lie algebroid structure in terms of $D'$ follows from 
\[ \Ger{\d_v''\ggamma_1,\ggamma_2}=\iota_h'(\d_v'' \ggamma_1)\d_h'\ggamma_2 
=\l \d_v'' \ggamma_1,\,\d_h'\ggamma_2 \r_{\CC^*}
=\iota_v''(\d_h'\ggamma_2)\d_v''\ggamma_1
=\Ger{\d_h'\ggamma_2,\ggamma_1}=-\Ger{\ggamma_1,\d_h'\ggamma_2}.\qedhere
\]
\end{proof}
The Lie algebroid bracket $\Lie{\ggamma_1,\ggamma_2}=\Ger{\ggamma_1,\d_v''\ggamma_2}$ 
on $\Gamma(\CC^*)$ extends to a Schouten bracket on $\Gamma(\wedge \CC^*)$, by 
\[ \Lie{\lambda_1,\lambda_2}=\Ger{\lambda_1,\d_v''\lambda_2}.\]

\subsection{Application to Lie algebroids}\label{subsec:TV5} 
Continuing the discussion from Sections \ref{subsec:TV1}, \ref{subsec:TV2}, \ref{subsec:TV3}, \ref{subsec:TV4}, suppose that the vector bundle $V\to M$ carries the structure of a Lie algebroid, $V\Ra M$. Then we obtain $\VB$-algebroid structures $TV\Ra TM$ and $T^*V\Ra V^*$, as well as a 
Poisson structure on $TV^*$, to be described below. Let us discuss the resulting structures on the Weil algebras. 
\subsubsection{Tangent bundle of $V$}
The  $\VB$-algebroid $TV\Ra TM$ is the \emph{tangent prolongation} of the Lie algebroid $V\Ra M$ \cite{cou:lie,mac:lie}: its anchor is 
the tangent map to the anchor of $V$ composed with the canonical involution of $TTM$, and the Lie bracket bracket is such that the tangent lift $\Gamma(V,M)\to \Gamma(TV,TM),\ \sigma\mapsto T\sigma$ is bracket preserving. The resulting  Lie algebroid structure on $\wh{A}=J^1(V)$ is the \emph{jet prolongation} of the Lie algebroid; the bracket is uniquely characterized by $\Lie{j^1(\sigma_1),j^1(\sigma_2)}=j^1(\Lie{\sigma_1,\sigma_2})$, and its representations on  
$B=TM$ and on $\CC^*=V$ are given by 
\[ \nabla_{j^1(\sigma)}^{TM}\mu=\L_{\a(\sigma)}\mu,\ \ \ \ 
\nabla_{j^1(\sigma_1)}^V\sigma_2=\Lie{\sigma_1,\sigma_2}\]
for $\mu\in \Omega^1(M),\ \sigma,\sigma_1,\sigma_2\in \Gamma(V)$. 
See \cite{cra:sec} for a detailed discussion. The invariant 
bilinear pairing $(\cdot,\cdot)\colon B^*\times_M\CC^*\to \R$ is given by the 
anchor $\a\colon V\to TM$. The resulting horizontal differential $\d_h$ on the Weil algebra $\W(TV)$ is uniquely determined by its properties that $\d_h$ agrees with the Lie algebroid differential $\d_{CE}$ on $\W^{\bullet,0}(TV)=\Gamma(\wedge V^*)$ and satisfies $[\d_h,\d_v]=0$, where $\d_v$ was described in \ref{subsec:TV4}.

\subsubsection{Tangent bundle of $V^*$}
The Weil algebra of $(TV)'=\on{flip}((TV^*)^-)$ is $C^\infty(M)$-linearly generated by differentials $\d f\in \Gamma(T^*M)$ (identified with the tangent lifts $f_T$), sections $\sigma\in \Gamma(V)$ 
(identified with vertical lifts $\sigma^v$) and their jets 
$j^1(\sigma)\in \Gamma(J^1(V))$ (identified with $\sigma_T$), subject to the relation $j^1(f\sigma)-f j^1(\sigma)= \d f\ \sigma$. It has a horizontal differential, characterized by $\d_h'(f)=\d f,\ \ \d_h'(\sigma)=j^1(\sigma)$, that is a derivation of the Gerstenhaber bracket. To describe the 
latter, note that the Lie algebroid structure on $V$ defines a linear Poisson structure  \eqref{eq:homogeneity} on $V^*$; its tangent lift is a double-linear Poisson structure on $TV^*$. 
By definition of the tangent lift of Poisson structures, 
\[ \{(\sigma_1)_T,(\sigma_2)_T\}=\Lie{\sigma_1,\sigma_2}_T,\ \ 
\{(\sigma_1)_T,(\sigma_2)_v\}=\Lie{\sigma_1,\sigma_2}_v,\ \ 
\]\[ 
\{\sigma_T,f_v\}=(\L_{\a(\sigma)} f)_v,\ \ 
\{\sigma_T,f_T\}=(\L_{\a(\sigma)} f)_T,\ \ 
\{\sigma_v,f_T\}=(\L_{\a(\sigma)} f)_v.
\]
 We read off the Gerstenhaber brackets as 
\[ \Ger{j^1(\sigma_1),j^1(\sigma_2)}=j^1 (\Lie{\sigma_1,\sigma_2}),\ \ 
 \Ger{j^1(\sigma_1),\sigma_2}=\Lie{\sigma_1,\sigma_2},\]
\[ \Ger{j^1(\sigma),f}=\L_{\a(\sigma)}f,\ \ 
\Ger{j^1(\sigma),\d  f}=\d \L_{\a(\sigma)}f,\ \ 
\Ger{\sigma,\d f}=-\L_{\a(\sigma)}f
.\]
Note that $\d_h'$ is a derivation of the Gerstenhaber bracket, as required.

\subsubsection{Cotangent bundle of $V$} 
The Weil algebra of $(TV)''=\on{flip}(T^*V)^-\cong T^*V^*$ (using the 
Mackenzie-Xu isomorphism \eqref{eq:mackenziexu}) is  $C^\infty(M)$-linearly generated by sections of $V,\ V^*,\ \on{At}(V)$,  subject to the relation that for $\sigma\in \Gamma(V)$, $\tau\in \Gamma(V^*)$
the product 
$\sigma\tau$ equals $i_{\on{At}(V)}(\sigma\otimes \tau)$. From the Poisson bracket relations between the corresponding functions $\phi_{\sigma^\sharp},\ \tau,\ \phi_\delta$ 
on $T^*V$ (see Section \ref{subsec:TV1}), we read off the Gerstenhaber brackets 
\[ \Ger{\delta_1,\delta_2}=\Lie{\delta_1,\delta_2},\ \ \Ger{\delta,\sigma}=\nabla_\delta\sigma,\ \ 
\Ger{\delta,\tau}=\nabla_\delta^*\tau,\ \ 
 \Ger{\delta,f}=\L_{\a(\delta}f,\ \ \ \Ger{\sigma,\tau}=-\l\tau,\sigma\r.\]
On the other hand, the Lie algebroid structure on $V$ determines a $\VB$-algebroid structure 
$T^*V^*\Ra V^*$, and hence vertical differential $\d_v''$. The latter agrees with the Chevalley-Eilenberg differential on $\W^{0,\bullet}(T^*V^*)=\Gamma(\wedge^\bullet V^*)$, while 
\[\d_v''\sigma=-\delta(\sigma)\in \Gamma(\on{At}(V))\]
where $\delta(\sigma)\in \Gamma(\on{At}(V))$ is the infinitesimal automorphism 
given in terms of its representation on $V$ by $\nabla_{\delta(\sigma_1)}(\sigma_2)=\Lie{\sigma_1,\sigma_2}$. 
This follows from  the formulas of Theorem \ref{th:uniquedifferential}:
\[ \iota_v''(j^1(\sigma_1))\d_v''\sigma_2=\nabla^{V}_{j^1(\sigma_1)}\sigma_2=\Lie{\sigma_1,\sigma_2}
=-\nabla_{\delta(\sigma_2)}\sigma_1
=-\iota_v''(j^1(\sigma_1))\delta(\sigma_2).\]
Finally, for $\delta\in\Gamma(\on{At}(V))$ the differential $\d_v''\delta\in \W^{1,2}((TV)'')$ is described by the formula 
\[  \iota_v''(j^1(\sigma_1))\iota_v''(j^1(\sigma_2))\d_v'\delta=
\nabla_\delta \Lie{\sigma_1,\sigma_2}
-\Lie{\sigma_1,\nabla_\delta\sigma_2}+\Lie{\nabla_\delta\sigma_1,\sigma_2}, \]
which may be deduced from Corollary \ref{cor:cor}. One finds that 
the differential on $\W^{1,\bullet}(T^*V^*)\cong \Gamma_\lin(\wedge^\bullet _{TM} (TV^*),\,TM)$
coincides with the restriction of the Chevalley-Eilenberg differential of the tangent prolongation $TV\Ra TM$. (Recall that the dual of the tangent prolongation is the bundle 
$TV^*\to TM$.)

\section{Applications, connections with other work}\label{sec:Applications}
In this section, we indicate connections between the results and constructions presented above and various ideas appearing in the literature.
\subsection{Matched pairs of Lie algebroids}\label{subsec:matched}
Consider a matched pair of Lie algebroids $A,B$, corresponding to a vacant 
double Lie algebroid $D=A\times_M B$, as in Example \ref{ex:doublelie}\eqref{it:matchepair}. 
Thus 
\[ \xymatrix{ D=A\times_M B \ar@{=>}[r] \ar@{=>}[d] & B\ar@{=>}[d]\\
A	\ar@{=>}[r] & M}\ \ \ \ \  
 \xymatrix{ D'=B\times_M A^* \ar@{=>}[r] \ar[d] & M\ar[d]\\
B	\ar@{=>}[r] & M}\ \ \ \ \  
	\xymatrix{ D''=A\times_M B^* \ar[r] \ar@{=>}[d] & 	A\ar@{=>}[d]\\
M	\ar[r] & M }
	\]
with corresponding Weil algebra bundles 
\[  W(D)=\wedge A^*\otimes \wedge B^*,\ \ \ \ \ 
W(D')=\wedge B^*\otimes \vee A,\ \ \ \ \ W(D'')=\wedge A^*\otimes \vee B.\] 
The double Lie algebroid structure defines commuting differentials $\d_h,\ \d_v$ on $\W(D)$.  This double complex was described in the work of Laurent-Gengoux, Stienon and Xu \cite[Section 4.2]{lau:hol}. 
Identifying $W(D)=\wedge (A\oplus B)^*$ (with the total grading), the sum $\d_h+\d_v$ is a degree $1$ differential, in such a way that the bundle maps 
to $\wedge A^*,\wedge B^*$ give cochain maps on sections.
We hence see  that $A\oplus B$ becomes a Lie algebroid, with $A,B$ as Lie subalgebroids.  The Weil algebra $\W(D')$ has a Gerstenhaber bracket $\Ger{\cdot,\cdot}$ and a compatible horizontal 
differential $\d'_h$. The restriction of the differential to $\W^{\bullet,0}(D')=\Gamma(\wedge^\bullet B^*)$ gives the Lie algebroid structure on $B$, and the restriction to $\W^{\bullet,1}(D')=\Gamma(\wedge^{\bullet-1} B^*\otimes A)$ gives the action of this Lie algebroid on $A$. On the other hand, the restriction of the Gerstenhaber bracket to  $\W^{1,1}(D')=\Gamma(A)$ recovers the Lie algebroid bracket of $A$, and the bracket with elements of $\W^{1,0}(D')=\Gamma(B^*)$ recovers 
the $A$-action on $B^*$.  The fact that the differential $\d'_h $ on  $\Gamma(\wedge B^*\otimes A)$
is a derivation of the (Gerstenhaber) bracket on this space is thus an equivalent formulation of the compatibility condition. A similar discussion applies to $D''$. 

\subsection{Multi-derivations}\label{subsec:multiderivation}
In \cite{cra:def}, motivated by the study of deformations of Lie algebroids, Crainic and Moerdijk associate to any vector bundle $V \to M$  a graded vector space $\on{Der}^\bullet(V)$ of  \emph{multi-derivations of $V$}, equipped with a Gerstenhaber bracket. 
Its simplest description  is in terms of the isomorphism  (see \cite[Section 4.9]{cra:def}) 
\[  \on{Der}^{\bullet}(V) \cong \mf X_{\on{lin}}^{1+\bullet} (V^*),\]
with bracket the usual Schouten bracket of multivector fields. A Lie algebroid structure on $V$ defines a compatible degree $1$ differential on this space, given by Schouten bracket $\Lie{\pi,\cdot}$ with the Poisson bivector field $\pi\in\mf{X}^2_{\on{lin}}(V^*)$ 
dual to the Lie algebroid structure. As shown in \cite{cra:def}, the Maurer-Cartan elements of this \emph{deformation complex} describe the deformations of the Lie algebroid structure. See also the work of Esposito-Tortorella-Vitagliano \cite{esp:inf}, where the deformation complex is generalized further to the setting of $\VB$-algebroids.
For any vector bundle $V\to M$, the cotangent bundle $T^*V$ is a Poisson double vector bundle, hence $\W(T^*V)$ inherits a Gerstenhaber bracket. By  Proposition \ref{prop:weilident} and Example \ref{ex:TV}, the isomorphism 
\[ \W^{\bullet,1}(T^*V)\cong \Gamma_\lin(\wedge^\bullet_{V^*} TV^*,V^*)\cong \mf X_{\on{lin}}^{\bullet} (V^*)\]
takes this to the Schouten bracket of multivector fields. Given a Lie algebroid structure on $V$, the resulting horizontal differential on $\W^{\bullet,1}(T^*V)$ is $\d_h=\Ger{\pi,\cdot}$, 
which is the differential on the deformation complex. In conclusion, the deformation complex is identified with $\W^{1+\bullet,1}(T^*V)$. 
Alternatively, this follows from the fact (Remark \ref{rem:cabdru})
that $\W^{\bullet,1}(T^*V)$ for 
$T^*V\Ra V^*$ is the linear Chevalley-Eilenberg complex $\mathsf{CE}_{\VB}^\bullet(T^*V)$ 
of Cabrera-Drummond \cite{cab:hom}; the isomorphism of the latter with the deformation complex was observed in  \cite[page 312]{cab:hom}.

\subsection{Abad-Crainic's Weil algebra of a vector bundle $V$}
 As already mentioned, the Weil algebra  for a vector bundle $V\to M$ was first discussed  in the work of Mehta \cite[Section 4.2.4]{meh:sup} in super-geometric terms. Our Weil algebra $\W(TV)$  is a 
 classical description of this algebra.  On the other hand, Abad and Crainic \cite{aba:wei} defined a  bigraded algebra $\mf{W}(V)$, also called \emph{Weil algebra}, as follows. An element of $\mf{W}^{p,q}(V)$ is given by a sequence of $\R$-multilinear skew-symmetric maps 
 \[ c_i\colon \underbrace{\Gamma(V)\times \cdots \times \Gamma(V)}_{p-i\ \ \on{times}}\to \Omega^{q-i}(M,\vee^i V^*).\]
 Here $c_0$ is considered the `leading term', $c_1$ measures the failure of $c_0$ to be multi-linear, 
 $c_2$ measures the failure of $c_1$ to be multi-linear, and so on. (See \cite{aba:wei} for details.) 
 We claim that every $w\in \W^{p,q}(TV)$ gives rise to such a sequence, thereby identifying 
 $\mf{W}(V)\cong \W(TV)$.
 
 To see this, note that for any double vector bundle $D$ as in \eqref{eq:dvb}, there is a canonical surjective morphism of bigraded algebra bundles $\Pi\colon W(D)\to W(B\times_M \CC^*)=\wedge B^*\otimes \vee \CC$, induced by the $\DVB$ morphism  $B\times_M \CC^*\hra D$. 
Explicitly, the maps $\Pi\colon W^{p,q}(D)\to \wedge^{q-p}B^*\otimes \vee^p \CC$ are given by 
$p$-fold contractions with elements $\ggamma\in \CC^*$.  In the case of $D=TV$, with $A=V,\ B=TM,\ \CC=V^*$, we obtain projection maps 
 \[ \Pi \colon W^{p,q}(TV)\to \wedge^{q-p}T^*M\otimes \vee^p V.\]

 
 On the other hand, for  $\sigma \in \Gamma(V)$, its jet prolongation $j^1(\sigma)\in \Gamma(J^1(V))=\Gamma(\wh{A})$ defines a contraction operator $\iota(j^1(\sigma))$ of bidegree $(-1,0)$. The map $c_i$ corresponding to $w\in \W^{p,q}(TV)$ is given by 
 \[ c_i(\sigma_{i+1},\ldots,\sigma_p)=\Pi\Big(\iota(j^1(\sigma_p))\cdots \iota(j^1(\sigma_i))\, w\Big)
\in \Omega^{q-i}(M,\vee^i V^*).\]
 The relation  $j^1(f\sigma)-f j^1(\sigma)=i_{J^1(V)}(\d f\otimes \sigma)=
 \sigma \d f $ 
  implies that 
 $\iota\big(j^1(f\sigma)\big)-f\ \iota\big(j^1(\sigma)\big)=-\d f\ \iota(\sigma)$, where $\iota(\sigma)$ is the contraction operator of bidegree $(-1,-1)$ defined by $\sigma$ (regarded as a section of $\CC^*$). Consequently, the failure of $C^\infty(M)$-linearity of $c_i$ is expressed in terms of $c_{i+1}$, leading to the conditions in \cite{aba:wei}.

\subsection{IM-forms and IM-multivector fields} 
Let $V\to M$ be a vector bundle. In Section \ref{subsec:TV3}, we discussed the spaces 
$\mf{X}_\lin^\bullet(V)$ of linear multivector fields and $\Omega_\lin^\bullet(V)$ of linear differential forms on $V$. The Schouten bracket of multivector fields and the de Rham differential of forms restrict to these linear subspaces. 

Given a Lie algebroid structure $V\Ra M$, there are notions of \emph{infinitesimally multiplicative}
(IM) multi-vector fields and differential forms, 
\[ \mf{X}_{\IM}^\bullet(V)\subset \mf{X}_\lin^\bullet(V),\ \ \ \ \Omega^\bullet_{\IM}(V)\subset \Omega^\bullet_\lin(V).\]
These are designed to be the infinitesimal versions of multiplicative multivector fields or forms  on Lie groupoids. 

IM-multivector fields were introduced by Iglesias-Ponte, Laurent-Gengoux and Xu \cite{igl:uni} under the name of \emph{$k$-differentials}. To define them, note that for any vector bundle $V$, the graded Lie algebra $\mf{X}^{1+\bullet}(V)$ acts on $\Gamma(\wedge V)$ by derivations. Using the identification 
$\Gamma(\wedge V)\cong \mf{X}_\core^\bullet(V)$, this action is just the Schouten bracket of multi-vector fields. In particular, for $\delta \in \mf{X}_\lin^k(V)$ and 
$\sigma\in \Gamma(\wedge^l V)$ we have that $\delta.\sigma\in \Gamma(\wedge^{k+l-1} V)$. If $V$ is a Lie algebroid, then the bracket $\Lie{\cdot,\cdot}$ on $\Gamma(V)$ extends to the exterior algebra. The element $\delta$ is called an \emph{IM-multivector field} if it is a derivation of this Lie bracket:
\begin{equation}\label{eq:IMmulti}
\delta.\Lie{ \sigma_1,\sigma_2}=\Lie{ \delta.\sigma_1,\sigma_2}+(-1)^{|\delta|(|\sigma_1|+1)}
\Lie{ \sigma_1,\delta.\sigma_2} \end{equation}
for all $\sigma_i\in \Gamma(\wedge^{l_i}V)$, $i=1,2$. Here $|\delta|=k+1$. 
Using the derivation property with respect to wedge product, it is actually enough to have this condition for $l_1=l_2=1$.  The  \emph{universal lifting theorem} \cite{igl:uni}, generalizing earlier results of Mackenzie-Xu \cite{mac:lif,mac:int}, integrates any such $\delta$ to a multiplicative vector field on the corresponding (local) Lie groupoid. 

To describe the IM condition for forms, recall (cf.~ Example \ref{ex:linforms}) that any $\phi\in \Omega^k_\lin(V)$ can be uniquely written as $\phi=\nu+\d_{\on{Rh}}\mu$ where 
$\nu\in \Gamma(V^*\otimes \wedge^k T^*M)$ and $\mu\in \Gamma(V^*\otimes \wedge^{k-1} T^*M)$
(viewed as linear differential forms), and 
$\d_{\on{Rh}}$ denotes the de Rham differential on linear forms.  If $V$ is a Lie algebroid, then 
$\phi=\nu+\d_{\on{Rh}}\mu$ is called an \emph{IM form} if the following three conditions
are satisfied
\begin{align} \iota_{\a(\sigma_1)}\mu(\sigma_2)&=-\iota_{\a(\sigma_2)}\mu(\sigma_1),
\label{eq:IMform1}
\\
\mu(\Lie{ \sigma_1,\sigma_2})&= \ca{L}_{\a(\sigma_1)}\mu(\sigma_2)
-\iota_{\a(\sigma_2)}\d_{\on{\on{Rh}}}\  \mu(\sigma_1)-\iota_{\a(\sigma_2)}\nu(\sigma_1), \label{eq:IMform2}\\ 
\nu(\Lie{ \sigma_1,\sigma_2})&= \ca{L}_{\a(\sigma_1)}\nu(\sigma_2)
- \iota_{\a(\sigma_2)}\d_{\on{\on{Rh}}}\ \nu(\sigma_1), \label{eq:IMform3}\end{align}
for all $\sigma_1,\sigma_2\in \Gamma(V)$. These conditions are due to Bursztyn and Cabrera \cite{bur:mul} (see \cite{bur:lin,bur:int} for the case of 2-forms);  as shown in  \cite{bur:mul}, these are exactly the conditions ensuring that $\phi$ integrates to a multiplicative form  on the associated (local) Lie groupoid. 

We will now give  interpretations of IM multivector fields and IM forms  in terms of the Weil algebras.  Recall from \ref{subsec:TV3} that 
for any vector bundle $V\to M$, 
\[ \mf{X}_\lin^\bullet(V)=\W^{1,\bullet}(T^*V),\ \ \ \Omega_\lin^\bullet(V)=\W^{1,\bullet}(TV).\]
The first isomorphism is compatible with the Gerstenhaber bracket $\Ger{\cdot,\cdot}$ on $\W(T^*V)$ defined by the Poisson structure on $T^*V$, the second isomorphism with the vertical differential $\d_v$ on $\W(TV)$ defined by the $\VB$-algebroid structure $TV\Ra V$.   A Lie algebroid structure $V\Ra M$ gives $\VB$-algebroid structures $T^*V\Ra V^*$ and $TV\Ra TM$, resulting in 
horizontal differentials $\d_h$ on both $\W(T^*V)$ and $\W(TV)$. The second part 
of the following result is  due to Bursztyn and Cabrera \cite{bur:mul}; for a more restrictive notion of 
IM forms it was observed in \cite[Section 6]{aba:wei}.
\begin{theorem}
For any Lie algebroid $V\Ra M$, 
\[ \mf{X}_\IM^\bullet(V)=\W^{1,\bullet}(T^*V)\cap \ker(\d_h),\ \ \ \Omega_\IM^\bullet(V)=\W^{1,\bullet}(TV)\cap \ker(\d_h).\]
\end{theorem}
\begin{proof}
Consider a $\VB$-algebroid $D\Ra B$ over $A\Ra M$, so that the Weil algebra $\W(D)$ carries a horizontal differential $\d_h$. Then $\W(D')$ has a Gerstenhaber bracket. By Corollary \ref{cor:cor}, an element $\phi\in \W^{1,\bullet}(D)$ is $\d_h$-closed if and only if 
\begin{equation}\label{eq:iff}
 \iota_h(\Ger{x_1,x_2})\phi=
  \Ger{x_1,\iota_h(x_2)\phi}+(-1)^{|x_1| |x_2|+1}
    \Ger{x_2,\iota_h(x_1)\phi}\end{equation}
for all $x_i\in \W^{p_i,1}(D')$.  (It suffices to verify this on generators.)

Suppose now that $D$ also carries a double-linear Poisson structure; thus $D''$ is a double Lie algebroid.  
In particular $B^*=\core(D'')$ is a Lie algebroid, with the bracket 
\[ \Lie{\sigma_1,\sigma_2}:=
\Ger{\sigma_1,\d_v'\sigma_2}\] for $\sigma_1,\sigma_2\in \Gamma(\wedge B^*)$. 
(We have in mind the case $D=T^*V$; the Lie algebroid structure on $B^*=V$ being the standard one.) 
The space $\W^{1,\bullet}(D)[1]$ is a graded Lie algebra for the Gerstenhaber bracket, with a 
representation on $\Gamma(\wedge B^*)$ given by (cf.~ \eqref{eq:theotherway})
\[ \phi.\sigma:=\Ger{\phi,\sigma}=-(-1)^{(|\phi|+1)\,|\sigma|}\iota_h(\d_v'\sigma)\phi\]
for $\sigma\in \Gamma(\wedge B^*)$. Let $x_i=\d_v'\sigma_i$ with $\sigma_i\in \Gamma(\wedge B^*)$. 
Then $\Ger{x_1,x_2}=\d_v'\Lie{\sigma_1,\sigma_2}$, and therefore
\begin{equation}\label{eq:lhs}
 \phi.\Lie{\sigma_1,\sigma_2}
=-(-1)^{(|\phi|+1)(|\sigma_1|+|\sigma_2|+1)} \iota_h(\Ger{x_1,x_2})\phi.\end{equation}
On the other hand, 
\begin{align*}
\lefteqn{(-1)^{|\phi|(|\sigma_1|+1)} 
          \Lie{\sigma_1,\phi.\sigma_2}+
\Lie{\phi.\sigma_1,\sigma_2}}
          \\ 
          &=-(-1)^{(|\phi|+1)(|\sigma_1|+|\sigma_2|+1)}
               \Big(\Ger{x_1,\iota_h(x_2)\phi}+(-1)^{(|\phi|+1)( |\sigma_2|+1)}
                    \Ger{\iota_h(x_1)\phi,x_2}\Big)\\
                    &=-(-1)^{(|\phi|+1)(|\sigma_1|+|\sigma_2|+1)}
                    \Big(\Ger{x_1,\iota_h(x_2)\phi}-(-1)^{(|\sigma_1|+1)( |\sigma_2|+1)}
                         \Ger{x_2,\iota_h(x_1)\phi}\Big).
\end{align*}
By \eqref{eq:iff}, if $\phi$ is $\d_h$-closed then this expression coincides with \eqref{eq:lhs}, proving that 
$\phi.\sigma=\Lie{\phi.\sigma_1,\sigma_2}+(-1)^{|\phi|(|\sigma_1|+1)} \Lie{\sigma_1,\phi.\sigma_2}$. 
For $D=T^*V$, the converse is true, because in that case the space $\W^{\bullet,1}(D')$ is spanned, 
as a $C^\infty(M)$-module, by $\d_v' \W^{\bullet,0}(D')$. The case of IM-differential forms can be discussed similarly; in terms of the Abad-Crainic approach to the Weil algebra $\W(TV)$ this is done in \cite{bur:mul}. 
\end{proof}

\subsection{Fr\"olicher-Nijenhuis and Nijenhuis-Richardson brackets}
Suppose $V\to M$ is a Lie algebroid, so that $V^*$ has a linear Poisson structure. The 
double-linear Poisson structure on $TV^*$ defines a Gerstenhaber bracket on $\W(TV^*)$,  
compatible with the  vertical differential $\d^v$. Hence, 
\begin{equation}\label{eq:nijr} 
\W^{1,1+\bullet}(TV^*)\cong \Omega^{\bullet+1}_\lin(V^*)
\cong \Omega^{\bullet+1}(M,V)\oplus \Omega^{\bullet}(M,V)
\end{equation}
becomes a differential graded Lie algebra. It comes with a morphism of graded Lie algebras
(also $\Omega(M)$-module morphism)
\begin{equation}\label{eq:nijr2}
\Omega^{\bullet+1}(M,V)\oplus \Omega^{\bullet}(M,V)\to 
\on{Der}^\bullet(\Omega(M))
\end{equation}
given by Gerstenhaber bracket with elements of  $\W^{0,\bullet}(TV^*)\cong \Omega(M)$. 
One verifies that the first summand in \eqref{eq:nijr2} acts as contractions via the anchor 
$V\to TM$, the second as Lie derivatives. 

For the Lie algebroid $V=TM$, 
the map \eqref{eq:nijr2} is an isomorphism, hence we recover the bracket on $\Omega^{\bullet+1}(M,TM)\oplus \Omega^\bullet(M,TM)$ given by the 
Fr\"olicher-Nijenhuis bracket on the first summand, the Nijenhuis-Richardson bracket on the second summand, and a cross term. See  \cite[Chapter II.8]{kol:nat} for a detailed discussion; see also  \cite{gra:tan,nij:vec} for related brackets and generalizations to Lie algebroids. 

\subsection{Representations up to homotopy}\label{subsec:RUTH}

Representations up to homotopy were introduced by Evens-Lu-Weinstein \cite{ev:tra} and  Abad and Crainic \cite{aba:rep} as generalizations of the usual concept of representations of a Lie algebroid.  Among other things,  they give a notion of the adjoint action of a Lie algebroid on itself, which is generally not possible using ordinary representations.  The essential idea is to represent Lie algebroids on complexes of vector bundles rather than just single vector bundles. We will adopt the definition from \cite{aba:rep}.
\begin{definition}\label{def:RUTH} 
Let $A\to M$ be a Lie algebroid.  A \emph{representation up to homotopy} of $A$ is a $\Z$-graded vector bundle $\mathsf{U}_\bullet$ over $M$ along with a degree 1 differential $\delta$ on sections of 
$\wedge A^*\otimes \mathsf{U}$ (using the graded tensor product) satisfying 
\[ \delta(\omega \eta)=d_A(\omega) \eta + (-1)^k \omega \delta(\eta)\] 
for all $\omega \in \Gamma(\wedge^k A^*)$, $\eta \in \Gamma(\wedge A^*\otimes \mathsf{U})$. 
\end{definition}
Given a Lie algebroid $A\to M$, Gracia-Saz and Mehta \cite{gra:vba} showed how 
to construct a 2-step representation up to homotopy of $A$ from a horizontal $\mathcal{VB}$-algebroid 
\[
 \xymatrix{ {D} \ar@{=>}[r]\ar[d] & B\ar[d]\\
A\ar@{=>}[r] & M}
\]
having $A$ as its horizontal side bundle. The construction depends on the choice of a splitting of $D$, 
and the resulting graded vector bundle is 
\[ \mathsf{U}=\CC^*[1]\oplus B;\]
that is, $\mathsf{U}_{-1}=\CC^*$ and $\mathsf{U}_0=B$. We briefly review their construction, making use of some of our observations in Section \ref{sec:PDVB}.   By Lemma \ref{lem:reps}, the double-linear Poisson structure on $D'$ gives the following data:
\begin{itemize}
\item a Lie algebroid structure on $\wh A$,
\item $\wh{A}$-representations $\wh \nabla^{B^*},\ \wh \nabla^{\CC^*}$ on $B^*$ and on $\CC^*$,
\item an invariant pairing $(\cdot,\cdot)\colon B^*\times \CC^* \to \R$.
\end{itemize}
A choice of splitting of the double vector bundle $D$ is equivalent to a choice of splitting $s\colon A \to \wh A$. 
In general, $s$ 
need not preserve Lie brackets, and so we can consider its curvature tensor $\Omega\in \Gamma(\wedge^2 A^* \otimes (B^*\otimes \CC^*))$ defined by 
\[
\Omega(a_1,a_2)=s(\Lie{ a_1,a_2})-[s(a_1),s(a_2)],\ \ \ a_1,a_2\in \Gamma(A).
\]
The $\mathsf{U}_0$--$\mathsf{U}_{-1}$-component of $\delta$ is the linear map $\Gamma(\wedge^\bullet A^* \otimes B) \to \Gamma(\wedge^{\bullet+2} A^*\otimes \CC^*)$ given by  $\Omega$ (with the identification 
$B^*\otimes \CC^*\cong \Hom(B,\CC^*)$). The $\mathsf{U}_{-1}$--$\mathsf{U}_0$-component of $\delta$ is the linear 
map  $\Gamma(\wedge^{\bullet}A^*\otimes \CC^*) \to \Gamma(\wedge^{\bullet}A^* \otimes B)$ 
defined by the pairing $(\cdot,\cdot)$ viewed as a bundle map $\CC^*\to B$. The connection $\wh \nabla^{\CC^*}$ pulls back under $s$ to a non-flat $A$-connection $\nabla^{\CC^*}$ on $\CC^*$; its extension to a map $\nabla^{\CC^*}\colon \Gamma(\wedge^\bullet A^*\otimes \CC^*)\to 
\Gamma(\wedge^{\bullet+1} A^*\otimes \CC^*)$ is the $\mathsf{U}_{-1}$--$\mathsf{U}_{-1}$-component of $\delta$. 
Similarly, the flat $\wh{A}$-connection on $B$ pulls back to a non-flat $A$-connection, and the resulting map on sections gives the $\mathsf{U}_0$--$\mathsf{U}_0$-component. See  \cite[Theorem 4.10]{gra:vba}.  This establishes a one-to-one correspondence:
\begin{theorem}\cite[Theorem 4.11 (2)]{gra:vba}\label{th:ruth}
Let $D$ be a double vector bundle with side bundles $A,\,B$ and core $\CC^*$, such that $A\Ra M$ is a Lie algebroid.  After choosing a splitting $s\colon A \to \hat A$, there is a one-to-one correspondence between horizontal $\mathcal{VB}$-algebroid structures $D\Ra B$ extending $A\Ra M$, 
and representations up to homotopy of $A$ on $\mathsf{U}=\CC^*[1]\oplus B$. 
\end{theorem}

This correspondence  has a simple interpretation in terms of the Weil algebras.  Recall from Section \ref{sec:PDVB} that horizontal $\mathcal{VB}$-algebroid structures $D\Ra B$ are in one-to-one correspondence with vertical differentials $\d_v$ on $\W(D'')$.  This restricts to a 
differential $\d_v$ on $\W^{1,\bullet}(D'')\cong \Gamma_\lin(\wedge_\CC^\bullet  D',\CC)$.  Once we choose a 
splitting of $D$, or equivalently a vector bundle splitting $s\colon A\to \wh{A}$, we obtain a decomposition (see Proposition \ref{prop:linses})
\[
\W^{1,\bullet}(D'')\cong \Gamma(\wedge^\bullet A^* \otimes B)\oplus \Gamma(\wedge^{\bullet+1} A^* \otimes \CC^*).
\]
With $\mathsf{U}_{-1}=\CC^*$ and $\mathsf{U}_0=A^*$, the differential $d$ defines a representation up to homotopy of $A$ on 
$\mathsf{U}$.  To see that this correspondence is bijective, we note that a vertical differential on the bigraded algebra $\W(D'')$ is uniquely determined by its restrictions to $\W^{1,\bullet}(D'')$ and $\W^{0,\bullet}(D'')= \Gamma(\wedge^\bullet A^*)$ thanks to the Leibniz rule.  

\begin{remark}
Horizontal $\mathcal{VB}$-algebroid structures on $D$ are also equivalent to differentials $d_h$ of bidegree $(1,0)$ on $\W(D)$.  After a choice of splitting, this induces a degree 1 operator on $\W^{\bullet,1}(D)\cong \Gamma(\wedge^\bullet A^* \otimes \CC)\oplus \Gamma(\wedge^{\bullet+1} A^* \otimes B^*)$, giving a representation up to homotopy of $A$ on the bundle $B^*[1]\oplus \CC$.  This representation up to homotopy is dual to the one on $\CC^*[1]\oplus B$, as discussed in \cite[Section 4.5]{gra:vba}.  
\end{remark}

Let us turn now to the case where $D$ has both a horizontal and a vertical $\mathcal{VB}$-algebroid structure.  Then both $D'$ and $D''$ are Poisson double vector bundles, which is equivalent to the following structures:
\begin{itemize}
\item Lie algebroid structures on $\wh A$ and $\wh B$,
\item $\wh A$-representations on $B^*$, $\CC^*$ and $\wh B$-representations on $A^*, \CC^*$,
\item an $\wh A$-invariant pairing $B^*\times \CC^* \to \R$ and a $\wh B$-invariant pairing $A^*\times \CC^* \to \R$
\end{itemize}
subject to certain compatibility conditions (Theorem \ref{th:c-liealgebroid}).  It is natural to ask what additional compatibility conditions on these data ensure that $D$ is a double Lie algebroid.  This question was answered in the work of Gracia-Saz, Jotz Lean, Mackenzie, and Mehta 
\cite[Theorem 3.4]{gra:dla} using a splitting and a notion of \emph{matched pair} for representations up to homotopy. 

\subsection{Van Est maps}
Recall that the classical van Est map \cite{vanest:gro} is a morphism from the cochain complex of a Lie group $G$ to the Chevalley-Eilenberg cochain complex of its Lie algebra $\mf g$.  This map was extended by Weinstein and Xu \cite{wei:ext} to the case of Lie groupoids $G$ and their Lie algebroids $A$, thus obtaining a morphism of cochain complexes
\begin{equation}\label{eq:vanest}
\on{VE}\colon \wt{C}^\infty(B_\bullet G) \to \Gamma(\wedge^\bullet A^*).
\end{equation}
Here $B_p G$ is the space of $p$-arrows in $G$, and the tilde signifies the normalized complex for the simplicial manifold $B_\bullet G$.  In \cite[Chapter 6]{meh:sup}, Mehta generalized \eqref{eq:vanest}
to a van Est map from `Q-groupoids' into the double complex of the corresponding `$Q$-algebroid'; 
in particular this gives a version of \eqref{eq:vanest} for differential forms. 
A construction in more classical terms was given by Abad and Crainic \cite{aba:wei}, in terms of the Weil algebra 
$\mf{W}(A)=\W(TA)$. 
The van Est map \eqref{eq:vanest} described in \cite{aba:wei} is a morphism of double complexes
\begin{equation}\label{eq:abadcrainic}
\on{VE}\colon \wt{\Omega}^\bullet(B_\bullet G) \to \W^{\bullet,\bullet}(TA).
\end{equation}
A geometric construction of this map was provided in \cite{lib:ve}. In an upcoming work, we will generalize this construction to the setting of $\LA$-groupoids, 
with Mehta's double complex of the $\LA$-groupoid
\cite{meh:qgr} as the domain and the Weil algebra $\W(D)$ of the corresponding double Lie algebroid $D$ as the codomain.
Here \eqref{eq:abadcrainic} corresponds to the case of the tangent groupoid $TG\rra TM$. 
This is related to work of Cabrera and Drummond in \cite{cab:hom};  as mentioned earlier (Remark \ref{rem:cabdru}), 
if $D$ is a (horizontal) $\VB$-algebroid, then the codomain $\mathsf{CE}_{\VB}^\bullet(D)$
of their 
\emph{van Est map for homogeneous cochains} for $\VB$-groupoids is our  $\W^{1,\bullet}(D)$. Going one step further, we will describe a van Est map for double Lie groupoids $\ca{H}$, a morphism from a certain double complex $\wt{C}^\infty(B_{\bullet,\bullet}\ca{H})$ to the double complex defined by the 
Weil algebra of its associated double Lie algebroid. The recent thesis of Angulo \cite{ang:th} considers similar questions in the context of  Lie 2-algebras and Lie 2-groups. 

\begin{appendix}
	\section{Splitting of double vector bundles}\label{app:technical}
	In this appendix, we give a quick proof of the existence of splittings of double vector bundles  $D$ (as in \eqref{eq:dvb}). Note that if 
	$D$ is regarded as a vector bundle over $A$, then its restriction to $M\subset A$ 
	is 	the vector bundle \[ (\kappa_0^v)^{-1}(M)=B \times_M \CC^*\equiv B\oplus \CC^*\]
	over $M$. 
	\begin{theorem}\label{th:splittings1}
		Every double vector bundle admits a splitting. 
	\end{theorem}
	\begin{proof}
		Regard  $D$ and $A\times_MB$ as vector bundles over $A$; their restrictions to the submanifold $M\subset A$ are canonically $B\oplus \CC^*$ and $B$, respectively. 
		The surjective submersion $\varphi\colon D\to A\times_M B$ from \eqref{eq:varphi}, 
		regarded as a morphism of vector bundles 
		over $A$, restricts along $M$ to the obvious projection $B\oplus \CC^*\to B$. 
		This restriction has a canonical splitting $B\to B \oplus \CC^*, b\mapsto (b,0)$. 
		 Choose any extension to a splitting of vector bundles over $A$,
		\[ \psi_1\colon A\times_M B\to D.\]	Then $\psi_1$ intertwines the vertical scalar multiplications $\kappa_t^v$, but not necessarily the horizontal scalar multiplications. Applying the normal bundle functor, we obtain a $\DVB$ morphism 
		\[ \nu(\psi_1)\colon \nu(A\times_M B,B)\to \nu(D,B).\]
		But recall that for any vector bundle $V\to M$, one has a canonical isomorphism  $TV|_M=V\oplus TM$, 
		giving rise to an isomorphism of vector bundles $\nu(V,M)\cong V$. 
		In a similar fashion, we have  canonical $\DVB$ isomorphisms  $\nu(D,B)\cong D$ and  $\nu(A\times_M B,B)\cong A\times_M B$. Under these identifications, $\nu(\psi_1)=:\psi$ is the desired splitting 
		$A\times_M B\to D$. 
	\end{proof} 
\end{appendix}

\bibliographystyle{amsplain}

\begin{thebibliography}{10}
	
	\bibitem{aba:wei}
	C.~Abad and M.~Crainic, \emph{The {W}eil algebra and the {V}an {E}st
		isomorphism}, Ann. Inst. Fourier (Grenoble) \textbf{61} (2011), no.~3,
	927--970.
	
	\bibitem{aba:rep}
	\bysame, \emph{Representations up to homotopy of {L}ie algebroids}, J. Reine
	Angew. Math. \textbf{663} (2012), 91--126.
	
	\bibitem{ang:th}
	C.~Angulo, \emph{{A cohomology theory for Lie 2-algebras and Lie 2-groups}},
	Ph.D. thesis, Sao Paulo 2018.
	
	\bibitem{bro:det}
	R.~Brown and K.~Mackenzie, \emph{{Determination of a double Lie groupoid by its
			core diagram}}, J. Pure Appl. Algebra \textbf{80} (1992), no.~3, 237--272.
	
	\bibitem{bur:mul}
	H.~Bursztyn and A.~Cabrera, \emph{Multiplicative forms at the infinitesimal
		level}, Math. Ann. \textbf{353} (2012), no.~3, 663--705.
	
	\bibitem{bur:vec}
	H.~Bursztyn, A.~Cabrera, and M.~del Hoyo, \emph{{Vector bundles over Lie
			groupoids and algebroids}}, Adv.~Math. \textbf{290} (2016), no.~2, 163--207.
	
	\bibitem{bur:lin}
	H.~Bursztyn, A.~Cabrera, and C.~Ortiz, \emph{Linear and multiplicative
		2-forms}, Lett. Math. Phys. \textbf{90} (2009), no.~1-3, 59--83.
	
	\bibitem{bur:int}
	H.~Bursztyn, M.~Crainic, A.~Weinstein, and C.~Zhu, \emph{Integration of twisted
		{D}irac brackets}, Duke Math.~J. \textbf{123} (2004), no.~3, 549--607.
	
	\bibitem{cab:hom}
	A.~Cabrera and T.~Drummond, \emph{Van {E}st isomorphism for homogeneous
		cochains}, Pacific J. Math. \textbf{287} (2017), no.~2, 297--336.
	
	\bibitem{cat:sup}
	A.~Cattaneo and F.~Sch\"{a}tz, \emph{Introduction to supergeometry}, Rev. Math.
	Phys. \textbf{23} (2011), no.~6, 669--690. \MR{2819233}
	
	\bibitem{che:omn}
	Z.~Chen and Z.~Liu, \emph{Omni-{L}ie algebroids}, J. Geom. Phys. \textbf{60}
	(2010), no.~5, 799--808.
	
	\bibitem{che:dou}
	Z.~Chen, Z.~Liu, and Y.~Sheng, \emph{On double vector bundles}, Acta Math. Sin.
	(Engl. Ser.) \textbf{30} (2014), no.~10, 1655--1673.
	
	\bibitem{cou:lie}
	T.~Courant, \emph{Tangent {L}ie algebroids}, J. Phys. A \textbf{27} (1994),
	no.~13, 4527--4536. \MR{1294955}
	
	\bibitem{cra:sec}
	M.~Crainic and R.~L. Fernandes, \emph{Secondary characteristic classes of {L}ie
		algebroids}, Quantum field theory and noncommutative geometry, Lecture Notes
	in Phys., vol. 662, Springer, Berlin, 2005, pp.~157--176.
	
	\bibitem{cra:def}
	M.~Crainic and I.~Moerdijk, \emph{Deformations of {L}ie brackets: cohomological
		aspects}, J. Eur. Math. Soc. (JEMS) \textbf{10} (2008), no.~4, 1037--1059.
	
	\bibitem{del:geo}
	F.~del Carpio-Marek, \emph{Geometric structures on degree $2$ manifolds}, Ph.D.
	thesis, IMPA 2015.
	
	\bibitem{esp:inf}
	C.~Esposito, A.G. Tortorella, and L.~Vitagliano, \emph{Infinitesimal
		automorphisms of {VB}-groupoids and algebroids}, Quart.~J.~Math. (2019).
	
	\bibitem{ev:tra}
	S.~Evens, J.-H. Lu, and A.~Weinstein, \emph{Transverse measures, the modular
		class and a cohomology pairing for {L}ie algebroids}, Quart.~J.~Math.~Oxford
	Ser.~(2) \textbf{50} (1999), no.~200, 417--436.
	
	\bibitem{fro:th}
	A.~Fr\"{o}licher and A.~Nijenhuis, \emph{Theory of vector-valued differential
		forms. {I}. {D}erivations of the graded ring of differential forms}, Nederl.
	Akad. Wetensch. Proc. Ser. A. {\bf 59} = Indag. Math. \textbf{18} (1956),
	338--359. \MR{0082554}
	
	\bibitem{gra:dua2}
	J.~Grabowski, M.~J\'{o}\'{z}wikowski, and M.~Rotkiewicz, \emph{Duality for
		graded manifolds}, Rep. Math. Phys. \textbf{80} (2017), no.~1, 115--142.
	\MR{3689898}
	
	\bibitem{gra:hig}
	J.~Grabowski and M.~Rotkiewicz, \emph{Higher vector bundles and multi-graded
		symplectic manifolds}, J. Geom. Phys. \textbf{59} (2009), no.~9, 1285--1305.
	
	\bibitem{gra:tan}
	J.~Grabowski and P.~Urba\'nski, \emph{Tangent and cotangent lifts and graded
		{L}ie algebras associated with {L}ie algebroids}, Ann. Global Anal. Geom.
	\textbf{15} (1997), no.~5, 447--486.
	
	\bibitem{gra:dla}
	A.~Gracia-Saz, M.~Jotz~Lean, K.~C.~H. Mackenzie, and R.~A. Mehta, \emph{Double
		{L}ie algebroids and representations up to homotopy}, J. Homotopy Relat.
	Struct. \textbf{13} (2018), no.~2, 287--319.
	
	\bibitem{gra:dua}
	A.~Gracia-Saz and K.~Mackenzie, \emph{Duality functors for triple vector
		bundles}, Lett. Math. Phys. \textbf{90} (2009), no.~1-3, 175--200.
	
	\bibitem{gra:vba}
	A.~Gracia-Saz and R.~Mehta, \emph{{Lie algebroid structures on double vector
			bundles and representation theory of Lie algebroids}}, Advances in
	Mathematics \textbf{223} (2010), 1236--1275.
	
	\bibitem{jot:mult}
	M.~Heuer and M.~Jotz Lean, \emph{{Multiple vector bundles: cores, splittings
			and decompositions }}, arXiv:1809.01484.
	
	\bibitem{ho:an3}
	L.~H\"ormander, \emph{The analysis of linear partial differential operators
		{III}}, second ed., Grundlehren der mathematischen Wissenschaften, vol. 256,
	Springer-Verlag, Berlin-Heidelberg-New York, 1990.
	
	\bibitem{hue:dif}
	J.~Huebschmann, \emph{Differential {B}atalin-{V}ilkovisky algebras arising from
		twilled {L}ie-{R}inehart algebras}, Poisson geometry ({W}arsaw, 1998), Banach
	Center Publ., vol.~51, Polish Acad. Sci. Inst. Math., Warsaw, 2000,
	pp.~87--102.
	
	\bibitem{igl:uni}
	D.~Iglesias-Ponte, C.~Laurent-Gengoux, and P.~Xu, \emph{Universal lifting
		theorem and quasi-{P}oisson groupoids}, J. Eur. Math. Soc. (JEMS) \textbf{14}
	(2012), no.~3, 681--731.
	
	\bibitem{kol:nat}
	I.~Kol\'a\v{r}, P.~Michor, and J.~Slov\'ak, \emph{Natural operations in
		differential geometry}, Springer-Verlag, Berlin, 1993.
	
	\bibitem{kon:dou}
	K.~Konieczna and P.~Urba\'nski, \emph{Double vector bundles and duality}, Arch.
	Math. (Brno) \textbf{35} (1999), no.~1, 59--95.
	
	\bibitem{kos:jac}
	Y.~Kosmann-Schwarzbach, \emph{Jacobian quasi-bialgebras and quasi-{P}oisson
		{L}ie groups}, Mathematical Aspects of Classical Field Theory (1991, Seattle,
	WA) (M.J. Gotay, J.E. Marsden, and V.~Moncrief, eds.), Contemp.~Math., vol.
	132, 1992, pp.~459--489.
	
	\bibitem{kos:exa}
	\bysame, \emph{Exact {G}erstenhaber algebras and {L}ie bialgebroids}, Acta
	Appl. Math. \textbf{41} (1995), no.~1-3, 153--165, Geometric and algebraic
	structures in differential equations. \MR{1362125 (97i:17021)}
	
	\bibitem{kos:poi1}
	\bysame, \emph{{From Poisson algebras to Gerstenhaber algebras}}, Ann. Inst.
	Fourier \textbf{46} (1996), 1243--1274.
	
	\bibitem{kos:poi}
	Y.~Kosmann-Schwarzbach and F.~Magri, \emph{Poisson-{L}ie groups and complete
		integrability. {I}. {D}rinfel\cprime d bialgebras, dual extensions and their
		canonical representations}, Ann. Inst. H. Poincar\'e Phys. Th\'eor.
	\textbf{49} (1988), no.~4, 433--460.
	
	\bibitem{lan:dou}
	H.~Lang, Y.~Li, and Z.~Liu, \emph{Double principal bundles}, 2016,
	arXiv:1611.00672.
	
	\bibitem{lau:hol}
	C.~Laurent-Gengoux, M.~Sti\'{e}non, and P.~Xu, \emph{Holomorphic {P}oisson
		manifolds and holomorphic {L}ie algebroids}, Int. Math. Res. Not. IMRN
	(2008), Art. ID rnn 088, 46. \MR{2439547}
	
	\bibitem{jot:lie2}
	M.~Jotz Lean, \emph{{Lie 2-algebroids and matched pairs of 2-representations -
			a geometric approach}}, {Pacific J. Math.}, {to appear, arXiv:1712.07035}.
	
	\bibitem{lib:ve}
	D.~Li-Bland and E.~Meinrenken, \emph{On the van {E}st homomorphism for {L}ie
		groupoids}, Enseign. Math. \textbf{61} (2015), no.~1-2, 93--137.
	
	\bibitem{lib:qua}
	D.~Li-Bland and P.~\v{S}evera, \emph{Quasi-{H}amiltonian groupoids and
		multiplicative {M}anin pairs}, International Mathematics Research Notices
	\textbf{2011} (2011), 2295--2350.
	
	\bibitem{lu:poi}
	J.-H. Lu, \emph{Poisson homogeneous spaces and {L}ie algebroids associated to
		{P}oisson actions}, Duke Math. J. \textbf{86} (1997), no.~2, 261--304.
	
	\bibitem{lu:po}
	J.-H. Lu and A.~Weinstein, \emph{Poisson {L}ie groups, dressing
		transformations, and {B}ruhat decompositions}, J.~Differential Geom.
	\textbf{31} (1990), no.~2, 501--526.
	
	\bibitem{mac:dou2}
	K.~Mackenzie, \emph{{Double Lie algebroids and the double of a Lie
			bialgebroid}}, arXiv:math/9808081.
	
	\bibitem{mac:not}
	\bysame, \emph{Notions of double for {L}ie algebroids}, arXiv:math/0011212.
	
	\bibitem{mac:dou}
	\bysame, \emph{Double {L}ie algebroids and second-order geometry.~{I}},
	Adv.~Math. \textbf{94} (1992), no.~2, 180--239.
	
	\bibitem{mac:sym}
	\bysame, \emph{On symplectic double groupoids and the duality of {P}oisson
		groupoids}, Internat. J. Math. \textbf{10} (1999), no.~4, 435--456.
	
	\bibitem{mac:dou1}
	\bysame, \emph{Double {L}ie algebroids and second-order geometry. {II}}, Adv.
	Math. \textbf{154} (2000), no.~1, 46--75.
	
	\bibitem{mac:gen}
	\bysame, \emph{General theory of {L}ie groupoids and {L}ie algebroids}, London
	Mathematical Society Lecture Note Series, vol. 213, Cambridge University
	Press, Cambridge, 2005.
	
	\bibitem{mac:ehr}
	\bysame, \emph{Ehresmann doubles and {D}rinfel'd doubles for {L}ie algebroids
		and {L}ie bialgebroids}, J. Reine Angew. Math. \textbf{658} (2011), 193--245.
	
	\bibitem{mac:jac}
	\bysame, \emph{Proving the {J}acobi identity the hard way}, Geometric methods
	in physics, Trends Math., Birkh\"{a}user/Springer, Basel, 2013, pp.~357--366.
	\MR{3364056}
	
	\bibitem{mac:lie}
	K.~Mackenzie and P.~Xu, \emph{Lie bialgebroids and {P}oisson groupoids}, Duke
	Math.~J. \textbf{73} (1994), no.~2, 415--452.
	
	\bibitem{mac:lif}
	\bysame, \emph{Classical lifting processes and multiplicative vector fields},
	Quart. J. Math. Oxford Ser. (2) \textbf{49} (1998), no.~193, 59--85.
	
	\bibitem{mac:int}
	\bysame, \emph{Integration of {L}ie bialgebroids}, Topology \textbf{39} (2000),
	no.~3, 445--467.
	
	\bibitem{maj:mat}
	S.~Majid, \emph{Matched pairs of {L}ie groups associated to solutions of the
		{Y}ang-{B}axter equations}, Pacific J. Math. \textbf{141} (1990), no.~2,
	311--332.
	
	\bibitem{meh:sup}
	R.~Mehta, \emph{Supergroupoids, double structures, and equivariant cohomology},
	2006, Thesis (Ph.D.)--University of California, Berkeley.
	
	\bibitem{meh:qgr}
	R.~Mehta, \emph{{$Q$}-groupoids and their cohomology}, Pacific J. Math.
	\textbf{242} (2009), no.~2, 311--332.
	
	\bibitem{mok:mat}
	T.~Mokri, \emph{Matched pairs of {L}ie algebroids}, Glasgow Math. J.
	\textbf{39} (1997), no.~2, 167--181.
	
	\bibitem{nij:vec}
	A.~Nijenhuis, \emph{Vector form brackets in {L}ie algebroids}, Arch. Math.
	(Brno) \textbf{32} (1996), no.~4, 317--323. \MR{1441402}
	
	\bibitem{nij:def}
	A.~Nijenhuis and R.~W. Richardson, \emph{Deformations of {L}ie algebra
		structures}, J. Math. Mech. \textbf{17} (1967), 89--105. \MR{0214636}
	
	\bibitem{pra:rep}
	J.~Pradines, \emph{Repr\'{e}sentation des jets non holonomes par des morphismes
		vectoriels doubles soud\'{e}s}, C. R. Acad. Sci. Paris S\'{e}r. A
	\textbf{278} (1974), 1523--1526. \MR{0388432}
	
	\bibitem{pra:fib}
	\bysame, \emph{Fibres vectoriels doubles et calcul des jets non holonomes},
	Esquisses Math\'ematiques [Mathematical Sketches], vol.~29, Universit\'e
	d'Amiens, U.E.R. de Math\'ematiques, Amiens, 1977.
	
	\bibitem{vai:lie}
	A.~Va{\u\i}ntrob, \emph{Lie algebroids and homological vector fields}, Uspekhi
	Mat. Nauk \textbf{52} (1997), no.~2(314), 161--162.
	
	\bibitem{vanest:gro}
	W.~T. van Est, \emph{Group cohomology and {L}ie algebra cohomology in {L}ie
		groups. {I}, {II}}, Nederl. Akad. Wetensch. Proc. Ser. A. {\bf 56} =
	Indagationes Math. \textbf{15} (1953), 484--492, 493--504.
	
	\bibitem{vor:q}
	T.~Voronov, \emph{{$Q$}-manifolds and {M}ackenzie theory}, Comm. Math. Phys.
	\textbf{315} (2012), no.~2, 279--310.
	
	\bibitem{we:oe1}
	A.~Weil, \emph{G\'eometrie diff\'erentielle des espaces fibr\'es, ({L}etters to
		{C}hevalley and {K}oszul), 1949}, in: {O}euvres scientifique, vol.1,
	Springer, 1979.
	
	\bibitem{wei:ext}
	A.~Weinstein and P.~Xu, \emph{Extensions of symplectic groupoids and
		quantization}, J. Reine Angew. Math. \textbf{417} (1991), 159--189.
	
\end{thebibliography}
\def\cprime{$'$} \def\polhk#1{\setbox0=\hbox{#1}{\ooalign{\hidewidth
			\lower1.5ex\hbox{`}\hidewidth\crcr\unhbox0}}} \def\cprime{$'$}
\def\cprime{$'$} \def\cprime{$'$} \def\cprime{$'$} \def\cprime{$'$}
\def\polhk#1{\setbox0=\hbox{#1}{\ooalign{\hidewidth
			\lower1.5ex\hbox{`}\hidewidth\crcr\unhbox0}}} \def\cprime{$'$}
\def\cprime{$'$} \def\cprime{$'$} \def\cprime{$'$} \def\cprime{$'$}
\providecommand{\bysame}{\leavevmode\hbox to3em{\hrulefill}\thinspace}
\providecommand{\MR}{\relax\ifhmode\unskip\space\fi MR }
\providecommand{\MRhref}[2]{%
	\href{http://www.ams.org/mathscinet-getitem?mr=#1}{#2}
}
\providecommand{\href}[2]{#2}

\end{document}